\documentclass{amsart}
\usepackage{amscd,amssymb}

\newcounter{noindnum}[subsection]
\renewcommand{\thenoindnum}{\roman{noindnum}}
\newcommand{\noindstep}{\refstepcounter{noindnum}{\rm(}\thenoindnum\/{\rm)} }
\newcommand{\stepzero}{\setcounter{noindnum}{0}
}
\renewcommand{\phi}{\varphi}
\renewcommand{\epsilon}{\varepsilon}

\newcommand{\C}{\mathbb C}
\renewcommand{\P}{\mathbb P}
\newcommand{\A}{\mathbb A}
\newcommand{\Z}{\mathbb{Z}}
\newcommand{\R}{\mathbb{R}}

\newcommand{\bL}{\mathbb{L}}

\newcommand{\kk}{\mathbf k}
\newcommand{\divisor}{D}

\newcommand{\Mot}{\mathrm{Mot}}
\newcommand{\cMot}{\overline{\Mot}}

\DeclareMathOperator{\Opf}{Opf}
\DeclareMathOperator{\GL}{GL}
\DeclareMathOperator{\Sh}{Sh}
\DeclareMathOperator{\Arg}{Arg}
\DeclareMathOperator{\Id}{Id}
\DeclareMathOperator{\Pic}{Pic}

\DeclareMathOperator{\res}{res}
\DeclareMathOperator{\Spec}{Spec}
\DeclareMathOperator{\rk}{rk}
\DeclareMathOperator{\cl}{cl}
\DeclareMathOperator{\Hom}{Hom}
\DeclareMathOperator{\Jac}{Jac}
\DeclareMathOperator{\End}{End}
\DeclareMathOperator{\Ob}{\mathcal{O}\mathit{b}}

\DeclareMathOperator\END{\mathcal{E}\mathit{nd}}
\DeclareMathOperator\HOM{\mathcal{H}\mathit{om}}

\DeclareMathOperator{\Ker}{Ker}
\DeclareMathOperator{\Ext}{Ext}

\DeclareMathOperator{\Exp}{Exp}
\DeclareMathOperator{\Log}{Log}
\DeclareMathOperator{\Pow}{Pow}

\newcommand{\gl}{\mathfrak{gl}}
\newcommand{\gm}{\mathbf{G_m}}
\newcommand{\Symm}{\mathfrak{S}}

\newcommand{\cB}{\mathcal B}
\newcommand{\cC}{\mathcal C}

\newcommand{\cE}{\mathcal E}
\newcommand{\cF}{\mathcal F}

\newcommand{\cH}{\mathcal H}

\newcommand{\cL}{\mathcal L}
\newcommand{\cM}{\mathcal M}
\newcommand{\cN}{\mathcal N}
\newcommand{\cO}{\mathcal O}

\newcommand{\cS}{\mathcal S}
\newcommand{\cT}{\mathcal T}
\newcommand{\cU}{\mathcal U}
\newcommand{\cV}{\mathcal V}
\newcommand{\cW}{\mathcal W}
\newcommand{\cY}{\mathcal Y}
\newcommand{\cX}{\mathcal X}
\newcommand{\cZ}{\mathcal Z}

\newcommand{\Conn}{\mathcal{C}onn}

\newcommand{\Bun}{\mathcal{B}un}
\newcommand{\Mod}{\mathcal{M}od}
\newcommand{\Coh}{\mathcal{C}oh}
\newcommand{\Lau}{\mathcal{L}au}

\newcommand{\cPic}{\mathcal{P}ic}

\renewcommand{\Im}{\mathop{\mathrm{Im}}}

\theoremstyle{plain}
\newtheorem{proposition}{Proposition}[subsection]
\newtheorem{lemma}[proposition]{Lemma}
\newtheorem{corollary}[proposition]{Corollary}
\newtheorem{theorem}[proposition]{Theorem}
\newtheorem*{claim}{Claim}

\theoremstyle{remark}
\newtheorem{remark}[proposition]{Remark}

\title[Motivic classes of moduli of Higgs bundles]{Motivic classes of moduli of Higgs bundles and  moduli of bundles with connections}

\author{Roman~Fedorov}
\address{Roman Fedorov, University of Pittsburgh, Pittsburgh, PA, USA}
\email{fedorov@pitt.edu}

\author{Alexander Soibelman}
\address{Alexander Soibelman, University of Southern California, Los Angeles, CA, USA}
\email{asoibelm@usc.edu}

\author{Yan Soibelman}
\address{Yan Soibelman, Kansas State University, Manhattan, KS, USA}
\email{soibel@math.ksu.edu}

\begin{document}

\begin{abstract}
Let $X$ be a smooth projective curve over a field of characteristic zero. We calculate the motivic class of the moduli stack of semistable  Higgs bundles on $X$. We also calculate the motivic class of the moduli stack of vector bundles with connections by showing that it is equal to the class of the stack of semistable Higgs bundles of the same rank and degree zero.

We follow the strategy of Mozgovoy and Schiffmann for counting Higgs bundles over finite fields. The main new ingredient is a motivic version of  a theorem of Harder about Eisenstein series claiming that all vector bundles have approximately the same motivic class of Borel reductions as the degree of Borel reduction tends to $-\infty$.
\end{abstract}

\maketitle
\tableofcontents

\section{Introduction and main results}\label{sect:Intro}
Let $X$ be a smooth projective curve over a field $\kk$. The following additive categories associated with $X$ will be of primary interest to us in this paper: the category of Higgs bundles and the category of vector bundles with connections. The moduli stacks of objects of these categories are Artin stacks locally of finite type. Furthermore, these categories have homological dimensions two. Hence one can apply to them (a version of) the theory of motivic Donaldson--Thomas invariants (DT-invariants for short) developed in~\cite{KontsevichSoibelman08} and ask about explicit formulas for the motivic DT-series (see loc.~cit.\ for the terminology and the details).

Let us denote by $\cM_{r,d}$ the moduli stack of rank $r$ degree $d$ Higgs bundles and by $\cM^{ss}_{r,d}$ its open substack of finite type classifying semistable Higgs bundles. Calculating the DT-series of the category of Higgs bundles is  equivalent to calculating the motivic classes of $\cM^{ss}_{r,d}$. In the case when $\kk$ is finite, these motivic classes are closely related to volumes of the stacks, see~\cite{SchiffmannIndecomposable}.

In this paper we consider the case of the field of characteristic zero and calculate the motivic classes of the stacks $\cM_{r,d}^{ss}$. We also show that the motivic class of the moduli stack of rank $r$ bundles with connections is equal to the motivic class of $\cM_{r,0}^{ss}$. We will give precise formulations of our results in the following subsections of the introduction. Our techniques are motivic generalizations of those of~\cite{SchiffmannIndecomposable} and~\cite{MozgovoySchiffmanOnHiggsBundles}. The main new ingredient is a motivic version of  a theorem of Harder about Eisenstein series. Furthermore, motivic versions of many results of the loc. cit. require more substantial use of algebraic geometry; in particular, we make a systematic use of generic points of schemes. Moreover, several proofs from ~\cite{MozgovoySchiffmanOnHiggsBundles} require substantial revision or replacement in the motivic case.

The reader will notice that besides the results of Schiffmann and Mozgovoy--Schiffmann our paper is largely motivated by the general philosophy of motivic DT-invariants developed in~\cite{KontsevichSoibelman08}, which is a right framework for many questions about motivic invariants of $3$-dimensional Calabi--Yau categories or categories of homological dimension less than $3$.

\subsection{{Motivic classes  of stacks}}\label{sect:IntroMotSt}
All the stacks considered in this paper will be Artin stacks locally of finite type over a field. For an arbitrary field $\kk$ one defines the abelian group $\Mot(\kk)$ as the group generated by isomorphism classes of $\kk$-stacks of finite type modulo the following relations:

(i) $[\cY_1]=[\cY_2]+[\cY_1-\cY_2]$ whenever $\cY_2$ is a closed substack of $\cY_1$;

(ii) $[\cY_2]=[\cY_1\times\A_\kk^r]$ whenever $\cY_2\to\cY_1$ is a vector bundle of rank $r$.

Note that $\Mot(\kk)$ is a commutative ring under the product $[\cX][\cY]=[\cX\times\cY]$. This material is well-known (see e.g.~\cite[Sect.~1]{Ekedahl09}, \cite{Joyce07},~\cite{KontsevichSoibelman08}).

We define the dimensional completion of $\Mot(\kk)$ as follows. Let $F^m\Mot(\kk)$ be the subgroup generated by the classes of stacks of dimension $\le-m$. This is a ring filtration and we define the completed ring $\cMot(\kk)$ as the completion of $\Mot(\kk)$ with respect to this filtration. In the completed ring one can take infinite sums of motivic classes of stacks of finite type provided these sums are convergent. The completion is necessary, for example, to define the residue of a series whose coefficients are motivic classes. The class~$[\cY]$ in $\Mot(\kk)$ or $\cMot(\kk)$ is called the \emph{motivic class\/} of the stack $\cY$.

Later, we will also need a relative version of $\Mot(\kk)$ and $\cMot(\kk)$. We refer the reader to Section~\ref{sect:MotFun} for details and references.

\subsection{Moduli stacks of Higgs bundles and connections}\label{sect:ModStack}
Fix a smooth geometrically connected projective curve $X$ over $\kk$. By a Higgs bundle on~$X$, we mean a pair $(E,\Phi)$, where $E$ is a vector bundle on~$X$, and $\Phi:E\to E\otimes\Omega_X$ is an $\cO_X$-linear morphism from $E$ to $E$ ``twisted'' by the sheaf of differential $1$-forms $\Omega_X$. By definition the rank of a pair $(E,\Phi)$ is the rank of $E$, the degree is the degree of $E$. Denote by $\cM_{r,d}$ the moduli stack of rank $r$ degree $d$ Higgs bundles on $X$. We define Higgs sheaves similarly, by replacing vector bundles with coherent sheaves in the definition.

We also note for further use that every coherent sheaf $F$ on $X$ can be written as $T\oplus E$, where $T$ is a torsion sheaf, $E$ is a torsion free sheaf (that is, a vector bundle). In this decomposition $T$ is unique, while~$E$ is unique up to isomorphism.

The Higgs bundle $(E,\Phi)$ is called \emph{semistable\/} if for any subbundle $F\subset E$ preserved by $\Phi$ we have
\[
    \frac{\deg F}{\rk F}\le\frac{\deg E}{\rk E}.
\]
Semistability is an open condition compatible with field extensions; we denote the open substack of semistable Higgs bundles by $\cM^{ss}_{r,d}\subset\cM_{r,d}$. The stack $\cM^{ss}_{r,d}$ is of finite type. We refer the reader to Section~\ref{Sect:Higgs} for more details. Similarly one can deal with Higgs sheaves. The latter form an abelian category.

Denote by $\Conn_r$ the moduli stack of rank $r$ vector bundles with connections on $X$. That is, the stack classifying pairs $(E,\nabla)$, where $E$ is a rank $r$ vector bundle on $X$, $\nabla:E\to E\otimes\Omega_X$ is a $\kk$-linear morphism of sheaves satisfying the Leibnitz rule: for any open subset $U$ of $X$, any $f\in H^0(U,\cO_X)$ and any $s\in H^0(U,E)$ we have
\[
    \nabla(fs)=f\nabla(s)+s\otimes df.
\]

Assume  that $\kk$ is a field of characteristic zero. Then the stack $\Conn_r$ is a stack of finite type. We will reprove this well-known fact in Section~\ref{sect:ConnIsoslopy}. Recall that every vector bundle admitting a connection has degree zero (Weil's theorem). Note also that in the case of bundles on curves every  connection is automatically flat as $\wedge^2\Omega_X=0$. Our first main result is the following theorem.

\begin{theorem}\label{th:conn=higgs} If $\kk$ is a field of characteristic zero, then we have the equality of motivic classes in $\Mot(\kk)${\rm:}
\[
    [\Conn_r]=[\cM^{ss}_{r,0}].
\]
\end{theorem}

This theorem will be proved in Section~\ref{sect:CompHiggsConn}. The proof is inspired by~\cite{MozgovoySchiffmanOnHiggsBundles}.
To give the reader the flavor of the statement, we sketch a direct proof in the case of $r=2$ in Section~\ref{Sect:Conn=Higgs2}.

\begin{remark}
Note that every bundle with connection $(E,\nabla)$ is semistable in the following sense: if $F$ is a subbundle preserved by $\nabla$, then $\deg F=0=\deg E$. Thus the theorem tells that the motivic classes of the stacks of semistable Higgs bundles and the stack of semistable bundles with connections are equal. In fact, if $\kk$ is the field of complex numbers, then the corresponding categories are equivalent by Simpson's non-abelian Hodge theory. However, we do not see how to derive the equality of motivic classes of stacks from Simpson's result.
\end{remark}

\subsection{Explicit formulas for motivic classes}
Our second main result is the explicit calculation of the motivic class of the stack of semistable Higgs bundles. This problem has some history including the paper by Mozgovoy~\cite{MozgovoyADHM} where the conjecture about the motivic class was made in the case when the rank and the degree are coprime, and the papers by Schiffmann~\cite{SchiffmannIndecomposable} and Mozgovoy--Schiffmann~\cite{MozgovoySchiffmanOnHiggsBundles} devoted to the calculation of the volume of the stack over a finite field. We assume that $\kk$ is a field of characteristic zero.

In order to formulate our result let us recall some standard notions.

\subsubsection{Motivic zeta-functions}\label{sect:zeta}
Here we follow ~\cite{KapranovMotivic}. For a variety $Y$ set
\[
    \zeta_Y(z):=\sum_{n=0}^\infty[Y^{(n)}]z^n\in\Mot(\kk)[[z]],
\]
where $Y^{(n)}=Y^n/\Symm_n$ is the $n$-th symmetric power of $Y$ ($\Symm_n$ denotes the group of permutations).

Assume now that $Y=X$ is our smooth curve. For the rest of the introduction, we assume that $X$ has a divisor of degree one defined over $\kk$. (Note that this condition is satisfied, if $X$ has a $\kk$-rational point.) Let $g$ be the genus of $X$. Set $\bL:=[\A_\kk^1]$.

\begin{proposition}\label{lm:zetaX}
(i)
\[
    \zeta_X(z)=P_X(z)/(1-z)(1-\bL z)
\]
for a polynomial $P_X(z)$ with coefficients in $\Mot(\kk)${\rm;}

(ii) $P_X(0)=1$ and the highest term of $P_X$ is $\bL^gz^{2g}$.

(iii) We have
\[
    \zeta_X(1/\bL z)=\bL^{1-g}z^{2-2g}\zeta_X(z).
\]

(iv) If $i\ne0,-1$, then $\zeta_X(\bL^i)\ne0$ is invertible in $\cMot(\kk)$.
\end{proposition}
Before giving the proof, we note that part $(i)$ is used to view $\zeta_X(z)$ as a function on $\cMot(\kk)$ defined for all $z$ such that $1-z$ and $1-\bL z$ are invertible in $\cMot(\kk)$. In particular, $\zeta_X(\bL^i)$ is defined for $i\ne0,-1$.
\begin{proof}
Statements (i) and (iii) is~\cite[Thm.~1.1.9]{KapranovMotivic}. It is obvious that $P_X(0)=1$, the statement about the highest term of $P_X$ now follows from (iii). Statement (iv) in the case $i\le-2$ follows from the fact that $\zeta_X(\bL^i)\in1+F^1\cMot(\kk)$, where $F^m\cMot(\kk)$ is the dimensional filtration. In the case $i\ge1$ the statement follows from (iii).
\end{proof}

It is convenient to introduce the ``normalized'' zeta-function $\tilde\zeta_X(z):=z^{1-g}\zeta_X(z)$ and the ``regularized'' zeta-function by setting
\[
    \zeta_X^*(\bL^{-u}z^v)=
    \begin{cases}
        \zeta_X(\bL^{-u}z^v)\text{ if }v>0\text{ or }u>1,\\
        \res_{z=\bL^{-1}}\zeta_X(z)\frac{dz}z:=\displaystyle\frac{P_X(\bL^{-1})}{1-\bL^{-1}}\text{ if }(u,v)=(1,0),\\
        \res_{z=1}\zeta_X(z)\frac{dz}z:=\displaystyle\frac{P_X(1)}{1-\bL}\text{ if }(u,v)=(0,0).
    \end{cases}
\]
(Cf.~the definition of the residue in Section~\ref{sect:res}.)

Let $(1+z\cMot(\kk)[[z]])^\times$ denote the multiplicative group of power series with constant term 1. One can uniquely extend the assignment $Y\mapsto\zeta_Y$ to a continuous homomorphism of topological groups
\begin{equation}\label{eq:zeta}
    \zeta:\cMot(\kk)\to(1+z\cMot(\kk)[[z]])^\times
\end{equation}
such that for any $A\in\cMot(\kk)$ and any $n\in\Z$ we have $\zeta_{\bL^nA}(z)=\zeta_A(\bL^nz)$. More precisely, any class $A\in\cMot(\kk)$ can be written as the limit of a sequence $([Y_i]-[Z_i])/\bL^{n_i}$, where $Y_i$ and $Z_i$ are varieties. We define
\[
    \zeta_A(z)=\lim_{i\to\infty}\frac{\zeta_{Y_i}(\bL^{-n_i}z)}{\zeta_{Z_i}(\bL^{-n_i}z)}.
\]
For more details, see Section~\ref{sect:MotVar}. Note that the operation $A\mapsto\zeta_A$ gives a pre-lambda structure on the ring $\cMot(\kk)$ (see~\cite{GuzeinZadeEtAlOnLambdaRingStacks}).

Consider the ring of formal power series in two variables $\cMot(\kk)[[z,w]]$ (this is, actually, an example of a quantum torus, cf.~Section~\ref{sect:IntroRemarks}). Let $\cMot(\kk)[[z,w]]^+$ denote the ideal of power series with vanishing constant term, let $(1+\cMot(\kk)[[z,w]]^+)^\times$ be the multiplicative group of series with constant term equal~1.
We define the \emph{plethystic exponent\/} $\Exp:\cMot(\kk)[[z,w]]^+\to(1+\cMot(\kk)[[z,w]]^+)^\times$ by
\[
    \Exp\left(\sum_{r,d} A_{r,d}w^rz^d\right)=\prod_{r,d}\Exp(A_{r,d}w^rz^d)=
    \prod_{r,d}\zeta_{A_{r,d}}(w^rz^d).
\]
One shows easily that this is an isomorphism of abelian groups. Denote the inverse isomorphism by $\Log$ (the \emph{plethystic logarithm\/}).

\subsubsection{Explicit formulas for motivic classes of the stacks of semistable Higgs bundles}\label{sect:explicit} The following is the motivic version of~\cite[Sect.~1.4]{SchiffmannIndecomposable}.

Let $\lambda=(\lambda_1\ge\lambda_2\ge\ldots\ge\lambda_l>0)$ be a partition. We can also write it as $\lambda=1^{r_1}2^{r_2}\ldots t^{r_t}$, where $r_i$ is the number of occurrences of $i$ among $\lambda_j, 1\le j\le l$. The Young diagram of $\lambda$ is the set of the points $(i,j)\in\Z^2$ such that $1\le i\le\lambda_j$. For a box $s\in\lambda$ its arm $a(s)$ (resp. leg $l(s)$) is the number of boxes lying strictly to the right of (resp.~strictly above) $s$.

For a partition $\lambda$, set
\[
    J^{mot}_\lambda(z)=\prod_{s\in\lambda}\zeta_X^*(\bL^{-1-l(s)}z^{a(s)})\in\cMot(\kk)[[z]],
\]
where the product is over all boxes of the Young diagram corresponding to the partition. In particular, for the empty Young diagram $\lambda$ we get $J^{mot}_\lambda(z)=1$.

Set\footnote{Following~\cite{SchiffmannIndecomposable} we use inverse numeration of variables.}
\[
    L^{mot}(z_n,\ldots,z_1)=
    \frac1{\prod_{i<j}\tilde\zeta_X\left(\frac{z_i}{z_j}\right)}
    \sum_{\sigma\in\Symm_n}\sigma\left\{
    \prod_{i<j}\tilde\zeta_X\left(\frac{z_i}{z_j}\right)
    \frac1{\prod_{i<n}\left(1-\bL\frac{z_{i+1}}{z_i}\right)}
    \cdot\frac1{1-z_1}
    \right\}.
\]
For a partition $\lambda=1^{r_1}2^{r_2}\ldots t^{r_t}$ such that $\sum_i r_i=n$, set $r_{<i}=\sum_{k<i}r_k$ and denote by $\res_\lambda$ the iterated residue along
\[
\begin{matrix}
    \frac{z_n}{z_{n-1}}=\bL^{-1},&\frac{z_{n-1}}{z_{n-2}}=\bL^{-1},&\ldots,&
    \frac{z_{2+r_{<t}}}{z_{1+r_{<t}}}=\bL^{-1},\\
    \vdots &  \vdots && \vdots  \\
    \frac{z_{r_1}}{z_{r_1-1}}=\bL^{-1},&
    \frac{z_{r_1-1}}{z_{r_1-2}}=\bL^{-1},&\ldots,&
    \frac{z_2}{z_1}=\bL^{-1}.
\end{matrix}
\]
Set
\[
    \tilde H_\lambda^{mot}(z_{1+r_{<t}},\ldots,z_{1+r_{<i}},\ldots,z_1):=
    \res_\lambda\left[
     L^{mot}(z_n,\ldots,z_1)\prod_{\substack{j=1\\j\notin\{r_{<i}\}}}^n\frac{dz_j}{z_j}
    \right]
\]
and
\begin{equation}\label{eq:subst}
    H_\lambda^{mot}(z):=\tilde H_\lambda^{mot}(z^t\bL^{-r_{<t}},\ldots,z^i\bL^{-r_{<i}},\ldots,z).
\end{equation}

\begin{remark}
Neither the notion of residue, nor the substitution~\eqref{eq:subst} is obvious for rational functions whose coefficients belong to $\cMot(\kk)$ because it is not known whether this ring is integral. Let us give the precise definition. For a polynomial $P(z_n,\ldots,z_1)\in\cMot(\kk)[z_n,\ldots,z_1]$, let $P_\lambda(z)$ denote the one-variable polynomial, obtained from $P$ by substituting (for each $i$) $\bL^{1-i}z^{j_i}$ instead of $z_i$, where $j_i$ is the unique number such that $r_{<j_i}<i\le r_{<j_i+1}$. Consider the product
\[
    \prod_{\substack{j=1\\j\notin\{r_{<i}\}}}^n
    \left(1-\frac{\bL z_{j+1}}{z_j}\right)L^{mot}(z_n,\ldots,z_1).
\]
Inspecting the formula for $L^{mot}$ and using Lemma~\ref{lm:zetaX}, we can show that the product can be written in the form $P(z_n,\ldots,z_1)/Q(z_n,\ldots,z_1)$, where $Q_\lambda(0)$ is invertible in $\cMot(\kk)$. Then
\[
    H_\lambda^{mot}(z):=\frac{P_\lambda(z)}{Q_\lambda(z)}.
\]
We expand this rational function in powers of $z$, so that $H_\lambda^{mot}(z)\in\cMot(\kk)[[z]]$. We refer the reader to Section~\ref{sect:res} and especially to Remark~\ref{rm:RationalRes} for the compatibility with the definition of the residue of a power series.
\end{remark}

Let us introduce the elements $B_{r,d}\in\cMot(\kk)$ via the formula
\begin{equation*}
\sum_{\substack{r,d\in\Z_{\ge0}\\(r,d)\ne(0,0)}}B_{r,d}w^rz^d=
\bL\Log\left(\sum_\lambda\bL^{(g-1)\langle\lambda,\lambda\rangle}J_\lambda^{mot}(z) H_\lambda^{mot}(z)w^{|\lambda|}
\right).
\end{equation*}
Here the sum is over all partitions, $\langle\lambda,\lambda\rangle:=\sum_i(\lambda_i')^2$, where $\lambda'=(\lambda_1'\ge\ldots\ge\lambda_i'\ge\ldots)$ is the conjugate partition, $|\lambda|=\sum_i\lambda_i$. Next, let $\tau{\ge}0$ be a rational number. {Define the elements $H_{r,d}\in\cMot(\kk)$ by}
\[
\sum_{d/r=\tau}{\bL^{(1-g)r^2}}H_{r,d}w^rz^d=
    \Exp\left(
        \sum_{d/r=\tau}B_{r,d}w^rz^d
    \right).
\]
Now we can formulate our second main result.
\begin{theorem}\label{th:ExplAnsw}
Let $\kk$ be a field of characteristic zero.

{(i) The elements $H_{r,d}$ are  periodic in $d$ with period $r$ {in the following sense}: for $d>r(r-1)(g-1)$ we have $H_{r,d}=H_{r,d+r}$.}

(ii) For any $r>0$ and any $d$ we have
\[
    [\cM_{r,d}^{ss}]=H_{r,d+er},
\]
whenever $e$ is large enough {\rm(}it suffices to take $e>(r-1)(g-1)-d/r$.{\rm)}
\end{theorem}

The proof of Theorem~\ref{th:ExplAnsw} will be given in Section~\ref{sect:ExplAnsw}. Note that $H_{r,d}$ can be computed explicitly in terms of the operations in the pre-lambda ring $\cMot(\kk)$. Thus, the above theorem gives an explicit answer for the motivic class of the stack $\cM_{r,d}^{ss}$.

A similar statement is not known, and probably not literally true, in the case when $\kk$ is a field of finite characteristic. However, for finite fields one can calculate the volume of the groupoid $\cM_{r,d}^{ss}(\kk)$. This volume has been calculated in~\cite{SchiffmannIndecomposable,MozgovoySchiffmanOnHiggsBundles}. Our answer is very similar to theirs. The proof in our case is also very similar to loc.~cit.~except for a few subtle points. One is a motivic version of a theorem of Harder, which we will discuss in Sections~\ref{sect:IntroHarder} and~\ref{sect:Borel}.

Combining part~(ii) of the above theorem with Theorem~\ref{th:conn=higgs}, we arrive at the following result.

\begin{corollary}
We have
\[
    [\Conn_r]=[H_{r,er}]
\]
for any $e>(r-1)(g-1)$.
\end{corollary}

\subsection{Vector bundles with nilpotent endomorphisms}\label{sect:NilpIntro}
The proof of Theorem~\ref{th:ExplAnsw} is based on the following statement of independent interest. Let the stack $\cE^{\ge0,nilp}$ classify pairs $(E,\Phi)$, where $E$ is a vector bundle on $X$ such that there are no non-zero morphisms $E\to F$, with $\deg F<0$, $\Phi$ is a nilpotent endomorphism of $E$. Then $\cE^{\ge0,nilp}$ decomposes according to the rank and degree of the bundles: $\cE^{\ge0,nilp}=\sqcup_{r,d}\cE_{r,d}^{\ge0,nilp}$. It follows easily from Lemma~\ref{lm:Bun+} below that $\cE_{r,d}^{\ge0,nilp}$ is an Artin stack of finite type.

\begin{theorem}\label{th:NilpEnd}
Let $\kk$ be a field of characteristic zero. We have the following identity in $\cMot(\kk)[[z,w]]$.
\[
    \sum_{r,d\ge0}[\cE_{r,d}^{\ge0,nilp}]w^rz^d=\sum_{\lambda}
    \bL^{(g-1)\langle\lambda,\lambda\rangle}J_\lambda^{mot}(z) H_\lambda^{mot}(z)w^{|\lambda|}.
\]
\end{theorem}
This theorem will be proved in Section~\ref{sect:NilpEnd}.

\subsection{Harder's theorem on motivic classes of Borel reductions}\label{sect:IntroHarder}
As we mentioned in Section~\ref{sect:IntroMotSt}, for any stack $\cX$ one can define the relative group of motivic functions $\Mot(\cX)$ and its completion $\cMot(\cX)$ (so that $\Mot(\kk)=\Mot(\Spec\kk)$ and $\cMot(\kk)=\cMot(\Spec\kk)$). If $\cY$ is a stack of finite type over $\cX$, we have a motivic function $[\cY\to\cX]$ (and $\Mot(\cX)$ is generated by these functions if $\cX$ is of finite type). In particular, we have the ``constant function'' $\mathbf1_\cX:=[\cX\to\cX]$. We review this standard material in Section~\ref{sect:MotFun}. Note that in a slightly different settings these groups were defined in~\cite{KontsevichSoibelman08}.
The group $\cMot(\cX)$ is a topological group.

\subsubsection{The stack of Borel reductions}\label{sect:BorelRed}
In this section $\kk$ is a field of arbitrary characteristic. Denote by $\Bun_{r,d}$ the moduli stack of rank $r$ degree $d$ vector bundles on $X$. By a \emph{Borel reduction\/} of a rank $r$ vector bundle $E$ we understand a full flag of subbundles
\[
    0=E_0\subset E_1\subset\ldots\subset E_r=E.
\]
In particular, $E_i$ is a vector bundle of rank $i$ and $E/E_i$ is a vector bundle of rank $r-i$. The \emph{degree} of the Borel reduction is given by
\[
(d_1,\ldots,d_r),\text{ where }d_i:=\deg E_i-\deg E_{i-1}.
\]
(In particular, $\deg E_1=d_1$.)

Let $\Bun_{r,d_1,\ldots,d_r}$ stand for the stack of rank $r$ vector bundles with a Borel reduction of degree $(d_1,\ldots,d_r)$.
We view  $\Bun_{r,d_1,\ldots,d_r}$ as a stack over $\Bun_{r,d_1+\ldots+d_r}$ via the projection $(E_1\subset\ldots\subset E_r)\mapsto E_r$. In Section~\ref{Sect:ProofBorel} we explain that this projection is of finite type and prove the following theorem:

\begin{theorem}\label{th:IntroHarder}
For any $r>0$ and $d\in\Z$ we have in $\cMot(\Bun_{r,d})$
\[
\lim_{d_1\to-\infty}\ldots\lim_{d_{r-1}\to-\infty}\frac{[\Bun_{r,d_1,\ldots,d_{r-1},d-d_1-\ldots-d_{r-1}}\to\Bun_{r,d}]}
{\bL^{-(2r-2)d_1-(2r-4)d_2-\ldots-2d_{r-1}}}=
\frac{
    \bL^{(r-1)\left(d+(1-g)\frac{r+2}2\right)}[\Jac]^{r-1}}
    {(\bL-1)^{r-1}\prod_{i=2}^r\!\zeta_X(\bL^{-i})}\;\mathbf1_{\Bun_{r,d}}.
\]
\end{theorem}

Here $\Jac=\Jac_X$ is the Jacobian variety of $X$. We note that $\zeta_X(\bL^{-i})$ converges for $i\ge2$.

\begin{remark}
It is very important that the right hand side is a product of an element of $\cMot(\Spec\kk)$ with $\mathbf1_{\Bun_{r,d}}$. This may be loosely reformulated as a statement that all vector bundles have approximately equal motivic classes of Borel reductions. Moreover, this is true uniformly over any substack of finite type.
\end{remark}

The generating functions for the motivic classes $[\Bun_{r,d_1,\ldots,d_r}\to\Bun_{r,d_1+\ldots+d_r}]$ are known as (motivic) Eisenstein series. The above theorem can be interpreted as a statement about the \emph{residue\/} of this Eisenstein series; see Theorem~\ref{th:ResHarder2} and Proposition~\ref{pr:HallFormulas}(vi)  below.

\subsection{Relation with results of Kontsevich and Soibelman}\label{sect:IntroRemarks}
In~\cite{KontsevichSoibelman08} and~\cite{KontsevichSoibelman10} the authors developed a general theory of motivic Donaldson--Thomas invariants (DT-invariants for short) of three-dimensional Calabi--Yau categories (3CY categories for short). The categories considered in~\cite{KontsevichSoibelman08} are ind-constructible. Roughly speaking, this means that the objects of such categories are parameterized by unions of Artin stacks of finite type (see loc.~cit.~for the precise definition).

Most of the categories that appear ``in nature'' are ind-constructible. Important examples are given by representations of algebras (or dg-algebras), (derived) categories of coherent sheaves, categories of Higgs sheaves etc. In the case of coherent sheaves or Higgs sheaves on projective curves, the homological dimension of either of these categories is less than 3. However, one can upgrade them to 3CY categories. For that reason many questions about cohomological and motivic invariants of these categories can be reduced to the general theory developed in~\cite{KontsevichSoibelman08,KontsevichSoibelman10}. This remark could explain an appearance of motivic DT-invariants in some questions about the Hodge theory of character varieties (see~\cite{HLRV}).

A recent spectacular example is the main conjecture from~\cite{HLRV}. The categories of Higgs sheaves and of connections on a curve studied in this paper have cohomological dimension 2, and moreover are 2-dimensional Calabi--Yau categories (2CY categories for short).

\subsubsection{Hall algebras and quantum tori}\label{sect:QuTorus}
Let $\cC$ be an ind-constructible abelian (or more generally $A_\infty$-triangulated) category endowed with  a homomorphism of abelian groups $\cl: K_0(\cC)\to\Gamma\simeq\Z^n$ (``Chern character''). We also assume that $\Gamma$ is endowed with an integer skew-symmetric form $\langle\bullet,\bullet\rangle$  and $\cl$ intertwines this form and the skew-symmetrization of the Euler form on $K_0(\cC)$.

One associates to this data two associative algebras. The first algebra is the motivic Hall algebra $H(\cC)$. As a $\cMot(\kk)$-module, it is equal to a group of stack motivic functions on the stack of objects $\Ob(\cC)$. The other algebra is the quantum torus $R_{\cC}=R_{\Gamma,\cC}:=\bigoplus_{\gamma\in\Gamma}\cMot(\kk)e_\gamma$ associated with $(\Gamma,\langle\bullet,\bullet\rangle)$ (see~\cite{KontsevichSoibelman08} for the definitions of the multiplications on both algebras). Note that $R_{\Gamma,\cC}$ is much more explicit but carries less information.

Suppose that the ind-constructible category $\cC$ carries a constructible stability structure with the central charge $Z: \Gamma\to\C$ (see~\cite{KontsevichSoibelman08} for the details). Identify $\C$ with $\R^2$. Then for  any strict sector $V\subset\R^2$ with the vertex in the origin, one can define a full subcategory $\cC(V)\subset\cC$ generated by the semistables with the central charge in~$V$. Furthermore, the corresponding motivic Hall algebra $H(\cC(V))$ and the quantum torus $R_{\cC(V)}$ admit natural completions  $\widehat{H}(\cC(V))$ and $\widehat{R}_{\cC(V)}$. The former contains an element
\[
	A_{\cC(V)}^{Hall}:=\sum_{\substack{\gamma\in\Gamma\\Z(\gamma)\in V}}\mathbf1_{\Ob_\gamma(\cC)}.
\]
Here $\mathbf1_{\Ob_\gamma(\cC)}$ is the identity motivic function on $\Ob_\gamma(\cC)$, where $\Ob_\gamma(\cC)\subset\Ob(\cC)$ is the substack parameterizing objects of class $\gamma$. In the case when ${\cC}$ is a 3CY category there is a homomorphism of algebras 	 $\Phi:=\Phi_V:\widehat{H}({\cC}(V))\to \widehat{R}_{\cC(V)}$. The element $A_{\cC(V)}^{mot}:=\Phi(A_{\cC(V)}^{Hall})$ is called the \emph{motivic DT-series\/} of $\cC(V)$. The homomorphism $\Phi$ is defined in terms of the motivic Milnor fiber of the potential of ${\cC}$, hence it exists literally for 3CY categories only.

In the case when $\cC$ has homological dimension less or equal than two, one can ``upgrade'' it to a 3CY category by introducing a kind of ``Lagrangian multipliers''. Then the homomorphism $\Phi$ gives rise to a linear map $\widehat{H}({\cC}(V))\to \widehat{R}_{\cC(V)}$ that partially respects the products.

In particular, if $\cC$ is a 2CY category then for each strict sector $V$ there  is a linear map $\Phi:=\Phi_V: \widehat{H}({\cC}(V))\to\widehat{R}_{\cC(V)}$ that satisfies the property
$\Phi([F_1][F_2])=\Phi([F_1])\Phi([F_2])$ as long as $\Arg(Z(F_1))>\Arg(Z(F_2))$ (see e.g.~\cite{RenSoibelman} for details).

Applying $\Phi$ to the element $A_{\cC(V)}^{Hall}$, we arrive at the element
\[
	A_{\cC(V)}^{mot}:=\sum_{\substack{\gamma\in\Gamma\\Z(\gamma)\in V}}w_\gamma[\Ob_\gamma(\cC)] e_\gamma,
\]	
where the ``weight'' $w_\gamma$ is derived from the general theory of  \cite{KontsevichSoibelman08} (there is an alternative approach in~\cite{KontsevichSoibelman10}). In the Higgs sheaves case, the weight is given by $w_\gamma=w_{(r,d)}=\bL^{(1-g)r^2}$.

The above series converge both in the motivic Hall algebra and in the quantum torus, since  a choice of the strict sector $V$ forces us to make a summation over  elements $\gamma$ that belong to a strict convex cone in $\Gamma\otimes\R$.

In the current paper we work with two different categories $\cC$: the category of coherent sheaves on~$X$ and the category of coherent sheaves with Higgs fields. The latter category is 2CY, which forces the quantum torus to be commutative (because the Euler form is symmetric). In this case $\Gamma=\Z^2$ is generated by rank and degree gradation. Thus $R_{\cC}=\cMot(\kk)[z,z^{-1},w,w^{-1}]$, where $z=e_{(0,1)}$, $w=e_{(1,0)}$. So the motivic DT-series $A_{{\cC}(V)}^{mot}$ becomes a generating function  in commuting variables.

The stability structure comes from the central charge $Z(F)=-\deg F+\sqrt{-1}\rk F$. The strict sector~$V$ is the second quadrant $\{x\le 0, y\ge 0\}$ in the plane $\R^2_{(x,y)}$. Therefore, our generating functions are series in two variables. It seems plausible that the whole theory of this paper can be developed for an arbitrary strict sector.

Following~\cite{SchiffmannIndecomposable}, we also use a slightly different framework, when a stability structure is imposed on the coherent sheaf itself rather than on the pair consisting of a sheaf with a Higgs field. Although in this case we do not have a stability structure on the category of Higgs sheaves, the natural forgetful functor from Higgs sheaves to coherent sheaves allows us to utilize the methods of ~\cite{KontsevichSoibelman08}.

\subsection{Further direction of work}
There are several questions that arise naturally in relation to our work.

(1) Generalization to the moduli stacks of $G$-connections (and Higgs $G$-bundles), where $G$ is an arbitrary reductive group. This would require a substantial change in the techniques, since the underlying categories are not additive.

(2) Generalization to the moduli stacks of connections and Higgs bundles with singularities. In the case of regular singularities, one can fix the types of parabolic structures at singular points and look for the motivic class of the moduli stack of parabolic connections and Higgs bundles. The paper~\cite{ChuangDiaconescuDonagiPantev}, although conjectural, contains an alternative approach to the problem via upgrading the computation of the motivic class of Higgs bundles to the problem about refined Pandharipande--Thomas invariants on the non-compact Calabi--Yau 3-fold associated with the spectral curve. The main target of this paper is the HLRV conjecture from~\cite{HLRV} and its generalizations. A different approach to parabolic Higgs bundles on the projective line was suggested (in the case of finite fields) in~\cite{LetellierParabolicHiggs}.

It looks plausible that the techniques of motivic Hall algebras employed in this paper can be used in the parabolic case as well (see e.g.~\cite{LinSpherical} for the case of Hall algebras over finite field).

The case of irregular singularities is less developed. Although one understands somehow the structure of the moduli stacks of Higgs bundles and connections (see e.g.~\cite{KontsevichSoibelman13},~\cite{SzaboIrregularHiggs}) the actual computations are not easy (see~\cite{HauselMerebWong} as well as ~\cite{Diaconescu}, ~\cite{DiaconescuDonagiPantev}).

(3) The relation to the HLRV conjecture in the parabolic and (especially) the irregular case is another natural question. So far, the formulas obtained by Mozgovoy and Schiffmann give an a priori different answer than was expected in~\cite{HLRV}. Since our approach is a motivic version of the one of Mozgovoy and Schiffmann, the same discrepancy is expected for the generalizations as well.

(4) Since the category of Higgs bundles on a curve is an example of a 2-dimensional Calabi--Yau category, one can try to speculate which of our results hold for more general 2CY categories. An interesting class of such categories was proposed in~\cite{RenSoibelman} in relation to semicanonical bases. It seems plausible that Cohomological Hall algebras provide the right framework for many questions arising from the HLRV conjecture. We should note  that in~\cite{KontsevichSoibelman10} the authors used motivic groups different from those considered in this paper. However, the motivic Donaldson--Thomas series obtained for two different versions of motivic groups in~\cite{KontsevichSoibelman08} and~\cite{KontsevichSoibelman10} agree in the end. On the other hand, motivic DT-invariants appear naturally in relation to the HLRV conjecture. This remark explains our optimism concerning the relation of Cohomological Hall algebras and motivic classes of Higgs bundles (and connections) in all above-mentioned cases.

\subsection{Plan of the paper}
In Section~\ref{sect:MotFun} we discuss motivic classes of stacks. This material is standard and is presented here for the reader's convenience.

In Section~\ref{sect:Conn=Higgs} we introduce various stacks and provide relations between their motivic classes. In particular, we prove Theorem~\ref{th:conn=higgs} and give a relation between the moduli stacks of Higgs bundles, moduli stacks of vector bundles with endomorphisms, and moduli stacks of vector bundles with nilpotent endomorphisms.

In Section~\ref{sect:Borel} we prove Theorem~\ref{th:IntroHarder}. This statement was not known in the motivic setup, and, in a sense, was the main stumbling block for re-writing the results of~\cite{SchiffmannIndecomposable} and~\cite{MozgovoySchiffmanOnHiggsBundles} in the motivic situation.

In Section~\ref{sect:Hall} we discuss the motivic Hall algebra of the category of coherent sheaves on a curve. We do some explicit calculations in this Hall algebra. These calculations are used in Section~\ref{sect:Proofs} to prove Theorems~\ref{th:NilpEnd} and~\ref{th:ExplAnsw}.

\subsection{Acknowledgments} We thank D.~Arinkin, L.~Borisov, A.~Braverman, P.~Deligne, E. Diaconescu, S.~Gusein-Zade, T.~Hausel, J.~Heinloth, O.~Schiffmann, and M.~Smirnov, for useful discussions and correspondence. A part of this work was done while R.F.~was an Assistant Professor at Kansas State University, another part was done while R.F.~was visiting Max Planck Institute for Mathematics in Bonn. The work of R.F.~was partially supported by NSF grant DMS--1406532. Y.S.~thanks IHES for excellent research conditions and hospitality. His work was partially supported by NSF grants.

\section{Stack motivic functions and constructible subsets of stacks}\label{sect:MotFun}
In this section $\kk$ is a field of any characteristic. Recall that in this paper we only work with Artin stacks locally of finite type over a field whose groups of stabilizers are affine. According to~\cite[Prop.~3.5.6, Prop.~3.5.9]{KreschStacks} every such stack has a stratification by global quotients of the form $X/\GL_n$, where $X$ is a scheme. We will often use this result below. In this section, we define the group of motivic functions on such a stack $\cX$ (Notation: $\Mot(\cX)$). Motivic functions on Artin or, more generally, constructible stacks were studied by different authors: see, for example,~\cite{Joyce07}, \cite{KontsevichSoibelman08} (or~\cite[Sect.~1]{Ekedahl09} in the case when $\cX$ is the spectrum of a field), so no results of this section are really new. We have included this section for convenience of the reader and in order to fix the notation.

Recall from~\cite[Ch.~5]{LaumonMoretBailly} the notion of points of a $\kk$-stack $\cS$. Let $K\supset\kk$ be a field extension. By a $K$-point of $\cS$ we mean an object of the groupoid $\cS(K)$. A $K'$-point $\xi'$ of $\cS$ is \emph{equivalent\/} to a $K''$-point $\xi''$ of $\cS$ if there is an extension $K\supset\kk$ and $\kk$-embeddings $K'\hookrightarrow K$, $K''\hookrightarrow K$ such that $\xi'_K$ is isomorphic to $\xi''_K$ (as an object of $\cS(K)$). The set of equivalence classes of points of $\cS$ is denoted by $|\cS|$; this set carries Zariski topology.\footnote{If $Y\to\cS$ is a surjective 1-morphism, where $Y$ is a scheme, then every point of $\cS$ is equivalent to a $K$-point, where $K$ is a residue field of a point of $Y$. This explains why $|\cS|$ is a set rather than a class.} If $Y$ is a scheme, then $|Y|$ is identified with the underlying set of $Y$. A 1-morphism $F:\cS\to\cS'$ induces a continuous map $|F|:|\cS|\to|\cS'|$.

\subsection{Stack motivic functions}\label{sect:MotFun1} Let $\cX$ be an Artin stack of finite type. The abelian group $\Mot(\cX)$ is the group generated by the isomorphism classes of finite type 1-morphisms $\cY\to\cX$ modulo relations

(i) $[\cY_1\to\cX]=[\cY_2\to\cX]+[(\cY_1-\cY_2)\to\cX]$ whenever $\cY_2$ is a closed substack of $\cY_1$.

(ii) If $\pi:\cY_1\to\cX$ is a 1-morphism of finite type and $\psi:\cY_2\to\cY_1$ is a vector bundle of rank $r$, then
\[
    [\cY_2\xrightarrow{\pi\circ\psi}\cX]=[\cY_1\times\A_\kk^r\xrightarrow{\pi\circ p_1}\cX].
\]
We call $\Mot(\cX)$ the \emph{group of stack motivic functions on $\cX$} (usually we will drop the word ``stack''). We write $[\cY]$ instead of $[\cY\to\Spec\kk]$. We write $\Mot(\kk)$ instead of $\Mot(\Spec\kk)$. We denote by $\bL\in\Mot(\kk)$ the element $[\A^1_{\kk}]$ (called the Lefschetz motive). Note that motivic functions `don't feel nilpotents' in the sense that $[\cX\to\cY]=[\cX_{red}\to\cY]$.

For a 1-morphism $f:\cX\to\cY$ of stacks of finite type, we have the pullback homomorphism $f^*:\Mot(\cY)\to\Mot(\cX)$, such that $f^*([\cS\to\cY])=[\cS\times_\cY\cX\to\cX]$ (note that $f$ is not necessarily of finite type because $\cX$ and $\cY$ may be stacks over different fields). For $A\in\Mot(\cY)$ we sometimes write $A|_\cX$ instead of $f^*A$.

Next, if $f:\cX\to\cY$ is of finite type, then we have the pushforward homomorphism $f_!:\Mot(\cX)\to\Mot(\cY)$, such that $f_!([\pi:\cS\to\cX])=[f\circ\pi:\cS\to\cY]$. For $\kk$-stacks of finite type $\cX$ and $\cY$ we also have an external product $\boxtimes:\Mot(\cX)\otimes_\Z \Mot(\cY)\to\Mot(\cX\times\cY)$.

We have the usual properties of pullbacks and pushforwards: $(fg)^*=g^*f^*$, $(fg)_!=f_!g_!$, the base change; easy proofs are left to the reader.

If $\cX$ is a stack locally of finite type, denote by $\Opf(\cX)$ the set of finite type open substacks $\cU\subset\cX$ ordered by inclusion. Set
\[
    \Mot(\cX):=\lim_{\longleftarrow}\Mot(\cU)
\]
where the limit is taken over the partially ordered set $\Opf(\cX)$. In other words, a motivic function on $\cX$ amounts to a motivic function $A_\cU$ on each $\cU\in\Opf(\cX)$ such that $(A_\cU)|_\cW=A_\cW$, whenever $\cW\subset\cU$.

If $\cX\to\cY$ is a 1-morphism of finite type (but $\cY$ is not necessarily of finite type), we write $[\cX\to\cY]\in\Mot(\cY)$ for the inverse system given by $\cU\mapsto\cX\times_\cY\cU$. Put $\mathbf1_\cX:=[\cX\to\cX]\in\Mot(\cX)$.

It is straightforward to check that the pullback extends to any 1-morphism of stacks, while the pushforward extends to any 1-morphism of finite type.

Set
\[
\Mot^{fin}(\cX):=\cup_{\cU\in\Opf(\cX)}(j_\cU)_!\Mot(\cU)\subset\Mot(\cX),
\]
where $j_\cU:\cU\to\cX$ is the open immersion. This is the group of ``motivic functions with finite support''. Note that the pushforward and the pullback preserve $\Mot^{fin}$, provided the 1-morphism is of finite type. In fact, the pushforward $f_!:\Mot^{fin}(\cX)\to\Mot^{fin}(\cY)$ may be defined for all morphisms locally of finite type. The importance of $\Mot^{fin}(\cX)$ will be clear in Section~\ref{sect:BilinearForm}.

Next, the direct product makes $\Mot(\kk)$ into a commutative associative unital ring. For any $\kk$-stack $\cX$ the group $\Mot(\cX)$ is a module over $\Mot(\kk)$. Moreover, pullbacks and pushforwards are homomorphisms of modules. (In fact, the fiber product makes any $\Mot(\cX)$ into a ring, but we will not use this structure when $\cX$ is not the spectrum of a field.)

Let $\cX$ be an Artin stack of finite type. Let $F^m\Mot(\cX)\subset\Mot(\cX)$ be the subgroup generated by $\cX$-stacks of relative dimension $\le-m$. We denote the completion with respect to this filtration by $\cMot(\cX)$ and call it \emph{the completed group of stack motivic functions}. Note that the operations $f^*$ and $f_!$ are continuous with respect to the topology given by the filtrations, so both the pullback and the pushforward extend to completed groups by continuity. We write $\cMot(\kk)$ instead of $\cMot(\Spec\kk)$.

If $\cX$ is an Artin stack locally of finite type, we define
\[
    \cMot(\cX):=\lim_{\longleftarrow}\cMot(\cU),
\]
where the inverse limit is taken over the partially ordered set $\Opf(\cX)$. Note that $\cMot(\cX)$ has the inverse limit topology: a subset of $\cMot(\cX)$ is open if and only if it is a preimage of an open subset in $\cMot(\cU)$ for some $\cU\in\Opf(\cX)$. It follows that a sequence $A_n\in\cMot(\cX)$ converges to $A$ if and only if for all $\cU\in\Opf(\cX)$ we have $\lim_{n\to\infty}(A_n|_\cU)=A|_\cU$.

Everything we said about the groups of stack motivic functions extends to completed groups by continuity. In particular, we can extend the pullbacks to all 1-morphisms of stacks, and the pushforwards to all 1-morphisms of finite type. The product on $\Mot(\kk)$ extends by continuity to $\cMot(\kk)$. The $\Mot(\kk)$-module structure on $\Mot(\cX)$ extends to a $\cMot(\kk)$-module structure on $\cMot(\cX)$. We also define
\[
\cMot^{fin}(\cX):=\cup_{\cU\in\Opf(\cX)}(j_\cU)_!\cMot(\cU)\subset\cMot(\cX)
\]
and note that the pushforward $f_!:\cMot^{fin}(\cX)\to\cMot^{fin}(\cY)$ may be defined for all morphisms locally of finite type.

We do not know whether the natural morphism $i:\Mot(\cX)\to\cMot(\cX)$ is injective. However, we abuse notation by writing $A$ instead of $i(A)$, that is, by viewing an element of $\Mot(\cX)$ as an element of $\cMot(\cX)$ if convenient.

\subsection{Algebraic groups}
\begin{lemma}\label{lm:GLnBun}
Let $\GL_n$ act on an Artin stack $\cX$. Then we have in $\Mot(\cX/\GL_n)$
\[
    [\GL_n]\mathbf1_{\cX/\GL_n}=[\cX\to\cX/\GL_n].
\]
{\rm(}We are using the $\Mot(\kk)$-module structure on $\Mot(\cX/\GL_n)${\rm)}.
\end{lemma}
\begin{proof}
Since $\cX$ is locally of finite type, $\cX/GL(n)$ is locally of finite type as well. Thus we may assume that $\cX$ is of finite type. Recall that we have
\begin{equation}\label{eq:MotGL}
[\GL_n]=\prod_{i=0}^{n-1}(\bL^n-\bL^i).
\end{equation}

Set $\cY:=\cX/\GL_n$. We use induction on $n$. The case $n=0$ is obvious. Let $\cV:=\A^n\times^{\GL_n}\cX$ be the rank $n$ vector bundle on $\cY$ associated with the principal $\GL_n$-bundle $\cX\to\cY$. More precisely, we have $\cV=(\A^n\times\cX)/\GL_n$, where $\GL_n$ acts on $\A^n$ via the standard representation. Set $\cV':=(\A^n-\{0\})\times^{\GL_n}\cX$ so that $\cV'$ is the complement of the zero section in $\cV$. Thus we have
\begin{equation}\label{eq:VectBun}
    \bL^n\mathbf1_\cY=[\cV\to\cY]=[\cV'\to\cY]+\mathbf1_\cY.
\end{equation}

Let $\GL_{n-1}\subset\GL_n$ be the subgroup of block matrices of the form
\[
\begin{bmatrix}
1 & 0\\0 & *
\end{bmatrix}.
\]
We claim that $\cX/\GL_{n-1}$ is a rank $n-1$ vector bundle on $\cV'$. Indeed, consider the 1-morphism $\cX\to(\A^n-\{0\})\times\cX$, sending $x$ to $((1,0,\ldots,0),x)$. Its composition with the projection to $\cV'$ is $\GL_{n-1}$-invariant because $\GL_{n-1}$ stabilizes $(1,0,\ldots,0)$. Thus we get a 1-morphism $\cX/\GL_{n-1}\to\cV'$. We need to show that it is a rank $n-1$ vector bundle. This is enough to check after a smooth base change, so we may assume that $\cX=\GL_n\times\cY$, where $\GL_n$ acts on the first factor. In this case the statement is standard.

By induction hypothesis, we get in $\Mot(\cX/\GL_{n-1})$: $[\cX\to\cX/\GL_{n-1}]=[\GL_{n-1}]\mathbf1_{\cX/\GL_{n-1}}$. Applying~$f_!$ to both sides, where $f:\cX/\GL_{n-1}\to\cY$ is the projection, we get $[\cX\to\cY]=\bL^{n-1}[\GL_{n-1}][\cV'\to\cY]$. Combining with~\eqref{eq:VectBun} and~\eqref{eq:MotGL} we get the statement of the lemma.
\end{proof}

\begin{corollary}\label{cor:GLnTorsor}
Assume that $\cX$ is a stack of finite type over a stack $\cS$ and that the action of $\GL_n$ on $\cX$ commutes with the projection to $\cS$. Then we have in $\Mot(\cS)$
\[
    [\cX\to\cS]=[\GL_n][\cX/\GL_n\to\cS].
\]
\end{corollary}
\begin{proof}
Apply $f_!$ to the equality given by the above lemma, where $f:\cX/\GL_n\to\cS$ is the structure 1-morphism.
\end{proof}

Recall that an algebraic $\kk$-group $G$ is called special, if every principal $G$-bundle on a scheme is Zariski locally trivial\footnote{We always assume $\kk$-groups to be smooth, which is automatic if $\kk$ has characteristic zero.}. Note that if $V$ is a $\kk$-vector space, then $V$ (with its additive group structure) can be viewed as an algebraic $\kk$-group; it is easily seen to be special (cf.~\cite[Sect.~2]{Joyce07}).

\begin{lemma}\label{lm:MotBG}
Let $G$ be a special $\kk$-group, $\mathrm BG$ be the classifying stack of $G$. Then we have $[G][\mathrm BG]=1$ in $\Mot(\kk)$.
\end{lemma}
\begin{proof}
First of all, in the case $G=\GL_n$ the statement follows easily from Corollary~\ref{cor:GLnTorsor}. In particular, the class of $\GL_n$ is invertible in $\Mot(\kk)$.

Consider a closed embedding $G\to\GL_n$ and note that the quotient $\GL_n/G$ is a scheme. Since $G$ is special, we have $[\GL_n]=[G][\GL_n/G]$. On the other hand, by Corollary~\ref{cor:GLnTorsor}, we get
\[
    [\GL_n/G]=[\GL_n][(\GL_n/G)/\GL_n]=[\GL_n][\mathrm BG].
\]
The lemma follows easily from the two equations and the fact that $[\GL_n]$ is invertible in $\Mot(\kk)$.
\end{proof}

\subsection{Bilinear form}\label{sect:BilinearForm}
Let $\cZ$ be a $\kk$-stack of finite type. If $\cX$ and $\cY$ are of finite type over $\cZ$, we set
\[
    ([\cX\to\cZ]|[\cY\to\cZ])=[\cX\times_{\cZ}\cY].
\]
Extending this by bilinearity, we get a symmetric bilinear form $\Mot(\cZ)\otimes\Mot(\cZ)\to\Mot(\kk)$. We extend this by continuity to a symmetric form $\cMot(\cZ)\otimes\cMot(\cZ)\to\cMot(\kk)$.

Now, let $\cZ$ be a $\kk$-stack locally of finite type. Let $A\in\cMot^{fin}(\cZ)$, $B\in\cMot(\cZ)$. Write $A=j_! A_\cV$, where $\cV\in\Opf(\cZ)$, $j:\cV\to\cX$ is the open immersion. Let $B$ be given by an inverse system $\cU\mapsto B_\cU$, where $\cU$ ranges over $\Opf(\cZ)$. Set $(A|B):=(A_\cV|B_\cV)$. One checks that this does not depend on the choice of $\cV$ (to prove this, one first proves  Lemma~\ref{lm:MotBilProd} below in the case, when $\cZ$ and $\cZ'$ are of finite type over~$\kk$). In this way, we get a continuous bilinear form
\[
    (\bullet|\bullet):\cMot^{fin}(\cZ)\otimes\cMot(\cZ)\to\cMot(\kk).
\]
Note that the restriction of this form to $\cMot^{fin}(\cZ)\otimes\cMot^{fin}(\cZ)$ is symmetric.

We abuse notation by writing sometimes $(A|B)$ instead of $(B|A)$ when $B\in\cMot^{fin}(\cZ)$ but $A\notin\cMot^{fin}(\cZ)$.

The following lemma is immediate.
\begin{lemma}\label{lm:MotBilProd} If $f:\cZ\to\cZ'$ is a 1-morphism of finite type, then

(i)
for all $A\in\cMot^{fin}(\cZ)$, $B\in\cMot(\cZ')$ we have
\[
    (A|f^*B)=(f_!A|B).
\]

(ii) for all $A\in\cMot^{fin}(\cZ')$, $B\in\cMot(\cZ)$ we have
\[
    (f^*A|B)=(A|f_!B).
\]
\end{lemma}

\subsection{Constructible subsets of stacks}\label{sect:Constr}
Let $\cS$ be an Artin stack of finite type over $\kk$. A subset $\cX\subset|\cS|$ is called \emph{constructible\/} if it belongs to the Boolean algebra generated by the sets of points of open substacks of $\cS$. A \emph{stratification\/} of a constructible subset $\cX\subset|\cS|$ is a finite collection $\cT_i$ of constructible subsets of $\cS$ such that $\cX=\sqcup_i\cT_i$.
Let $\cX$ be a constructible subset of a finite type stack $\cS$. Consider a stratification $\cX=\sqcup_i|\cY_i|$, where $\cY_i$ are locally closed substacks. Set
\[
    \mathbf1_{\cX,\cS}:=\sum_i[\cY_i\to\cS]\in\Mot(\cS).
\]
It is easy to see that $\mathbf1_{\cX,\cS}$ does not depend on the stratification of $\cX$.

If $\cS$ is a stack locally of finite type, we call $\cX\subset|\cS|$ \emph{constructible\/}, if for every $\cU\in\Opf(\cS)$ the set $\cX\cap|\cU|$ is a constructible subset of $\cU$. In this case, we define $\mathbf1_{\cX,\cS}$ via the inverse system $\cU\mapsto\mathbf1_{\cX\cap|\cU|,\cU}$. We sometimes write $\mathbf1_\cX$ instead of $\mathbf1_{\cX,\cS}$, when $\cS$ is clear.
If $g:\cS\to\cT$ is a 1-morphism of finite type and $\cX\subset|\cS|$ is a constructible subset, we use the notation $[\cX\to\cT]:=g_!\mathbf1_{\cX,\cS}$. When $\cT=\Spec\kk$, we write $[\cX]$ for $[\cX\to\Spec\kk]$.

A constructible subset $\cX\subset|\cS|$ is of \emph{finite type\/}, if there is $\cU\in\Opf(\cS)$ such that $\cX\subset|\cU|$. In this case, $\mathbf1_{\cX,\cS}\in\Mot^{fin}(\cS)$, so if $g:\cS\to\cT$ is a 1-morphism locally of finite type, then we can define $[\cX\to\cT]:=g_!\mathbf1_{\cX,\cS}$.

Let $\cS$ and $\cS'$ be stacks of finite type and $\cX\subset|\cS|$, $\cX'\subset|\cS'|$ be their constructible subsets. Let $\cX=\sqcup_i|\cT_i|$, $\cX'=\sqcup_i|\cT'_i|$ be their stratifications by locally closed substacks. We define the product of $\cS$ and $\cS'$ via $\cS\times\cS'=\sqcup_{i,j}|\cT_i\times\cT'_j|$.\footnote{Note that this is not the usual product of sets. The reason is that for stacks (even schemes) $\cT$ and $\cT'$ we have in general $|\cT\times\cT'|\ne|\cT|\times|\cT'|$.} It is easy to check that this product does not depend on stratifications and that we have $[\cS\times\cS']=[\cS]\times[\cS']$ in $\Mot(\kk)$. It is also easy to extend the definition to any finite number of multiples.

\begin{remark}\label{rm:KS}
Our definition of motivic functions is essentially equivalent to that of~\cite[Sect.~4.2]{KontsevichSoibelman08}. In~\cite{KontsevichSoibelman08} a category of constructible stacks is defined. Intuitively, constructible stacks are Artin stacks ``up to stratification''. Precisely, the objects of the category are pairs $(X,G)$ where $X$ is a $\kk$-scheme of finite type, $G$ is a linear group acting on $X$. We will not spell out the precise definition of morphisms here but note that one can define an equivalent category as a category whose objects are pairs $(\cX,\cS)$, where $\cX$ is a constructible subset of the stack $\cS$. The equivalence of categories is given by $(X,G)\mapsto(|X/G|,X/G)$. In loc.~cit.~the group of stack motivic functions is defined over a constructible stack.
\end{remark}

\subsection{Relation with motivic classes of varieties}\label{sect:MotVar}
Define $\Mot_{var}(\kk)$ as the abelian group generated by the isomorphism classes of $\kk$-varieties (=reduced schemes of finite type over $\kk$) subject to the relation $[Z_1]=[Z_2]+[(Z_1-Z_2)]$ whenever $Z_2$ is a closed subvariety of $Z_1$. The direct product equips $\Mot_{var}(\kk)$ with a ring structure. There is an obvious homomorphism $\Mot_{var}(\kk)\to\Mot(\kk)$. This homomorphism clearly extends to the localization
\[
    \Mot_{var}(\kk)[\bL^{-1},(\bL^i-1)^{-1}|i>0]\to\Mot(\kk).
\]
It is easy to see that the above homomorphism is an isomorphism (see e.g.~\cite[Thm.~1.2]{Ekedahl09}).

Following~\cite[Sect.~2.1]{BehrendDhillon}, we define the dimensional completion of $\Mot_{var}(\kk)$ as follows. Denote by $F^m\Mot_{var}(\kk)[\bL^{-1}]$ the subgroup of the localization $\Mot_{var}(\kk)[\bL^{-1}]$ generated by $[X]/\bL^n$ with $\dim X-n\le-m$. This is a ring filtration and we define the completed ring $\cMot_{var}(\kk)$ as the completion of $\Mot_{var}(\kk)[\bL^{-1}]$ with respect to this filtration. Obviously, $\bL$ and $\bL^i-1$ are invertible in $\cMot_{var}(\kk)$ so we have a homomorphism $\Mot(\kk)\to\cMot_{var}(\kk)$. It is not difficult to show that this extends by continuity to an isomorphism
$\cMot(\kk)\xrightarrow{\simeq}\cMot_{var}(\kk)$.

Recall that in Section~\ref{sect:zeta} we defined motivic zeta-functions of varieties. Now we can define the zeta-function of a motivic class (see eq.~\eqref{eq:zeta}), when $\kk$ is a field of characteristic zero. Note the following well-known statement.
\begin{lemma}\label{lm:zeta}
(i) If $Z\subset Y$ is a closed subvariety, then $\zeta_Y=\zeta_Z\zeta_{Y-Z}$.

(ii) $\zeta_{\A^n\times Y}(z)=\zeta_Y(\bL^n z)$.
\end{lemma}

Now we can use $(i)$ to extend zeta-functions to $\Mot_{var}(\kk)$ and use $(ii)$ to extend to $\Mot_{var}(\kk)[\bL^{-1}]$. It remains to extend $\zeta$ to $\cMot_{var}(\kk)=\cMot(\kk)$ by continuity.

\subsection{Checking equality of motivic functions fiberwise}
The following statement will be our primary way to check that motivic functions are equal.

\begin{proposition}\label{pr:pointwise zero}
Let $A,B\in\Mot(\cX)$ be motivic functions. Assume that for any field $K$ and any point $\xi:\Spec K\to\cX$ we have $\xi^*A=\xi^*B$. Then $A=B$.
\end{proposition}
Viewing $\xi^*A$ as the ``value'' of $A$ at $\xi$, we can reformulate the proposition as the statement that equality of motivic functions can be checked pointwise.

\begin{proof}
We may assume that $B=0$ and that $\cX$ is of finite type. If $\cX=\sqcup_i\cT_i$ is a stratification of $\cX$ by locally closed substacks, and for all $i$ we have $A|_{\cT_i}=0$, then $A=0$. Thus, using~\cite[Prop.~3.5.6, Prop.~3.5.9]{KreschStacks}, we may assume that $\cX=X/\GL_n$ is a global quotient, where $X$ is a scheme of finite type over a field. By Lemma~\ref{lm:GLnBun} it is enough to show that the pullback of $A$ to $X$ is zero, so we may assume that $\cX=X$ is a scheme. We may also assume that $X$ is integral. Let $\xi:\Spec K\to X$ be the generic point. It is enough to show that $\xi^*A=0$ implies that there is an open subset $U\subset X$ such that $A|_U=0$.

Next, multiplying $A$ by an invertible element of $\Mot(\kk)$, we may assume that $A=\sum_i n_i[V_i\to X]$, where $V_i$ are $X$-schemes. Indeed, we may assume that $A$ is a combination of classes of stacks of the form $V/\GL_n$ but we have $[\GL_n][V/\GL_n\to X]=[V\to X]$ by Corollary~\ref{cor:GLnTorsor}.

Now by~\cite[Thm.~1.2]{Ekedahl09}, multiplying once more by an invertible element of $\Mot(\kk)$ if necessary, we may assume that in the free abelian group generated by isomorphism classes of $K$-varieties we have
\[
    \sum_i n_i[(V_i)_\xi]=\sum_i m_i([Y_i]-[Z_i]-[Y_i-Z_i]),
\]
where $Y_i$ are affine $K$-varieties and $Z_i$ are their closed subvarieties. Clearing denominators, we see that there is an open subset $W\subset X$, varieties $Y'_i$ over $W$, and their closed subvarieties $Z'_i$ such that $(Y'_i)_\xi\approx Y_i$, and under this isomorphism $(Z'_i)_\xi$ goes to $Z_i$.

Thus $\sum_i n_i[(V_i)_\xi]=\sum_i m_i([(Y'_i)_\xi]-[(Z'_i)_\xi]-[(Y'_i)_\xi-(Z'_i)_\xi]$. It follows that there is an open subset $U\subset W$ such that
\[
    \left.\left(\sum_i n_i[V_i\to X]\right)\right|_U=
    \left.\left(\sum_i m_i([Y'_i\to W]-[Z'_i\to W]-[(Y'_i-Z'_i)\to W]\right)\right|_U.
\]
We see that $A|_U=0$.
\end{proof}

\begin{corollary}\label{cor:pointwEqual}
Let $f:\cX\to\cY$ be a finite type 1-morphism of stacks inducing for every $K\supset\kk$ an equivalence of groupoids $\cX(K)\to\cY(K)$. Then $[\cX\to\cY]=\mathbf1_\cY$.
\end{corollary}
\begin{proof}
We would like to apply the previous proposition. Let $\xi:\Spec K\to\cY$ be a point and let $\cX_\xi$ be the $\xi$-fiber of $f$. We need to show that $[\cX_\xi]=[\Spec K]$ in $\Mot(K)$. This is easy if $\cX$ and $\cY$ are schemes: then the fiber is a 1-point scheme, so we have $[\cX_\xi]=[(\cX_\xi)_{red}]=[\Spec K]$.

In general, $\cX_\xi$ is a $K$-stack such that for all extensions $K'\supset K$ the groupoid $\cX_\xi(K')$ is equivalent to the trivial one. In particular, $|\cX_\xi|$ consists of a single point. Thus, according to the Kresch's result, we have $(\cX_\xi)_{red}=X/\GL_n$, where $X$ is a $K$-scheme. The unique $K$-point of $\cX_\xi$ gives rise to a $\GL_n$-equivariant morphism $\GL_n\to X$. Also, for any extension $K'\supset K$ this morphism induces an isomorphism $\GL_n(K')\to X(K')$. But we already know the statement for schemes, so we have $[\GL_n]=[X]$. It follows that $[\cX_\xi]=[(\cX_\xi)_{red}]=[X]/[\GL_n]=[\Spec K]$. Now we can apply the proposition.
\end{proof}

\section{Moduli stacks of connections, Higgs bundles, and vector bundles with nilpotent endomorphisms}\label{sect:Conn=Higgs}
In this section we introduce various stacks and provide relations between their motivic classes. In particular, we prove Theorem~\ref{th:conn=higgs} in Section~\ref{sect:CompHiggsConn}. We also give a relation between the moduli stacks of Higgs bundles, moduli stacks of vector bundles with endomorphisms (Lemma~\ref{lm:HiggsEnd}), and moduli stacks of vector bundles with nilpotent endomorphisms (Proposition~\ref{pr:NilpEndPow}). In this section $\kk$ is a fixed field of characteristic zero and $K$ denotes an arbitrary field extension of $\kk$.

\subsection{Krull--Schmidt theory for coherent sheaves}\label{App:KrSchm}
The results of this section are well-known but we include them here for the reader's convenience. In this section $X$ is a smooth connected projective variety over $\kk$. For a vector bundle $E$ on $X$ we denote by $\End^{nil}(E)$ the nilradical of the finite dimensional $\kk$-algebra $\End(E)$.
\begin{proposition}\label{pr:KrullSchmidt1}
(i) Let $F$ be an indecomposable vector bundle on $X$ and $\Psi\in\End(F)$. Then either $\Psi$ is nilpotent, or $\Psi$ is an automorphism.

(ii) Write a vector bundle $F$ as $F=\bigoplus_{i=1}^t F_i$, where $F_i\approx E_i^{\oplus n_i}$, and $E_i$ are pairwise non-isomorphic indecomposable bundles, $n_i>0$. Then we have
\[
    \bigoplus_{i\ne j}\Hom(F_i,F_j)\subset\End^{nil}(F).
\]
\end{proposition}
\begin{proof}
(i) The increasing sequence of subsheaves $\Ker(\Psi^n)\subset F$ must stabilize at some $n$. Replacing $\Psi$ by $\Psi^n$ we may thus assume that $\Ker(\Psi^2)=\Ker\Psi$, that is, $\Ker\Psi\cap\Im\Psi=0$. Thus the inclusion morphism $\Ker\Psi\oplus\Im\Psi\to F$ is injective. Since both sheaves have the same Hilbert polynomial, the morphism must be an isomorphism. The statement follows.

(ii) Let $F'$ be an indecomposable component of $F_i$, $F''$ be an indecomposable component of $F_j$. It is enough to show that $\Hom(F',F'')\subset\End^{nil}(F)$. Let $\Psi\in\Hom(F',F'')$. We need to show that for any $\Psi'\in\End(F)$, $\Psi\Psi'$ is nilpotent. Replacing $\Psi'$ by its component with respect to a direct sum decomposition, we may assume that $\Psi'\in\Hom(F'',F')$. By part (i) $\Psi\Psi'$ is either nilpotent or an isomorphism. But the second possibility is ruled out by an assumption.
\end{proof}

The following proposition is~\cite[Thm.~3]{Atiyah-KrullSchmidt} when $\kk$ is algebraically closed. The proof, in fact, goes through for any field. Alternatively, it is easy to derive this proposition from the previous one.
\begin{proposition}\label{pr:KrullSchmidt2}
Let $F$ be a vector bundle on $X$. Write $F=\bigoplus_{i=1}^t F_i$, where $F_i\approx E_i^{\oplus n_i}$, and $E_i$ are pairwise non-isomorphic indecomposable bundles, $n_i>0$. This decomposition of $F$ into the direct sum of indecomposables is unique up to permutation. That is, if $F=\bigoplus_{i=1}^t F'_i$, $F'_i\approx(E'_i)^{\oplus m_i}$, where $E'_i$ are pairwise non-isomorphic indecomposable bundles, $m_i>0$, then after renumeration of summands we get $n_i=m_i$, $E_i\approx E'_i$.
\end{proposition}

\subsection{Stacks of vector bundles and HN-filtrations}\label{sect:HN}
Let $\Bun_r$ be the stack of vector bundles of rank~$r$ (over the fixed curve $X$). Let $\Bun_{r,d}$ be its connected component classifying bundles of degree $d$. Recall that a vector bundle $E$ on $X_K$ (where, as usual, $K$ is an extension of $\kk$) is \emph{semistable} if for every subbundle $F\subset E$ we have
\[
    \frac{\deg F}{\rk F}\le\frac{\deg E}{\rk E}.
\]
According to~\cite[Prop.~3]{Langton75}, if $K'\supset K$ is a field extension, then $E_{K'}$ is semistable if and only if $E$ is semistable. The number $\deg E/\rk E$ is called the \emph{slope\/} of $E$. It is well known that each vector bundle $E$ on $X_K$ possesses a unique filtration
\[
    0=E_0\subset E_1\subset\ldots\subset E_t=E
\]
such that for $i=1,\ldots,t$ the sheaf $E_i/E_{i-1}$ is a semistable vector bundle and for $i=1,\ldots,t-1$ the slope of $E_i/E_{i-1}$ is strictly greater than the slope of $E_{i+1}/E_i$ (see ~\cite[Sect.~1.3]{HarderNarasimhan}). This filtration is called the Harder--Narasimhan filtration (or HN-filtration for brevity) on $E$ and the sequence of slopes $(\tau_1>\ldots>\tau_t)$, where $\tau_i=\deg(E_i/E_{i-1})/\rk(E_i/E_{i-1})$, is called the \emph{HN-type\/} of $E$. It follows from~\cite[Prop.~3]{Langton75} that HN-type is compatible with field extensions.

\begin{lemma}\label{lm:Bun+}
(i) There is an open substack $\Bun_{r,d}^{\ge\tau}\subset\Bun_{r,d}$ classifying vector bundles whose HN-type $(\tau_1>\ldots>\tau_t)$ satisfies $\tau_t\ge\tau$.

(ii) A vector bundle $E\in\Bun_{r,d}(K)$ is in $\Bun_{r,d}^{\ge\tau}(K)$ if and only if there is no surjective morphism of vector bundles $E\to F$ such that the slope of $F$ is less than $\tau$.

(iii) The stack $\Bun_{r,d}^{\ge\tau}$ is of finite type.

(iv) A constructible subset $\cX\subset|\Bun_r|$ is of finite type if and only if there are $\tau$ and $d_1$,\ldots,$d_n$ such that
\[
    \cX\subset\cup_{i=1}^n|\Bun_{r,d_i}^{\ge\tau}|.
\]
\end{lemma}
\begin{proof}
    (i) Since the HN-type is compatible with field extensions, we may assume that $\kk$ is algebraically closed. In this case this follows from~\cite[Thm.~3 and Prop.~10]{ShatzHN} (or~\cite[Thm.~1.7]{MaruyamaBoundedness}).

    (ii) The `if' direction is obvious. For the `only if', let $0=E_0\subset E_1\subset\ldots\subset E_t=E$ be the HN filtration of $E$ and assume that we have a surjective morphism $E\to F$, where the slope of $F$ is less than $\tau$. Let $0=F_0\subset F_1\subset\ldots\subset F_s=F$ be the HN filtration on $F$. Clearly, the slope of $F_s/F_{s-1}$ is less than $\tau$. Thus, replacing $F$ with $F_s/F_{s-1}$ and the morphism $E\to F$ with its composition with the projection to $F_s/F_{s-1}$, we may assume that $F$ is semistable.

Now, if the slope of $E_t/E_{t-1}$ is greater or equal to $\tau$, then for all $i$ the are no non-zero morphisms $E_i/E_{i-1}\to F$ (because these bundles are semistable and the slope of $E_i/E_{i-1}$ is greater, then the slope of $F$.) But then there are no non-zero morphisms from $E$ to $F$ and we come to a contradiction.

(iii) According to~\cite[Lemma~1.7.6]{HuybrechtsLehnModuli}, it is enough to show that all vector bundles $E$ in $\Bun_{r,d}^{\ge\tau}$
       are $m$-regular for $m>3-2g-\tau$.

    Since $X$ is a curve, $m$-regularity just means that $H^1(X,E\otimes\cO_X(m-1))=0$. By Serre duality, this cohomology group is dual to $\Hom(E,\Omega_X^{-1}\otimes\cO_X(1-m))$. The latter space is zero by part (ii), since the slope of $\Omega_X^{-1}\otimes\cO_X(1-m)$ is equal to $3-2g-m$.

    (iv) The `if' part follows from (i). For the converse note that
    \[
    \{\Bun_{r,d}^{\ge\tau}|d\in\Z,\tau\in\Z\}
    \]
    is an open cover of $\Bun_r$. Thus $\cX$, being quasi-compact, is covered by finitely many $\Bun_{r,d}^{\ge\tau}$.
\end{proof}

We will be mostly interested in the stack $\Bun_{r,d}^{\ge0}$. We will call such vector bundles `HN-nonnegative'. Note that the tensorisation with a line bundle of degree $e$ gives an isomorphism $\Bun_{r,d}^{\ge0}\simeq\Bun_{r,d+er}^{\ge e}$. It follows from Lemma~\ref{lm:Bun+}(ii) that $E$ is HN-nonnegative if and only if there are no surjective morphisms $E\to F$, where $F$ is a vector bundle such that $\deg F<0$.

\subsubsection{Isoslopy vector bundles}
We will call a vector bundle $E$ on $X_K$ \emph{isoslopy\/} if it cannot be written as the direct sum of two vector bundles of different slope.
\begin{lemma}\label{lm:isoiii}
A vector bundle $E$ on $X$ is isoslopy if and only if its pullback to $X_K$ is isoslopy.
\end{lemma}
\begin{proof}
The `if' direction is obvious. For the `only if' direction we note first that the sum of isoslopy bundles of the same slope is isoslopy because of uniqueness of decomposition (Proposition~\ref{pr:KrullSchmidt2}). Thus it is enough to prove that if $E$ is an indecomposable vector bundle on $X$, then $E_K$ is isoslopy. We follow the strategy of the proof of~\cite[Prop.~3]{Langton75}. We may assume that $K\supset\kk$ is a finitely generated extension. In view of the `if' direction, it is enough to consider two cases: (i) $K$ is an algebraic closure of $\kk$, (ii) $K=\kk(t)$ is purely transcendental of degree~1. In case (i) the statement follows from the fact that the Galois group of $K$ over $\kk$ acts transitively on indecomposable summands of $E_K$.

Finally, if $E_{\kk(t)}$ is the direct sum of two vector bundles of different slopes, then there is an open subset $U\subset\A^1_\kk$ such that the pullback of $E$ to $X\times U$ is the direct sum of two vector bundles of different slopes (just clear denominators). Restricting this pullback to $X\times u$, where $u\in U$ is a $\kk$-rational point, we come to contradiction.
\end{proof}

By the above lemma, the equivalence relation from Section~\ref{sect:MotFun} on the points of $\Bun_{r,d}$ preserves isoslopy bundles. Thus we have a well-defined set $\Bun_{r,d}^{iso}\subset|\Bun_{r,d}|$. Set also
$\Bun_{r,d}^{\ge0,iso}=|\Bun_{r,d}^{\ge0}|\cap\Bun_{r,d}^{iso}$.

\begin{lemma}\label{lm:iso}
(i) If $\ell$ is a line bundle on $X$ of degree $N>(r-1)(g-1)-d/r$, then tensorisation with~$\ell$ (which is a 1-morphism $\Bun_{r,d}\to\Bun_{r,d+Nr}$) induces a bijection
    \[
        \Bun^{iso}_{r,d}\xrightarrow{\simeq}\Bun^{\ge0,iso}_{r,d+Nr}.
    \]

(ii) $\Bun^{iso}_{r,d}\subset|\Bun_{r,d}|$ is a constructible subset of finite type.
\end{lemma}
\begin{proof}
    (i) Clearly, tensorisation with $\ell$ induces a bijection $\Bun^{iso}_{r,d}\xrightarrow{\simeq}\Bun^{iso}_{r,d+Nr}$.
It remains to show that $\Bun^{\ge0,iso}_{r,d+Nr}=\Bun^{iso}_{r,d+Nr}$. By contradiction, assume that $E\in\Bun^{iso}_{r,d+Nr}$ but
$E\notin\Bun^{\ge0,iso}_{r,d+Nr}$. Then $E$ is decomposable by~\cite[Cor.~4.2]{MozgovoySchiffmanOnHiggsBundles} (Formally speaking, the statement is only formulated for curves over finite fields, but the proof works over any field). Let $E_0$ be an indecomposable summand of $E$ such that its HN-type $(\tau_1>\ldots>\tau_t)$ satisfies $\tau_t<0$ (it exists by Lemma~\ref{lm:Bun+}(ii)). By the definition of isoslopy bundles, the slope of $E_0$ is equal to $d/r+N$, and clearly $d/r+N>(\rk E_0-1)(g-1)$, so $E_0$ cannot be indecomposable (again by~\cite[Cor.~4.2]{MozgovoySchiffmanOnHiggsBundles}). This contradiction completes the proof of (i).

Let us prove part (ii). By part (i), it is enough to prove the statement for $\Bun^{\ge0,iso}_{r,d}$ (just replace $d$ by $d+Nr$ with large $N$). Let $\Pi$ be the set of all quadruples $(r',d',r'',d'')\in\Z_{\ge0}^4$ such that $r'+r''=r$, $d'+d''=d$, and $r'/d'\ne r/d$. Note that this set is finite.

For $\pi=(r',d',r'',d'')\in\Pi$, let $\cB_\pi$ be the image of $\Bun^{\ge0}_{r',d'}\times\Bun^{\ge0}_{r'',d''}$ under the morphism, sending two vector bundles to their direct sum. Combining Lemma~\ref{lm:Bun+}(iii) with the stacky Chevalley theorem, we see that $\cB_\pi$ is a constructible subset of $\Bun^{\ge0}_{r,d}$. One easily checks that
\[
    \Bun^{\ge0,iso}_{r,d}=|\Bun^{\ge0}_{r,d}|-\cup_{\pi\in\Pi}\cB_\pi.
\]
Thus $\Bun^{\ge0,iso}_{r,d}$ is constructible. Obviously, it is of finite type.
\end{proof}

\subsection{Higgs bundles whose underlying vector bundle is HN-nonnegative}\label{Sect:Higgs}
Recall from Section~\ref{sect:ModStack} the Artin stack $\cM_{r,d}$ classifying Higgs bundles of rank $r$ and degree $d$. A simple argument similar to the proof of~\cite[Prop.~1]{FedorovIsoStokes} shows that it is an Artin stack locally of finite type and that the forgetful 1-morphism $(E,\Phi)\mapsto E$ is a schematic 1-morphism of finite type $\cM_{r,d}\to\Bun_{r,d}$. Set
\[
    \cM^{\ge0}_{r,d}:=\cM_{r,d}\times_{\Bun_{r,d}}\Bun_{r,d}^{\ge0};
\]
by Lemma~\ref{lm:Bun+}(i), it is an open substack of finite type of $\cM_{r,d}$.

On the other hand, recall from Section~\ref{sect:ModStack} that a Higgs bundle $(E,\Phi)\in\cM_{r,d}(K)$ is called semistable if the slope of any subbundle $F\subset E$ is less or equal than the slope of $E$, provided that $F$ is preserved by $\Phi$. An argument similar to~\cite[Prop.~3]{Langton75} shows that this notion is stable with respect to field extensions. We emphasize that semistability of $(E,\Phi)$ does not imply in general semistability of $E$. According to~\cite[Lemma~3.7]{Simpson1}\footnote{This Lemma is formulated in the case, when the field is the field of complex numbers. However, the proof goes through for any field.}, there is an open substack $\cM^{ss}_{r,d}$ classifying semistable Higgs bundles.

We call a Higgs bundle $(E,\Phi)$ on $X_K$ \emph{nonnegative-semistable\/} if $E$ is HN-nonnegative and whenever $F\subset E$ is an HN-nonnegative vector subbundle preserved by $\Phi$, the slope of $F$ is less or equal than the slope of $E$; an argument similar to~\cite[Prop.~3]{Langton75} shows that this notion is stable with respect to field extensions. Denote the stack of nonnegative-semistable Higgs bundles of rank $r$ and degree $d$ by $\cM_{r,d}^{\ge0,ss}$; an argument similar to~\cite[Lemma~3.7]{Simpson1} shows that this is an open substack of $\cM_{r,d}$.

\begin{remark}
In general, $\cM_{r,d}^{\ge0,ss}\ne\cM_{r,d}^{\ge0}\cap\cM_{r,d}^{ss}$. The reason is that a nonnegative semistable Higgs bundle $(E,\Phi)$ might have a destabilizing subbundle $F$ such that $F$ is not HN-nonnegative, so it is not necessarily semistable in the usual sense.
\end{remark}

\begin{lemma}\label{lm:+ss}
(i) If $\ell$ is a line bundle on $X$ of degree $N>(r-1)(g-1)-d/r$, then tensorisation with $\ell$ induces an isomorphism
    \[
        \cM^{ss}_{r,d}\xrightarrow{\simeq}\cM^{\ge0,ss}_{r,d+Nr}.
    \]

(ii) $\cM^{ss}_{r,d}$ is a stack of finite type.
\end{lemma}
\begin{proof}
    (i)
    Clearly, tensorisation with $\ell$ induces an isomorphism $\cM^{ss}_{r,d}\xrightarrow{\simeq}\cM^{ss}_{r,d+Nr}$. It remains to notice that, according to~\cite[Cor.~3.3]{MozgovoySchiffmanOnHiggsBundles}, we have $\cM^{ss}_{r,d+Nr}=\cM^{\ge0,ss}_{r,d+Nr}$. (Formally speaking, the statement is only formulated for curves over finite fields, but the proof works over any field.)

Part (ii) is an obvious corollary of part (i).
\end{proof}

\subsection{Connections and isoslopy Higgs bundles}\label{sect:ConnIsoslopy}
Recall that $\Conn_r$ is the moduli stack of rank $r$ vector bundles with connections. An argument similar to the proof of~\cite[Prop.~1]{FedorovIsoStokes} shows that it is an Artin stack locally of finite type and that the forgetful 1-morphism $(E,\nabla)\mapsto E$ is a schematic 1-morphism of finite type $\Conn_r\to\Bun_{r,0}$. We are using the well-known fact that a vector bundle admitting a connection must be of degree zero.

Let $\cM^{\ge0,iso}_{r,d}\subset|\cM^{\ge0}_{r,d}|$ be the set of points corresponding to Higgs bundles $(E,\Phi)$ such that $E\in\Bun^{\ge0,iso}_{r,d}$. It follows from Lemma~\ref{lm:iso}(ii) that $\cM^{\ge0,iso}_{r,d}$ is a constructible subset of finite type.

\begin{proposition}\label{pr:+iso}
The stack $\Conn_r$ is of finite type and we have in $\Mot(\kk)$
    \[
        [\Conn_r]=[\cM^{iso}_{r,0}].
    \]
\end{proposition}
\begin{proof}
By Weil's theorem the image of $\Conn_r$ in $\Bun_{r,0}$ is exactly $\Bun^{iso}_{r,0}$ (Note that we only need to use the Weil's theorem for an algebraic closure of $\kk$ because we know a priori that this image is constructible. By elementary logic, it is enough to know that the theorem is true for the field of complex numbers).

It is enough to show that we have in $\Mot(\Bun_{r,0})$:
\begin{equation}\label{eq:Conn=HiggsRel}
    [\cM^{iso}_{r,0}\to\Bun_{r,0}]=[\Conn_r\to\Bun_{r,0}].
\end{equation}

We want to apply Proposition~\ref{pr:pointwise zero}. Let $\xi:\Spec K\to\Bun_{r,0}$ be a point. It corresponds to a vector bundle $E$ on $X_K$. If $E$ is not isoslopy, then the pullbacks of both sides of~\eqref{eq:Conn=HiggsRel} are zero. If $E$ is isoslopy, then the pullback of the LHS of~\eqref{eq:Conn=HiggsRel} is the class of the vector space $V=H^0(X_K,\END(E)\otimes\Omega_{X_K})$, while the pullback of the RHS is the class of an affine space over this vector space, that is, a principal $V$-bundle. (Note that a priori this affine space only has a section after extending the field, but, as we noted above, a vector space with its additive group structure is a special group, so there are no non-trivial $V$-bundles on $\Spec K$.) Thus the fibers are isomorphic as schemes, so we can apply Proposition~\ref{pr:pointwise zero}, which proves~\eqref{eq:Conn=HiggsRel}. (Here $\END(E)$ denotes the sheaf of endomorphisms of $E$.)
\end{proof}

\subsection{Comparing Higgs fields and Higgs fields with isoslopy underlying vector bundle}\label{sect:Isoslopy}
Consider the following generating series
\[
    H^{\ge0}(z,w):=1+\sum_{\substack{r>0\\d\ge 0}}\bL^{(1-g)r^2}[\cM^{\ge0}_{r,d}]w^rz^d\in1+w\Mot(\kk)[[w,z]]
\]
and for a rational number $\tau\ge0$
\[
    H_\tau^{\ge0,iso}(z,w):=1+\sum_{\substack{r>0\\d/r=\tau}}\bL^{(1-g)r^2}[\cM^{\ge0,iso}_{r,d}]w^rz^d\in1+w\Mot(\kk)[[w,z]].
\]

\begin{proposition}\label{pr:IsoProd}
We have
\[
    H^{\ge0}(z,w)=\prod_{\tau\ge0}H_\tau^{\ge0,iso}(z,w).
\]
\end{proposition}
\begin{proof}
First of all we would like to reformulate the proposition. Let $\cE_{r,d}$ be the stack classifying the pairs $(E,\Psi)$, where $E$ is a vector bundle of rank $r$ and degree $d$, $\Psi$ is an endomorphism of $E$. Set $\cE^{\ge0}_{r,d}:=\cE_{r,d}\times_{\Bun_{r,d}}\Bun^{\ge0}_{r,d}$. Let $\cE^{\ge0,iso}_{r,d}\subset|\cE^{\ge0}_{r,d}|$ be the preimage of $\Bun_{r,d}^{\ge0,iso}$.

\begin{lemma}\label{lm:HiggsEnd}
We have in $\Mot(\kk)$
\[
    [\cM_{r,d}^{\ge0}]=\bL^{(g-1)r^2}[\cE^{\ge0}_{r,d}],\qquad
    [\cM_{r,d}^{\ge0,iso}]=\bL^{(g-1)r^2}[\cE^{\ge0,iso}_{r,d}].
\]
\end{lemma}
\begin{proof}
Let us prove the first equation (the second is analogous). It is enough to show that
\[
    [\cM_{r,d}^{\ge0}\to\Bun^{\ge0}_{r,d}]=\bL^{(g-1)r^2}[\cE_{r,d}^{\ge0}\to\Bun^{\ge0}_{r,d}].
\]
We want to apply Proposition~\ref{pr:pointwise zero}. Consider a point $\xi:\Spec K\to\Bun^{\ge0}_{r,d}$ given by a vector bundle $E$ on $X_K$. The $\xi$-pullback of the LHS is the class of the vector space $H^0(X_K,\END(E)\otimes\Omega_{X_K})$, while the $\xi$-pullback of the RHS is the class of the vector space $\A^{(g-1)r^2}\oplus H^0(X_K,\END(E))$. Thus we only need to check that
\[
    h^0(X_K,\END(E)\otimes\Omega_{X_K})=(g-1)r^2+h^0(X_K,\END(E)).
\]
This follows from Riemann--Roch Theorem and Serre duality.
\end{proof}
In view of this lemma we can re-write the proposition as
\[
1+\sum_{\substack{r>0\\d\ge0}}[\cE^{\ge0}_{r,d}]w^rz^d=
\prod_{\tau\ge0}\left(
    1+\sum_{\substack{r>0\\d/r=\tau}}[\cE^{\ge0,iso}_{r,d}]w^rz^d
\right).
\]

Let $\Pi_{r,d}$ be the set of all sequences
\[
    ((r_1,d_1),(r_2,d_2),\ldots,(r_t,d_t))\in(\Z_{>0}\times\Z_{\ge0})^t,
\]
where $\sum_i r_i=r$, $\sum_i d_i=d$ and the sequence $d_i/r_i$ is strictly decreasing. We note that $\Pi_{r,d}$ as a finite set. Now our proposition is equivalent to the following lemma.

\begin{lemma}\label{lm:IsoStrat}
We have in $\Mot(\kk)$
\[
    [\cE^{\ge0}_{r,d}]=\sum_{((r_i,d_i))\in\Pi_{r,d}}\prod_{i}[\cE^{\ge0,iso}_{r_i,d_i}].
\]
\end{lemma}
\begin{proof}
For any sequence $\pi=((r_i,d_i))\in\Pi_{r,d}$ consider the 1-morphism
\[
    i_\pi:\prod_i\Bun_{r_i,d_i}^{\ge0}\to\Bun_{r,d},
\]
sending a sequence of vector bundles to their direct sum. It follows from Lemma~\ref{lm:Bun+}(ii) that the image of this 1-morphism is contained in $\Bun_{r,d}^{\ge0}$. Consider the constructible subset (recall the definition of product of constructible subsets from Section~\ref{sect:Constr})
\[
    \prod_i\Bun_{r_i,d_i}^{\ge0,iso}\subset\left|\prod_i\Bun_{r_i,d_i}^{\ge0}\right|.
\]
By the stacky Chevalley theorem its image under $i_\pi$ is constructible, denote it by $\Bun_\pi$. It follows easily from the fact that isotypic components of a vector bundle are unique up to isomorphism (see Proposition~\ref{pr:KrullSchmidt2}), that $\{\Bun_\pi|\pi\in\Pi_{r,d}\}$ is a stratification of $\Bun_{r,d}^{\ge0}$. Let $\cE_\pi$ be the preimage of $\Bun_\pi$ in $\cE_{r,d}$. We see that it is enough to show that for all $\pi\in\Pi_{r,d}$ we have
\begin{equation}\label{eq:KrSchm}
    [\cE_\pi\to\Bun^{\ge0}_{r,d}]=
    \left[\prod_i\cE^{\ge0,iso}_{r_i,d_i}\to\Bun^{\ge0}_{r,d}\right].
\end{equation}
We want to apply Proposition~\ref{pr:pointwise zero}. Let $\xi:\Spec K\to\Bun_{r,d}$ be a point. If it is not in $\Bun_\pi$, then the pullbacks of both sides of the equation are zero. Otherwise, let $E$ be the vector bundle on $X_K$ corresponding to $\xi$.

\begin{claim}
The vector bundle $E$ can be written as $\bigoplus_i E_i$, where $E_i$ is an HN-nonnegative isoslopy vector bundle of rank $r_i$ and degree~$d_i$.
\end{claim}
\begin{proof}
Let $\overline K$ be an algebraic closure of $K$. By definition of $\Bun_\pi$, there is a $\overline K$-point on the fiber of $i_\pi$ over $\xi$. This means that the base-changed vector bundle $E_{\overline K}$ can be decomposed as $\overline E_1\oplus\ldots\oplus\overline E_t$, where $\overline E_i\in\Bun^{\ge0}_{r_i,d_i}(\overline K)$ is isoslopy. We need to show that $E$ can be decomposed similarly. Let us write $E=E'_1\oplus\ldots\oplus E'_s$, where $E'_1$,\ldots,$E'_s$ are indecomposable bundles. By Lemma~\ref{lm:isoiii}, $(E'_i)_{\overline K}$ is isoslopy. Note that the bundles $\overline E_1$, \ldots, $\overline E_t$ cannot have isomorphic indecomposable summands (being isoslopy of different slopes). Now the uniqueness of indecomposable summands (Proposition~\ref{pr:KrullSchmidt2}) shows that there is a partition $\{1,\ldots,s\}=I_1\sqcup\ldots\sqcup I_t$ such that $\overline E_i\approx\oplus_{j\in I_i}(E_j')_{\overline K}$. It remains to set $E_i=\oplus_{j\in I_i}E_j'$.
\end{proof}

Fix a decomposition provided by the claim. Note that the $\xi$-pullback of the LHS of~\eqref{eq:KrSchm} is the class of the vector space $\End(E)$. One checks that the $\xi$-pullback of the RHS of~\eqref{eq:KrSchm} is the class of the algebraic space representing the following functor:
\[
    S\mapsto\{(G_1,\ldots,G_t,\Psi_1,\ldots,\Psi_t):E_S=G_1\oplus\ldots\oplus G_t,\Psi_i\in\End(G_i)\},
\]
where the vector bundle $G_i$ on $X\times S$ has rank $r_i$ and degree $d_i$. (Note that if $S$ is a spectrum of a field, then each $G_i$ is isoslopy according to Proposition~\ref{pr:KrullSchmidt2} and Lemma~\ref{lm:isoiii}.) Denote this space by~$Y_E$. We need to show that $[\End(E)]=[Y_E]\in\Mot(K)$. To this end we first construct a map of sets $I:\End(E)\to Y_E(K)$ as follows. For $\Psi\in\End(E)$ let us write $\Psi=(\Psi_{ij})$, where $\Psi_{ij}\in\Hom(E_i,E_j)$. Set $\Psi':=1+\sum_{i\ne j}\Psi_{ij}$. We will use this notation through the end of the subsection.
\begin{claim}
$\sum_{i\ne j}\Psi_{ij}$ belongs to the nilpotent radical of $\End(E)$.
\end{claim}
\begin{proof}
Note that $E_i$ and $E_j$ are isoslopy and their slopes are different, so these bundles cannot have isomorphic indecomposable summands. Now the statement follows from Proposition~\ref{pr:KrullSchmidt1}(ii).
\end{proof}

By the above claim $\Psi'$ is an automorphism of $E$. Define a map $I$ by
\[
    I(\Psi)=(\Psi'(E_1),\ldots,\Psi'(E_t),\Psi'\Psi_{11}(\Psi')^{-1},\ldots,\Psi'\Psi_{tt}(\Psi')^{-1}).
\]

\begin{claim}
$I$ is an isomorphism.
\end{claim}
\begin{proof}
Assume that $I(\Psi_1)=I(\Psi_2)$. Then $\Psi'_2=\Psi'_1\Theta$, where $\Theta$ preserves the decomposition $E=E_1\oplus\ldots\oplus E_t$. Then
\[
\Id_{E_i}=(\Psi'_2)_{ii}=(\Psi'_1)_{ii}\Theta_{ii}=\Theta_{ii}.
\]
We see that $\Theta=1$, so that $\Psi'_2=\Psi'_1$. Now we also see that $(\Psi_2)_{ii}=(\Psi_1)_{ii}$, so that $\Psi_2=\Psi_1$, which proves injectivity.

Assume that $E=G_1\oplus\ldots\oplus G_t$, where $G_i$ is of rank $r_i$ and degree $d_i$; let $\Psi_i\in\End(G_i)$ for $i=1,\ldots,t$. By Proposition~\ref{pr:KrullSchmidt2}, we have an isomorphism $\Theta_i:E_i\to G_i$. Then $\Theta:=\bigoplus_i\Theta_i$ is an automorphism of~$E$. Let us write $\Theta=\Theta_1+\Theta_2$, where $\Theta_1\in\bigoplus_i\End(E_i)$, $\Theta_2\in\bigoplus_{i\ne j}\Hom(E_i,E_j)$. We have
\[
    \Theta_1=\Theta(1-(\Theta)^{-1}\Theta_2),
\]
so $\Theta_1$ is an automorphism (because $\Theta_2\in\End^{nil}(E)$). Set $\tilde\Theta:=\Theta\Theta_1^{-1}$ and finally
\[
    \Psi=\tilde\Theta-1+\sum_i(\tilde\Theta)^{-1}\Psi_i(\tilde\Theta).
\]
Note that $\tilde\Theta(E_i)=G_i$ and $\tilde\Theta_{ii}=1\in\End(E_i)$. It follows that $\Psi'=\tilde\Theta$, so that $\Psi'(E_i)=G_i$. We see that $I(\Psi)=(G_i,\Psi_i)$, which shows surjectivity of $I$.
\end{proof}

Now we complete the proof of Lemma~\ref{lm:IsoStrat}. It is easy to see that the construction of $I$ works in families, so, in fact, $I$ gives a morphism from $\End(E)$ to $Y_E$. If $K'$ is an extension of $K$, then, applying the previous claim to $E_{K'}$, we see that $I(K')$ is a bijection. Thus, by Corollary~\ref{cor:pointwEqual}, we see that $[\End(E)]=[Y_E]$. This proves~\eqref{eq:KrSchm}.
\end{proof}

Lemma~\ref{lm:IsoStrat} completes the proof of Proposition~\ref{pr:IsoProd}.
\end{proof}

\subsection{Kontsevich--Soibelman product}
The main  result of this section is a simple corollary of the general formalism of~\cite{KontsevichSoibelman08} (see also ~\cite{RenSoibelman} for the formulas in the case of 2CY categories that is most interesting for us). The general theory relies on the notion of motivic Hall algebra introduced in~\cite{KontsevichSoibelman08}. For the reader not interested in the general framework, we present below a direct proof of the necessary wall-crossing formula. The general approach is outlined in  Remark~\ref{rem:KSProduct} below.

Recall that in Section~\ref{sect:Isoslopy} we defined the generating series $H^{\ge0}(z,w)$. For $\tau\ge0$ consider one more generating series
\[
    H_\tau^{\ge0,ss}(z,w):=1+\sum_{\substack{r>0\\d/r=\tau}}\bL^{(1-g)r^2}[\cM^{\ge0,ss}_{r,d}]w^rz^d\in1+w\Mot(\kk)[[w,z]].
\]

\begin{proposition}\label{pr:KS}
\[
    H^{\ge0}(z,w)=\prod_{\tau\ge0}H_\tau^{\ge0,ss}(z,w).
\]
\end{proposition}

\begin{proof}
Let $\Pi_{r,d}$ be as in the proof of Proposition~\ref{pr:IsoProd}. For $\pi=((r_1,d_1),\ldots,(r_t,d_t))\in\Pi_{r,d}$ consider the stack classifying collections
\begin{equation*}
    (0\subset E_1\subset\ldots\subset E_t=E,\Phi),
\end{equation*}
where $E_i/E_{i-1}$ is a vector bundle of degree $d_i$ and rank $r_i$, $\Phi$ is a Higgs field on $E$ preserving each $E_i$. Denote by
$\cM^\pi$ its open substack classifying collections such that for all $i$ the Higgs pair $(E_i/E_{i-1},\Phi_i)$, where $\Phi_i$ is induced by $\Phi$, is nonnegative-semistable.

\begin{lemma}\label{lm:VectSpStack}
The stack $\cM^\pi$ is of finite type and we have in $\Mot(\kk)$
\[
    [\cM^\pi]=\bL^{(g-1)(r^2-r_1^2-\ldots-r_t^2)}\prod_i[\cM^{\ge0,ss}_{r_i,d_i}].
\]
\end{lemma}
\begin{proof}
Set
\[
    \pi'=((r_1,d_1),\ldots,(r_{t-1},d_{t-1}))\in\Pi_{(r_1+\ldots+r_{t-1},d_1+\ldots+d_{t-1})}.
\]
We will show that
\begin{equation}\label{eq:KSprod}
    [\cM^\pi]=\bL^{(2g-2)r_t(r_1+\ldots+r_{t-1})}[\cM^{\pi'}][\cM^{\ge0,ss}_{r_t,d_t}].
\end{equation}
Since $r=r_1+\ldots+r_t$, the lemma will follow by induction on $t$.

There is a 1-morphism $\Lambda:\cM^\pi\to\cM^{\pi'}\times\cM^{\ge0,ss}_{r_t,d_t}$, sending $(E,\Phi)$ to
\[
    ((E_1\subset\ldots\subset E_{t-1},\Phi|_{E_{t-1}}),(E_t/E_{t-1},\Phi')),
\]
where $\Phi'$ is the Higgs field induced by $\Phi$ on $E_t/E_{t-1}$. Let
$\xi_1=(E_1\subset\ldots\subset E_{t-1},\Phi_1)$ be a $K$-point of $\cM^{\pi'}$, $\xi_2=(E,\Phi_2)$ be a $K$-point of $\cM^{\ge0,ss}_{r_t,d_t}$. The fiber of $\Lambda$ over $(\xi_1,\xi_2)$ is the quotient
\[
    \Ext^1((E,\Phi_2),(E_{t-1},\Phi_1))/\Hom((E,\Phi_2),(E_{t-1},\Phi_1)),
\]
where the Hom space acts on the Ext space trivially. (Here $\Ext$ and $\Hom$ are calculated in the category of Higgs sheaves.) Since the additive group is special, by Lemma~\ref{lm:MotBG} the motivic class of this stack in $\Mot(K)$ is equal to $\bL^d$, where $d$ is the dimension of this stack. Using the results of~\cite[Sect.~2.1]{MozgovoySchiffmanOnHiggsBundles}, we see that
\[
    d=\deg\Omega_X\,\rk E\,\rk E_{t-1}=(2g-2)r_t(r_1+\ldots+r_{t-1}).
\]
In particular, this dimension is constant, so by Proposition~\ref{pr:pointwise zero} we have
\[
  [\cM^\pi\to\cM^{\pi'}\times\cM^{\ge0,ss}_{r_t,d_t}]=
  \bL^{(2g-2)r_t(r_1+\ldots+r_{t-1})}\mathbf1_{\cM^{\pi'}\times\cM^{\ge0,ss}_{r_t,d_t}}.
\]
Applying pushforward, we get~\eqref{eq:KSprod}.
\end{proof}

Let us return to the proof of the proposition. We have an obvious forgetful 1-morphism $\sqcup_{\pi\in\Pi_{r,d}}\cM^\pi_{r,d}\to\cM_{r,d}$. It follows from Harder--Narasimhan theory (applied to the category of Higgs bundles whose underlying vector bundle is HN-nonnegative) and Corollary~\ref{cor:pointwEqual} that
\[
    \left[\sqcup_{\pi\in\Pi_{r,d}}\cM^\pi_{r,d}\right]=[\cM^{\ge0}_{r,d}].
\]
Combining this with the previous lemma, we get
\[
    [\cM^{\ge0}_{r,d}]=\sum_{((r_1,d_1),\ldots,(r_t,d_t))\in\Pi_{r,d}}\bL^{(g-1)(r^2-r_1^2-\ldots-r_t^2)}
    \prod_i[\cM^{\ge0,ss}_{r_i,d_i}].
\]
This is equivalent to the proposition.
\end{proof}

\begin{remark}\label{rem:KSProduct}
Let us recall the general approach to the wall-crossing formulas from \cite{KontsevichSoibelman08}. Let $\cC$ be an ind-constructible category endowed with a class map $\cl: K_0(\cC)\to\Gamma\simeq\Z^n$. Let $\Gamma$ be endowed with an integer skew-symmetric form $\langle\bullet,\bullet\rangle$ such that $\cl$ intertwines this form and the skew-symmetrization of the Euler form on $K_0(\cC)$. Assume also that we are given a constructible stability structure on $\cC$ and that $V\subset\R^2$ is a strict sector. In Section~\ref{sect:QuTorus} we explained that in this situation one obtains an element $A_{{\cC}(V)}^{Hall}$ of the motivic Hall algebra $H(\cC)$. Then the following factorization formula holds:
\[
A_{{\cC}(V)}^{Hall}=\prod^\to_{l\subset V}A_{{\cC}(l)}^{Hall}.
\]	
Here $A_{{\cC}(l)}^{Hall}$ are defined similarly to $A_{{\cC}(V)}^{Hall}$ but for the categories ${\cC}(l)$ associated with each ray $l\subset V$ with the vertex at $(0,0)$. The product is taken in the clockwise order. In general there are countably many factors in the product. In the case of 3CY categories we apply the homomorphism $\Phi=\Phi_V$ and obtain a similar factorization formula for quantum DT-series, which are elements of the corresponding quantum tori. In the case of 2CY categories we apply the linear map $\Phi$ from Section~\ref{sect:QuTorus}, since it respects the product in the clockwise order. Then we obtain a similar factorization formula for quantum DT-series (they are elements of a commutative quantum torus).

In our case, the category $\cC$ is the category of Higgs bundles on $X$ such that the underlying vector bundle is HN-nonnegative. As in Section~\ref{sect:QuTorus}, the stability structure is standard with the central charge $Z(F)=-\deg F+\sqrt{-1}\rk F$ and we take strict sector $V$ to be the second quadrant $\{x\le 0, y\ge 0\}$ in the plane $\R^2_{(x,y)}$. In this case, $\cC(V)=\cC$. Applying the above considerations we obtain Proposition \ref{pr:KS}.
\end{remark}

\subsection{Comparing Higgs bundles and bundles with connections}\label{sect:CompHiggsConn}
We need a simple lemma.
\begin{lemma}\label{lm:EqSlp}
Let $R$ be a commutative ring. For a rational number $\tau\ge0$ define
\[
    R_\tau[[z,w]]=\sum_{\substack{r>0\\d/r=\tau}}Rw^rz^d\subset wR[[z,w]].
\]
For $\tau\ge0$ and $i=1,2$, assume that we are given series $H_\tau^i(z,w)\in1+R_\tau[[z,w]]$. Then
\begin{equation}\label{eq:slopes}
    \prod_{\tau\ge0} H^1_\tau(z,w)=\prod_{\tau\ge0} H^2_\tau(z,w)
\end{equation}
implies that for all $\tau$ we have $H_\tau^1(z,w)=H_\tau^2(z,w)$.
\end{lemma}
\begin{proof}
For $r>0$, $d\ge0$, let $a^i_{r,d}$ be the coefficient of $H_{d/r}^i$ at $w^rz^d$. We need to show that $a^1_{r,d}=a^2_{r,d}$. We may assume that we have $a^1_{r',d'}=a^2_{r',d'}$ whenever $r'+d'<r+d$. Equating coefficients of~\eqref{eq:slopes} at $w^rz^d$ proves the claim.
\end{proof}

\begin{theorem}\label{th:ss=iso}
We have in $\Mot(\kk)$
\[
    [\cM^{ss}_{r,d}]=[\cM^{iso}_{r,d}].
\]
\end{theorem}
\begin{proof}
Combining Proposition~\ref{pr:IsoProd} and Proposition~\ref{pr:KS} we get
\[
    \prod_{\tau\ge0}H_\tau^{\ge0,iso}(z,w)=H^{\ge0}(z,w)=\prod_{\tau\ge0}H_\tau^{\ge0,ss}(z,w).
\]
By Lemma~\ref{lm:EqSlp} we get
\[
    H_\tau^{\ge0,iso}(z,w)=H_\tau^{\ge0,ss}(z,w),
\]
that is,
\[
    [\cM^{\ge0,iso}_{r,d}]=[\cM^{\ge0,ss}_{r,d}].
\]
Next, for $N>(r-1)(g-1)-d/r$, we get using Lemmas~\ref{lm:+ss} and~\ref{lm:iso}
\[
    [\cM^{ss}_{r,d}]=[\cM^{\ge0,ss}_{r,d+Nr}]=
    [\cM^{\ge0,iso}_{r,d+Nr}]=[\cM^{iso}_{r,d}].
\]
\end{proof}

Now we are ready to prove Theorem~\ref{th:conn=higgs}.

\begin{proof}[Proof of Theorem~\ref{th:conn=higgs}]
Combine Proposition~\ref{pr:+iso} and Theorem~\ref{th:ss=iso}.
\end{proof}

\subsection{Vector bundles with nilpotent endomorphisms}
Recall from Section~\ref{sect:NilpIntro} the stack $\cE_{r,d}^{\ge0,nilp}\subset\cE_{r,d}^{\ge0}$ parameterizing HN-nonnegative vector bundles with nilpotent endomorphisms. Define a power structure on the ring $\cMot(\kk)$:
\[
    \Pow: (1+\cMot(\kk)[[w,z]]^+)^\times\times\cMot(\kk)\to\cMot(\kk)[[z,w]]:
    (f,A)\mapsto\Exp(A\Log(f)),
\]
where $\Exp$ and $\Log$ were defined in Section~\ref{sect:zeta}. This power structure has been studied in~\cite{GuzeinZadeEtAlOnLambdaRingStacks} in the case when $\kk$ is algebraically closed (it also appeared earlier in the proof of~\cite[Prop.~7]{KontsevichSoibelman10}). Our main result in this section is the following proposition.

\begin{proposition}\label{pr:NilpEndPow}
We have in $\cMot(\kk)$
\[
    1+\sum_{r,d}[\cE_{r,d}^{\ge0}]w^rz^d=\Pow\left(
    1+\sum_{r,d}[\cE_{r,d}^{\ge0,nilp}]w^rz^d,\bL
    \right).
\]
\end{proposition}
\begin{proof}
Given a collection of $\kk$-stacks $\cX_{r,d}$, where $r,d\in\Z$, we view the stack $\sqcup_{r,d}\cX_{r,d}$ as a $\Z^2$-graded stack. If $K\supset\kk$ is a finite extension and $\phi:\Spec K\to\cX_{r,d}$ is a point, we define $\cl(\phi):=[K:\kk](r,d)\in\Z^2$.
If $T$ is a reduced scheme finite over $\Spec\kk$, and $\phi:T\to\sqcup_{r,d}\cX_{r,d}$ is a 1-morphism, we set $\cl(\phi):=\sum_{x\in T}\cl(\phi|_x)$. The proof of the proposition is based on the following lemma.

\begin{lemma}\label{lm:Pow}
Let $V$ be a variety over $\kk$, let $\cX_{r,d}$ be stacks of finite type over $\kk$, where $r$ and $d$ run over the set of nonnegative integers not equal to zero simultaneously. Let $\cY_{r,d}$ be the stack parameterizing pairs $(T,\phi)$, where $T$ is a finite subset of closed points of $V$, $\phi:T\to\sqcup_{s,e}\cX_{s,e}$ is a 1-morphism such that $\cl(\phi)=(r,d)$. Then we have
\[
    \left(1+\sum_{r,d}[\cX_{r,d}]w^rz^d\right)^{[V]}=1+\sum_{r,d}[\cY_{r,d}]w^rz^d.
\]
\end{lemma}
\begin{proof}
See~\cite[Sect.~2]{BryanMorrison} in the case $\kk=\C$, and note that it generalizes immediately to any field of characteristic zero (cf.~also~\cite[Prop.~7]{KontsevichSoibelman10} for the case of arbitrary field of characteristic zero).
\end{proof}

Let $\cY_{r,d}$ be the stack parameterizing pairs $(T,\phi)$, where $T$ is a finite subset of $\A^1_\kk$, $\phi:\A^1_\kk\to\sqcup_{r,d}\cE^{\ge0,nilp}_{r,d}$ is a 1-morphism with $\cl(\phi)=(r,d)$. According to Lemma~\ref{lm:Pow} we just need to show that $[\cY_{r,d}]=[\cE^{\ge0}_{r,d}]$. Define the 1-morphism $\cY_{r,d}\to\cE^{\ge0}_{r,d}$ as follows. Consider a pair $(T,\phi)\in\cY_{r,d}$. Write $T=\{x_1,\ldots,x_t\}$, $\phi(x_i)=(E_i,\Psi_i)\in\cE^{\ge0,nilp}_{r_i,d_i}\kk(x_i))$. The 1-morphism sends $(T,\phi)$ to
\begin{equation*}
    \bigoplus_{i=1}^t R_{\kk(x_i)/\kk}(E_i,x_i\Id+\Psi_i).
\end{equation*}
Here $x_i\in\A^1_\kk$ is viewed as an element of $\kk(x_i)$, the functor of restriction of scalars $R_{\kk(x_i)/\kk}$ is the pushforward with respect to the finite morphism $X_{\kk(x_i)}\to X$. One checks that this construction works in families, so we get a required 1-morphism.

According to Corollary~\ref{cor:pointwEqual} it remains to prove the following version of Jordan decomposition.

\begin{lemma}\label{lm:Jordan}
(i) Let $(E,\Psi)\in\cE_{r,d}(K)$ be a bundle with an automorphism. There is a finite set $\{x_1,\ldots,x_t\}$ of closed points of $\A_K^1$, a sequence of pairs $(E_i,\Psi_i)\in\cE_{r_i,d_i}^{nilp}(\kk(x_i))$, and an isomorphism
\[
    (E,\Psi)\xrightarrow{\simeq}\bigoplus_{i=1}^t R_{\kk(x_i)/K}(E_i,x_i\Id+\Psi_i).
\]
(ii) Such a set $\{x_1,\ldots,x_t\}$ is unique and $(E_i,\Psi_i)$ are unique in the following sense: if $(E_i',\Psi_i')$ is another sequence with an isomorphism $(E,\Psi)\xrightarrow{\simeq}\bigoplus_{i=1}^t R_{\kk(x_i)/K}(E_i',x_i\Id+\Psi_i')$, then there are unique isomorphisms $(E_i,\Psi_i)\to(E_i',\Psi_i')$ making the obvious diagram commutative.
\end{lemma}

\begin{proof}
Consider the characteristic polynomial of $\Psi$: $f(x)=\det(x\Id-\Psi)$. The coefficients of this polynomial are global sections of $\cO_{X_K}$ so $f(x)\in K[x]$. Let $T=\{x_1,\ldots,x_t\}\subset\A_K^1$ be the set of roots of $f(x)$. Then $f(x)=\prod_{i=1}^t f_i(x)^{r_i}$, where $f_i(x)$ is an irreducible polynomial of $x_i$ over $K$. The Cayley--Hamilton Theorem (applied at the generic point of $X$) shows that $f(\Psi)=0$, so we have a homomorphism $\pi:K[x]/(f(x))\to\End(E)$, sending the image $\bar x$ of $x$ in $K[x]/(f(x))$ to $\Psi$. Let $\epsilon_i\in K[x]/(f(x))$ be the components of the unity with respect to the decomposition $K[x]/(f(x))=\oplus_{i=1}^t K[x]/(f_i(x)^{r_i})$. Then we have $(E,\Phi)=\bigoplus(E'_i,\Psi'_i)$, where $E'_i:=(\pi(\epsilon_i))(E)$, $\Psi'_i:=\Psi|_{E'_i}$.

Since we have $\Psi'_i=\pi(\bar x\epsilon_i)$, we see that $f_i(\Psi'_i)^{r_i}=0$. By Hensel's Lemma there is $g_i\in x+f_iK[x]$ such that $f_i(g_i)$ is divisible by $f_i^{r_i}$. Set $\Lambda_i=g_i(\Psi'_i)$. Then $f_i(\Lambda_i)=0$ so $\Lambda_i$ gives a $\kk(x_i)$-structure on $E'_i$. Thus $(E'_i,\Psi'_i)=R_{\kk(x_i)/K}(E_i,\Psi_i)$ for a pair $(E_i,\Psi_i)$. It is easy to see that $\Psi_i-x_i\Id$ is nilpotent. We have proved the existence part of the lemma. We leave the uniqueness to the reader.
\end{proof}

Lemma~\ref{lm:Jordan} completes the proof of the proposition.
\end{proof}

\subsection{Example: rank two case of Theorem~\ref{th:conn=higgs}}\label{Sect:Conn=Higgs2} In this section we give a direct proof of Theorem~\ref{th:conn=higgs} for bundles of rank two.

Let us write $\Bun_{2,0}=\cB'\sqcup\cB''$, where $\cB'$ is the open substack of semistable vector bundles, $\cB''$ is the complement.
Set $\cM':=\cM^{ss}_{2,0}\times_{\Bun_{2,0}}\cB'$, $\cC':=\Conn_2\times_{\Bun_{2,0}}\cB'$. Define $\cM''$ and $\cC''$ similarly. Clearly, it is enough to show that $[\cM']=[\cC']$ and $[\cM'']=[\cC'']$.

To show the first equation, note that every Higgs bundle whose underlying vector bundle is semistable, is also semistable. Thus, the fibers of the projection $\cM'\to\cB'$ are vector spaces. The fibers of the projection $\cC'\to\cB'$ are affine spaces modeled over these vector spaces: the crucial point is that every semistable vector bundle admits a connection, which follows from Weil's Theorem.

Showing that $[\cM'']=[\cC'']$ is more involved. Let $\cB'''$ be the moduli stack of pairs $L\subset E$ where $E$ is a vector bundle in $\Bun_{2,0}$, $L$ is a line subbundle of positive degree. Note that $L$ is the unique destabilizing subbundle, so the image of the forgetful 1-morphism $\cB'''\to\Bun_{2,0}$ is exactly $\cB''$ and the fibers of the 1-morphism are points. It follows that the motivic classes of $\cM''':=\cM''\times_{\cB''}\cB'''$ and of $\cM''$ are equal. Similarly, we define $\cC''':=\cC''\times_{\cB''}\cB'''$ and show that $[\cC''']=[\cC'']$. It remains to show that $[\cM''']=[\cC''']$.

Let $\cPic$ denote the stack of line bundles on $X$. We have a 1-morphism $\cB'''\to\cPic\times\cPic$, given by $L\subset E\mapsto(L,E/L)$. This gives 1-morphisms $\cM'''\to\cPic\times\cPic$ and $\cB'''\to\cPic\times\cPic$.

Take a $K$-point of $\cPic\times\cPic$, which is represented by a pair of line bundles $(L_1,L_2)$ on $X_K$. Denote the corresponding fibers of $\cM'''$ and $\cC'''$ by $\cM(L_1,L_2)$ and $\cC(L_1,L_2)$. By Proposition~\ref{pr:pointwise zero}, it suffices to show that $[\cM(L_1,L_2)]=[\cC(L_1,L_2)]$.

Now, let $\cN(L_1,L_2)$ be the stack classifying collections $(L_1\hookrightarrow E\twoheadrightarrow L_2,\Phi)$, where $L_1\hookrightarrow E\twoheadrightarrow L_2$ is a short exact sequence, $\Phi\in\Hom(E,E\otimes\Omega_X)$ is a Higgs field. Set $V:=\Ext^1(L_2,L_1)$, let $V^\vee:=\Hom(L_1,L_2\otimes\Omega_X)$ be the dual vector space. The stack classifying extensions $0\to L_1\to E\to L_2\to0$ is the quotient $V/\Hom(L_2,L_1)$. Thus, we get a forgetful morphism $\pi:\cN(L_1,L_2)\to V/\Hom(L_2,L_1)$.

Similarly, we have a 1-morphism  $\pi^\vee:\cN(L_1,L_2)\to V^\vee$, sending $(L_1\hookrightarrow E\twoheadrightarrow L_2,\Phi)$ to
\[
    L_1\hookrightarrow E\xrightarrow{\Phi}E\otimes\Omega_X\twoheadrightarrow L_2\otimes\Omega_X.
\]
Next, $\cM(L_1,L_2)$ is the open substack of $\cN(L_1,L_2)$ corresponding to semistable Higgs bundles. Since $L_1$ is the only possible destabilizing subbundle, we see that $\cM(L_1,L_2)=\cN(L_1,L_2)-(\pi^\vee)^{-1}(0)$. On the other hand, for $(L_1\hookrightarrow E\twoheadrightarrow L_2,\Phi)\in\cN(L_1,L_2)$, $E$ admits a connection if and only if
\[
    \pi(L_1\hookrightarrow E\twoheadrightarrow L_2,\Phi)\ne0
\]
as follows easily from Weyl's Theorem. Arguing as in the proof of $[\cM']=[\cB']$, we see that $[\cC(L_1,L_2)]=[\cN(L_1,L_2)]-[\pi^{-1}(0)]$.

Thus we are left with showing that $[\pi^{-1}(0)]=[(\pi^\vee)^{-1}(0)]$. Now, $[\pi^{-1}(0)]$ is
\[
    \Hom(L_1\oplus L_2,(L_1\oplus L_2)\otimes\Omega_X)/\Hom(L_2,L_1);
\]
On the other hand, $(\pi^\vee)^{-1}(0)$ classifies Higgs fields preserving $L_1$. Thus we have a 1-morphism $\psi:(\pi^\vee)^{-1}(0)\to\Hom(L_2,L_2\otimes\Omega_X)$. The 1-morphism $\psi\times\pi$ makes $(\pi^\vee)^{-1}(0)$ into an affine bundle on $\Hom(L_2,L_2\otimes\Omega_X)\times V/\Hom(L_2,L_1)$ with the fiber isomorphic to $\Hom(E,L_1\otimes\Omega_X)$. Thus the class of $(\pi^\vee)^{-1}(0)$ is the class of a quotient of a vector space by an action of a vector space. Using the Riemann--Roch theorem, Serre duality, an exact sequence for $\Hom$, and the fact that $\Hom(L_1,L_2)=0$, one calculates the dimensions of these stacks and sees that they are the same. This completes the proof.

\section{Motivic classes of Borel reductions}\label{sect:Borel}
The goal of this section is to prove Theorems~\ref{th:Harder1} and~\ref{th:ResHarder2} (we will see that in fact these theorems are equivalent). Theorem~\ref{th:ResHarder2} is the motivic analogue of Harder's residue formula~\cite[Thm.~2.2.3]{HarderAnnals} for $\GL_n$. A slightly different form of Theorem~\ref{th:Harder1} appeared in Section~\ref{sect:IntroHarder} as Theorem~\ref{th:IntroHarder}.

In the current section $\kk$ is a field of any characteristic and $X$ is a smooth geometrically connected projective curve over $\kk$. Recall that when $\kk$ is a field of characteristic zero, we set $X^{(i)}=X^i/\Symm_i$. In this section, we let $X^{(i)}$ denote the Hilbert scheme of degree $i$ finite subschemes of $X$. When $\kk$ has characteristic zero, this definition agrees with the previous one. We assume that there is a divisor $\divisor$ on $X$ defined over $\kk$ such that $\deg\divisor=1$. We denote by $\Jac$ the Jacobian variety of~$X$. As before, $K$ denotes an arbitrary extension of $\kk$.

\subsection{Limits of motivic classes of Borel reductions}\label{Sect:ProofBorel}
\begin{lemma}\label{lm:neutral}
For all $d\in\Z$ the moduli stack $\cPic^d$ of degree $d$ line bundles on $X$ is the neutral $\gm$-gerbe over $\Jac$. That is, $\cPic^d\simeq\Jac\times\mathbf{B}\gm$, where $\mathbf{B}\gm$ is the classifying stack of $\gm$.
\end{lemma}
\begin{proof}
First of all, tensorisation with $\cO_X(d\divisor)$ gives an isomorphism $\cPic^0\to\cPic^d$.

Let us write $\divisor=\divisor_1-\divisor_2$, where $\divisor_i$ are effective divisors on $X$, we view $\divisor_i$ as a closed subscheme of~$X$ (not necessarily reduced). Let $S$ be a test scheme. By abuse of notation we denote two projections $D_i\times S\to S$ by $p$.

For a scheme $S$ we denote by $Pic(S)$ the abelian group of isomorphism classes of line bundles on $S$. The Picard variety $\Pic(X)$ represents the functor $S\mapsto Pic(X\times S)/Pic(S)$; note that $\Jac$ is just its neutral component.

For a line bundle $\ell$ on $S\times X$, let $\det(\ell|_\divisor)$ denote the line bundle on $S$ given by
\[
    \wedge^{\deg \divisor_1}(p_*(\ell|_{S\times \divisor_1}))\otimes\bigl(\wedge^{\deg \divisor_2}(p_*(\ell|_{S\times \divisor_2}))\bigr)^{-1}.
\]
It is easy to see that $\det(\bullet|_\divisor):Pic(X\times S)\to Pic(S)$ is left inverse to the pullback functor. Using this fact, it is easy to see that $\Pic(X)$ represents the functor, sending $S$ to the set of pairs $(\ell,s)$, where $\ell$ is a line bundle on $S\times X$, $s$ is a trivialization of $\det(\ell|_\divisor)$ (cf.~\cite[Lemma~2.9]{KleimanFGA}). Thus we have a universal line bundle $L$ on $\Pic(X)\times X$, whose restriction to $\Jac\times X$ trivializes the $\gm$-gerbe.
\end{proof}

Recall that in Section~\ref{sect:BorelRed} we defined the stack $\Bun_{r,d_1,\ldots,d_r}$ classifying vector bundles on $X$ with Borel reductions of degree $(d_1,\ldots,d_r)$. We view  $\Bun_{r,d_1,\ldots,d_r}$ as a stack over $\Bun_{r,d_1+\ldots+d_r}$ via the projection $(E_1\subset\ldots\subset E_r)\mapsto E_r$. This projection is schematic and of finite type (for the proof, embed the fiber of this projection into a product of Quot schemes). Set
\[
    \Bun_{r,d_1,\ldots,d_r}^{\ge\tau}:=\Bun_{r,d_1,\ldots,d_r}\times_{\Bun_{r,d_1+\ldots+d_r}}\Bun_{r,d_1+\ldots+d_r}^{\ge\tau}.
\]

\begin{theorem}\label{th:Harder1}
For any $\tau$, any $r\in\Z_{>0}$, and $d\in\Z$ we have in $\cMot(\Bun_{r,d}^{\ge\tau})$
\[
\lim_{d_1\to-\infty}\ldots\lim_{d_{r-1}\to-\infty}
\frac{[\Bun_{r,d_1,\ldots,d_{r-1},d-d_1-\ldots-d_{r-1}}^{\ge\tau}\to\Bun_{r,d}^{\ge\tau}]}
{\bL^{-(2r-2)d_1-(2r-4)d_2-\ldots-2d_{r-1}}}=
\frac{
    \bL^{(r-1)\left(d+(1-g)\frac{r+2}2\right)}[\Jac]^{r-1}}
    {(\bL-1)^{r-1}\prod_{i=2}^r\!\zeta_X(\bL^{-i})}\;\mathbf1_{\Bun_{r,d}^{\ge\tau}}.
\]
\end{theorem}

This theorem will be proved in Section~\ref{sect:ProofHarder}.

\begin{proof}[Derivation of Theorem~\ref{th:IntroHarder} from Theorem~\ref{th:Harder1}]
It is enough to show that the statement holds after restriction to any finite type open substack of $\Bun_{r,d}$ (see the discussion of topology on $\cMot$ in Section~\ref{sect:MotFun1}). It remains to use the previous theorem and Lemma~\ref{lm:Bun+}(iv).
\end{proof}

\subsection{Eisenstein series}\label{sect:Eisenstein} We would like to reformulate the above theorem using residues. Define the Eisenstein series
\[
    E_{r,d}^{\ge\tau}(z_1,\ldots,z_r):=
    \sum_{\substack{d_1,\ldots,d_r\in\Z\\ d_1+\ldots+d_r=d}}\bL^{(r-1)d_1+(r-2)d_2+\ldots+d_{r-1}}
    [\Bun_{r,d_1,\ldots,d_r}^{\ge\tau}\to\Bun_{r,d}^{\ge\tau}]z_1^{d_1}\ldots z_r^{d_r}.
\]

\begin{remark}
The reason we restrict to a single degree $d$ and to HN-nonnegative vector bundles is that we want to work with the group of motivic functions on a finite type stack $\Bun_{r,d}^{\ge\tau}$, cf.~Lemma~\ref{lm:Bun+}(iv). In Section~\ref{sect:Hall} we will have to work over non-finite type stacks.
\end{remark}

For an abelian group $M$, consider the following group of formal Laurent series
\begin{equation}\label{eq:ring}
    M\left(\!\left(z_1,\frac{z_2}{z_1},\ldots,\frac{z_r}{z_{r-1}}\right)\!\right):=
    M[z_1^{\pm1},\ldots, z_r^{\pm1}]\otimes_{\Z[z_1,\frac{z_2}{z_1},\ldots,\frac{z_r}{z_{r-1}}]}
    \Z\left[\left[
    z_1,\frac{z_2}{z_1},\ldots,\frac{z_r}{z_{r-1}}
    \right]\right].
\end{equation}

\begin{lemma}\label{lm:LaurSer}
We have
\[
        E_{r,d}^{\ge\tau}(z_1,\ldots,z_r)\in \cMot(\Bun_{r,d}^{\ge\tau})\left(\!\left(z_1,\frac{z_2}{z_1},\ldots,\frac{z_r}{z_{r-1}}\right)\!\right).
\]
\end{lemma}

\begin{proof}
We note that by Lemma~\ref{lm:Bun+}(ii) we have for each point $(E_0\subset\ldots\subset E_r)$ of $\Bun_{r,d_1,\ldots,d_r}^{\ge\tau}$ and $i=0,\ldots,r-1$
\begin{equation*}
    \frac{\deg(E_r/E_i)}{\rk(E_r/E_i)}=\frac{d_{i+1}+\ldots+d_r}{r-i}\ge\tau.
\end{equation*}

We can re-write
\begin{multline*}
    E_{r,d}^{\ge\tau}(z_1,\ldots,z_r)=\\ \sum_{\substack{d_1,\ldots,d_r\in\Z\\ d_1+\ldots+d_r=d}}\bL^{(r-1)d_1+(r-2)d_2+\ldots+d_{r-1}}[\Bun_{r,d_1,\ldots,d_r}^{\ge\tau}\to\Bun_{r,d}^{\ge\tau}]
    z_1^{d_1+\ldots+d_r}\left(\frac{z_2}{z_1}\right)^{d_2+\ldots+d_r}\ldots
    \left(\frac{z_r}{z_{r-1}}\right)^{d_r}.
\end{multline*}
The statement follows.
\end{proof}

In the next section we will give precise definitions of residues of power series and prove the following theorem. (In fact, we will see that this theorem is just a reformulation of Theorem~\ref{th:Harder1}.)

\begin{theorem}\label{th:ResHarder2}
We have
\[
    \res_{\frac{z_2}{z_1}=\frac{z_3}{z_2}=\ldots=\frac{z_r}{z_{r-1}}
    =\bL^{-1}}E_{r,d}^{\ge\tau}(z_1,\ldots,z_r)\prod_{i=2}^r\frac{dz_i}{z_i}=
    z_1^d\frac{
    \bL^{\frac{(1-g)(r-1)(r+2)}2}[\Jac]^{r-1}}
    {(\bL-1)^{r-1}\prod_{i=2}^r\!\zeta_X(\bL^{-i})}\;\mathbf1_{\Bun_{r,d}^{\ge\tau}}.
\]
\end{theorem}

\subsection{Residues of formal series}\label{sect:res}
Let $M$ be a topological abelian group (in our applications we will take $M=\cMot(\Bun_{r,d}^{\ge\tau})$, $M=\cMot(\Bun_{r,d}$) etc). Let $A(z)\,dz\in M((z))\,dz$. Plugging in $z=x$ into the product $(x-z)A(z)$ we get an infinite series. If it converges in $M$, we define the residue of as the sum of this series:\footnote{Usually residue is defined with opposite sign. We follow conventions of~\cite{HarderAnnals} and~\cite{SchiffmannIndecomposable}. This simplifies our formulas.}
\[
    \res_{z=x}A(z)\,dz:=((x-z)A(z))|_{z=x}.
\]
\begin{remark}\label{rm:RationalRes}
Assume that $R(z)$ is a rational function with coefficients in $\cMot(\kk)$, that is, an element of the total ring of fractions of $\cMot(\kk)[z]$. We say that it has at most first order pole at $x\in\cMot(\kk)$ if it can be written as
\[
    \frac{P(z)}{(x-z)Q(z)},
\]
where $Q(x)$ is not a zero divisor in $\cMot(\kk)$. In this case we can define
\[
    \res_{z=x}R(z)\,dz=\frac{P(x)}{Q(x)}.
\]
On the other hand, if $Q(x)$ is invertible, we can expand $R(z)$ in powers of $z$. The residue of the corresponding series may or may not exist, but if it exists, it is equal to the residue of the rational function. Similar considerations apply to the case of many variables considered below.
\end{remark}

\begin{lemma}\label{lm:res1dim}
If $A(z)=\sum_{-\infty}^\infty A_dz^d$ with $A_d=0$ for $d\ll0$, then
\[
    \res_{z=x}A(z)\,dz=\lim_{d\to\infty}A_dx^{d+1}.
\]
Moreover, the residue exists if and only if the limit exists.
\end{lemma}
\begin{proof}
    The $d$-th partial sums of the series $((x-z)A(z))|_{z=x}$ is
\[
    \sum_{-\infty}^d(A_ix^{i+1}-A_{i-1}x^i)=A_dx^{d+1}
\]
(note that the infinite sum has only finitely many non-zero terms because $A_d=0$ for $d\ll0$). The statement follows.
\end{proof}

Let $\cX$ be a stack of finite type. Note that an infinite series with coefficients in $\cMot(\cX)$ converges if and only if its terms tend to zero. Using this fact, it is not difficult to prove the following statement.
\begin{lemma}\label{lm:CauchyProd}
Assume that $A(z)\in\cMot(\kk)((z))$ converges at a certain $x\in\cMot(\kk)$ and that for $B(z)\in\cMot(\cX)((z))$ the residue $\res_{z=x}B(z)\,dz$ exists. Then
\[
    \res_{z=x}A(z)B(z)\,dz=A(x)\res_{z=x}B(z)\,dz.
\]
\end{lemma}

Now consider the case of many variables. For a series $A(z_1,\ldots,z_r)$ in~\eqref{eq:ring} and $x\in M$ define
\[
    \res_{\frac{z_r}{z_{r-1}}=x}A(z_1,\ldots,z_r)\,\frac{dz_r}{z_r}=
    \left.\left(1-\frac{z_r}{xz_{r-1}}\right)A(z_1,\ldots,z_r)\right|_{z_r=xz_{r-1}}.
\]
Note that, by definition, this residue (if it exists) is a series in variables $z_1,\ldots,z_{r-1}$, and, moreover, we have
\[
    \res_{\frac{z_r}{z_{r-1}}=x}A(z_1,\ldots,z_r)\in M\left(\!\left(z_1,\frac{z_2}{z_1},\ldots,\frac{z_{r-1}}{z_{r-2}}\right)\!\right).
\]
Thus we can define the iterated residue
\begin{multline}\label{}
    \res_{\frac{z_2}{z_1}=x_1,\ldots,\frac{z_r}{z_{r-1}}=x_{r-1}}A(z_1,\ldots,z_r)\prod_{i=2}^r\frac{dz_i}{z_i}:=\\
    \res_{\frac{z_2}{z_1}=x_1}
    \left(
    \ldots
    \left(
    \res_{\frac{z_{r-1}}{z_{r-2}}=x_{r-2}}\left(\res_{\frac{z_r}{z_{r-1}}=x_{r-1}}A(z_1,\ldots,z_r)\frac{dz_r}{z_r}\right)
    \frac{dz_{r-1}}{z_{r-1}}\right)
    \ldots
    \frac{dz_2}{z_2}\right).
\end{multline}
We see that this iterated residue is a Laurent series in one variable $z_1$.

\begin{lemma}\label{lm:HighDimRes}
Let $A(z_1,\ldots,z_r)=\sum_{d_1,\ldots,d_r}A_{d_1\ldots d_r}z_1^{d_1}\ldots z_r^{d_r}$ be a series in~\eqref{eq:ring} and let
\[
    \res_{\frac{z_2}{z_1}=x_1,\ldots,\frac{z_r}{z_{r-1}}=x_{r-1}}A(z_1,\ldots,z_r)\prod_{i=2}^r\frac{dz_i}{z_i}=
    \sum_i B_dz_1^d.
\]
Then
\[
    B_d=\lim_{d_1\to-\infty}\ldots\lim_{d_{r-1}\to-\infty} A_{d_1,\ldots,d_{r-1},d-d_1-\ldots-d_{r-1}}
    x_1^{d-d_1}x_2^{d-d_1-d_2}\ldots x_{r-1}^{d-d_1-\ldots-d_{r-1}}.
\]
Moreover, the iterated residue exists if and only if the limits exist.
\end{lemma}
\begin{proof}
We proceed by induction on $r$. If $r=1$ the statement holds trivially. Assuming that the statement holds for $r-1$, we calculate as in the proof of Lemma~\ref{lm:res1dim}
\begin{multline*}
    \res_{\frac{z_r}{z_{r-1}}=x_{r-1}}A(z_1,\ldots,z_r)\frac{dz_r}{z_r}=\\
    \sum_{d_1,\ldots,d_r}(A_{d_1,\ldots,d_r}x_{r-1}^{d_r}-A_{d_1,\ldots,d_{r-2},d_{r-1}+1,d_r-1}x_{r-1}^{d_r-1})
    z_1^{d_1}\ldots z_{r-2}^{d_{r-2}}z_{r-1}^{d_{r-1}+d_r}.
\end{multline*}
Let us perform a change of variables: $j_1=d_1$, \ldots, $j_{r-2}=d_{r-2}$, $j_{r-1}=d_{r-1}+d_r$, $j=d_{r-1}$. Then we get
\begin{multline*}
    \sum_{d_1,\ldots,d_r}(A_{d_1,\ldots,d_r}x_{r-1}^{d_r}-A_{d_1,\ldots,d_{r-2},d_{r-1}+1,d_r-1}x_{r-1}^{d_r-1})
    z_1^{d_1}\ldots z_{r-2}^{d_{r-2}}z_{r-1}^{d_{r-1}+i_r}=\\
    \sum_{j_1,\ldots,j_{r-1}}\left(
    \sum_j\left(A_{j_1,\ldots,j_{r-2},j,j_{r-1}-j}x_{r-1}^{j_{r-1}-j}-A_{j_1,\ldots,j_{r-2},j+1,j_{r-1}-j-1}x_{r-1}^{j_{r-1}-j-1}\right)
    \right)
    z_1^{j_1}\ldots z_{r-1}^{j_{r-1}}=\\
    \sum_{j_1,\ldots,j_{r-1}}\left(
    \lim_{j\to-\infty}A_{j_1,\ldots,j_{r-2},j,j_{r-1}-j}x_{r-1}^{j_{r-1}-j}
    \right)
    z_1^{j_1}\ldots z_{r-1}^{j_{r-1}}.
\end{multline*}
Now by induction hypothesis we have
\begin{multline*}
        \res_{\frac{z_2}{z_1}=x_1,\ldots,\frac{z_r}{z_{r-1}}=x_{r-1}}A(z_1,\ldots,z_r)\prod_{i=2}^r\frac{dz_i}{z_i}=\\
        \res_{\frac{z_2}{z_1}=x_1,\ldots,\frac{z_{r-1}}{z_{r-2}}=x_{r-2}}
        \left(
        \sum_{j_1,\ldots,j_{r-1}}\left(
        \lim_{j\to-\infty}A_{j_1,\ldots,j_{r-2},j,j_{r-1}-j}x_{r-1}^{j_{r-1}-j}
        \right)
        z_1^{j_1}\ldots z_{r-1}^{j_{r-1}}
        \right)\prod_{i=2}^{r-1}\frac{dz_i}{z_i}=\\
        \sum_dz_1^d
        \left(
            \lim_{j_1,\ldots,j_{r-2},j\to-\infty}A_{j_1,\ldots,j_{r-2},j,d-j_1-\ldots-j_{r-2}-j}
            x_1^{d-j_1}\ldots x_{r-2}^{d-j_1-\ldots-j_{r-2}}x_{r-1}^{d-j_1-\ldots-j_{r-2}-j}
        \right).
\end{multline*}
\end{proof}

\subsection{Derivation of Theorem~\ref{th:ResHarder2} from Theorem~\ref{th:Harder1}}
Let us write
\[
        \res_{\frac{z_2}{z_1}=\frac{z_3}{z_2}=\ldots=\frac{z_r}{z_{r-1}}=        \bL^{-1}}E_{r,d}^{\ge\tau}(z_1,\ldots,z_r)\prod_{i=2}^r\frac{dz_i}{z_i}=
        \sum_n B_n z_1^n.
\]
It follows easily from Lemma~\ref{lm:HighDimRes} that $B_n=0$ if $n\ne d$. On the other hand, by the same lemma, we have
\begin{multline}
B_d=\lim_{d_1\to-\infty}\ldots\lim_{d_{r-1}\to-\infty}
\frac{\bL^{(r-1)d_1+\ldots+d_{r-1}}[\Bun^{\ge\tau}_{r,d_1,\ldots,d_{r-1},d-d_1-\ldots-d_{r-1}}\to\Bun_{r,d}^{\ge\tau}]}
{\bL^{d-d_1}\bL^{d-d_1-d_2}\ldots\bL^{d-d_1-\ldots-d_{r-1}}}=\\    \bL^{(1-r)d}\lim_{d_1\to-\infty}\ldots\lim_{d_{r-1}\to-\infty}
\frac{[\Bun^{\ge\tau}_{r,d_1,\ldots,d_{r-1},d-d_1-\ldots-d_{r-1}}\to\Bun_{r,d}^{\ge\tau}]}{\bL^{-(2r-2)d_1-(2r-4)d_2-\ldots-2d_{r-1}}}=\\
\bL^{(1-r)d}\cdot\frac{
    \bL^{(r-1)\left(d+(1-g)\frac{r+2}2\right)}[\Jac]^{r-1}}
    {(\bL-1)^{r-1}\prod_{i=2}^r\!\zeta_X(\bL^{-i})}\;\mathbf1_{\Bun_{r,d}^{\ge\tau}}
    =\frac{
    \bL^{(r-1)\left((1-g)\frac{r+2}2\right)}[\Jac]^{r-1}}
    {(\bL-1)^{r-1}\prod_{i=2}^r\!\zeta_X(\bL^{-i})}\;\mathbf1_{\Bun_{r,d}^{\ge\tau}}.
\end{multline}

\subsection{Stacks of partial flags}
Before we prove Theorem~\ref{th:Harder1}, we need some preliminaries. We introduce the stack $\Bun_{r,d_1,\ldots,d_l}^{\ge\tau}$ classifying collections $0=E_0\subset E_1\subset\ldots\subset E_l$, where $E_l$ is vector bundle in $\Bun_{r,d_1+\ldots+d_l}^{\ge\tau}$, $E_i$ is a vector subbundle of rank $i$ for $i=1,\ldots,l-1$, and we have $\deg(E_i/E_{i-1})=d_i$ for $i=1,\ldots,l$. We study this stack for $l=2$ first.

\begin{proposition}\label{pr:l=2}
We have in $\cMot(\Bun_{r,d}^{\ge\tau})$:
\[\
    \lim_{d_1\to-\infty}\frac{[\Bun_{r,d_1,d-d_1}^{\ge\tau}\to\Bun_{r,d}^{\ge\tau}]}{\bL^{-rd_1}}=
            \frac{\bL^{d+r(1-g)}[\Jac]}{(\bL-1)\zeta_X(\bL^{-r})}\mathbf1_{\Bun_{r,d}^{\ge\tau}}.
\]
\end{proposition}

We will prove this proposition in Section~\ref{sect:pr:l=2}. We first need to introduce more stacks. Let $\Lau_{r,d_1,d_2}^{\ge\tau}$ classify collections $0=E_0\subset E_1\subset E_2$, where $E_2$ is a vector bundle in $\Bun_{r,d_1+d_2}^{\ge\tau}$, $E_1$ is a subsheaf of rank one and degree $d_1$. Note that $E_1$ is a line bundle but $E_2/E_1$ might have torsion. We view $\Lau_{r,d_1,d_2}^{\ge\tau}$ as a stack over $\Bun_{r,d_1+d_2}^{\ge\tau}$ via the obvious projection.

\begin{remark}
The stacks $\Lau_{r,d_1,d_2}^{\ge\tau}$ are Laumon's relative compactifications of  $\Bun_{r,d_1,d_2}^{\ge\tau}\to\Bun_{r,d_1+d_2}^{\ge\tau}$. Thus the notation. The reason we need these stacks is that they are simpler than $\Bun_{r,d_1,d_2}^{\ge\tau}$, when $d_1$ is small, see Lemma~\ref{lm:Laumon_l=2} below. On the other hand, their relation with $\Bun_{r,d_1,d_2}^{\ge\tau}$ is also quite simple as we see momentarily.
\end{remark}

\begin{lemma}\label{lm:LauStrat}
We have in $\Mot(\Bun_{r,d_1+d_2}^{\ge\tau})$
\begin{equation*}
    [\Lau_{r,d_1,d_2}^{\ge\tau}\to\Bun_{r,d_1+d_2}^{\ge\tau}]=
    \sum_{i\ge0}[X^{(i)}\times\Bun_{r,d_1+i,d_2-i}^{\ge\tau}\to\Bun_{r,d_1+d_2}^{\ge\tau}].
\end{equation*}
\end{lemma}
\begin{proof}
Note first, that the sum in the RHS is finite. Indeed, $\Bun_{r,d_1+i,d_2-i}^{\ge\tau}$ is empty, when $i$ is large enough. Recall that $X^{(i)}$ parameterizes length $i$ subschemes of $X$. We have a 1-morphism
\[
    \phi:\sqcup_{i\ge0}(X^{(i)}\times\Bun_{r,d_1+i,d_2-i}^{\ge\tau})\to
    \Lau_{r,d_1,d_2}^{\ge\tau}
\]
sending $(\divisor,E_1\subset E_2)$ to $(E_1(-\divisor),E_2)$. We claim that this 1-morphism induces an isomorphism on $K$-points for any field $K$. Indeed, take $(E_1\subset E_2)\in\Lau_{r,d_1,d_2}^{\ge\tau}(K)$. The coherent sheaf $E_2/E_1$ can be written as $T\oplus E$, where $T$ is a uniquely defined torsion sheaf and $E$ is a vector bundle. Let $E_1'$ be the inverse image of $T$ under the projection $E_2\to E_2/E_1$. Then $E_1=E_1'(-\divisor)$ for an effective divisor $\divisor\subset X$ (because $E_1'/E_1\approx T$ is torsion). Now $(E_1\subset E_2)=\phi(\divisor,E_1'\subset E_2)$ and it is easy to see that this is the only $K$-point mapping to $(E_1\subset E_2)$. It remains to use Corollary~\ref{cor:pointwEqual}.
\end{proof}

\begin{lemma}\label{lm:Laumon_l=2}
If $d_1<2-2g+\tau$, then we have in $\Mot(\Bun_{r,d_1+d_2}^{\ge\tau})$
\[
    [\Lau^{\ge\tau}_{r,d_1,d_2}\to\Bun_{r,d_1+d_2}^{\ge\tau}]=
    \frac{\bL^{d_1+d_2+r(-d_1-g+1)}-1}{\bL-1}[\Jac]\mathbf1_{\Bun_{r,d_1+d_2}^{\ge\tau}}.
\]
\end{lemma}

\begin{proof}
Recall that the stack $\cPic_{d_1}$ classifies degree $d_1$ line bundles on $X$. Let $\cL$ be the universal line bundle on $\cPic_{d_1}\times X$. Denote by $\cE$ the universal vector bundle on $\Bun_{r,d_1+d_2}^{\ge\tau}\times X$. Denote by $p_{ij}$ the projections from
\[
    \cPic_{d_1}\times\Bun_{r,d_1+d_2}^{\ge\tau}\times X
\]
to the products of the $i$-th and the $j$-th factors. Set
\[
    \cF:=\HOM(p_{13}^*\cL,p_{23}^*\cE)\quad
    \text{ and }\cV:=(p_{12})_*\cF,
\]
where $\HOM$ stands for the sheaf of homomorphisms. Note that $\cV$ is a coherent sheaf because $p_{12}$ is a proper 1-morphism.

\emph{Claim.} If $d_1<2-2g+\tau$, then the coherent sheaf $\cV$ is locally free of rank $d_1+d_2+r(-d_1-g+1)$.
\begin{proof}[Proof of the claim]
Let $\xi=(\ell,E)$ be a $K$-point of $\cPic_{d_1}\times\Bun^{\ge\tau}_{r,d_1+d_2}$ so that $\ell$ is a line bundle on $X_K$, $E$ is a vector bundle on $X_K$. The fiber of $\cF$ over $\xi\times X=X_K$ is $\cF_\xi=\HOM(\ell,E)$. According to~\cite[Sect.~5, Cor.~2]{MumfordAbelian} we only need to show that $h^0(X_K,\cF_\xi)=d_1+d_2+r(-d_1-g+1)$ (and, in particular, this dimension does not depend on $\xi$).

First of all, we claim that $H^1(X_K,\cF_\xi)=0$. Indeed, by Serre duality the vector space $H^1(X_K,\cF_\xi)$ is dual to $\Hom(E,\Omega_{X_K}\otimes\ell)$. The latter space is zero by Lemma~\ref{lm:Bun+}(ii). Now by Riemann--Roch we have
\[
    h^0(X_K,\cF_\xi)=h^0(X_K,E\otimes\ell^{-1})=d_1+d_2+r(-d_1-g+1).
\]
The claim is proved.
\end{proof}

Consider the complement of the zero section in the total space of the vector bundle $\cV$. It is clear from the construction that this complement classifies triples $(\ell,E,s)$, where $\ell$ is a degree $d_1$ line bundle on $X$, $E$ is a vector bundle on $X$ of rank $r$ and degree $d_1+d_2$, $s:\ell\to E$ is a non-zero (=injective) morphism. Now it is easy to see that this complement is isomorphic to $\Lau_{r,d_1,d_2}^{\ge\tau}$. Thus, by the previous claim and Lemma~\ref{lm:neutral} we have in $\Mot(\Bun_{r,d_1+d_2}^{\ge\tau})$:
\begin{multline*}
    [\Lau_{r,d_1,d_2}^{\ge\tau}\to\Bun_{r,d_1+d_2}^{\ge\tau}]=\\
    (\bL^{d_1+d_2+r(-d_1-g+1)}-1)[\cPic_{d_1}]\mathbf1_{\Bun_{r,d_1+d_2}^{\ge\tau}}=
\frac{\bL^{d_1+d_2+r(-d_1-g+1)}-1}{\bL-1}[\Jac]\mathbf1_{\Bun_{r,d_1+d_2}^{\ge\tau}}.
\end{multline*}
\end{proof}

\subsection{Proof of Proposition~\ref{pr:l=2}}\label{sect:pr:l=2}
Consider the generating series
\[
    E(z):=\sum_{d_1}[\Bun_{r,-d_1,d+d_1}^{\ge\tau}\to\Bun_{r,d}^{\ge\tau}]z^{d_1}\in\cMot(\Bun_{r,d}^{\ge\tau})((z))
\]
and
\[
    \tilde E(z):=\sum_{d_1}[\Lau_{r,-d_1,d+d_1}^{\ge\tau}\to\Bun_{r,d}^{\ge\tau}]z^{d_1}\in\cMot(\Bun_{r,d}^{\ge\tau})((z)).
\]
It follows from Lemma~\ref{lm:LauStrat} that
\begin{equation*}
    \tilde E(z)=\zeta_X(z) E(z).
\end{equation*}
Now we calculate, using Lemma~\ref{lm:res1dim} twice and Lemma~\ref{lm:CauchyProd}.
\begin{multline*}
    \lim_{d_1\to-\infty}\frac{[\Bun_{r,d_1,d-d_1}^{\ge\tau}\to\Bun_{r,d}^{\ge\tau}]}{\bL^{-rd_1}}=
    \bL^r\res_{z=\bL^{-r}}E(z)\,dz=\\
    \frac{\bL^r}{\zeta_X(\bL^{-r})}\res_{z=\bL^{-r}}\tilde E(z)\,dz=
    \frac1{\zeta_X(\bL^{-r})}\lim_{d_1\to-\infty}\frac{[\Lau_{r,d_1,d-d_1}^{\ge\tau}\to\Bun_{r,d}^{\ge\tau}]}{\bL^{-rd_1}}.
\end{multline*}
We also used that $\zeta_X(\bL^{-r})$ converges and is invertible in $\cMot(\kk)$ for any $r\ge2$ (see Lemma~\ref{lm:zetaX}). Now by Lemma~\ref{lm:Laumon_l=2} we have
\begin{multline*}
\frac1{\zeta_X(\bL^{-r})}\lim_{d_1\to-\infty}\frac{[\Lau_{r,d_1,d-d_1}^{\ge\tau}\to\Bun_{r,d}^{\ge\tau}]}{\bL^{-rd_1}}=\\
    \frac{[\Jac]}{(\bL-1)\zeta_X(\bL^{-r})}
    \lim_{d_1\to-\infty}\frac{\bL^{d+r(-d_1-g+1)}-1}{\bL^{-rd_1}}
    \mathbf1_{\Bun_{r,d}^{\ge\tau}}=
        \frac{\bL^{d+r(1-g)}[\Jac]}{(\bL-1)\zeta_X(\bL^{-r})}
    \mathbf1_{\Bun_{r,d}^{\ge\tau}}.
\end{multline*}
\qed

\begin{remark}
It is an easy consequence of the above calculations that $\tilde E(z)$, $E(z)$, and $E_{2,d}^{\ge\tau}(z_1,z_2)$ are expansions of rational functions. We do not know if $E_{r,d}^{\ge\tau}$ is an expansion of a rational function for $r\ge3$.
\end{remark}

\subsection{Proof of Theorem~\ref{th:Harder1}}\label{sect:ProofHarder}
We will prove for $2\le l\le r$ that
\[
   \lim_{d_1\to-\infty}\ldots\lim_{d_{l-1}\to-\infty}
    \frac{[\Bun_{r,d_1,\ldots,d_{l-1},d-d_1-\ldots-d_{l-1}}^{\ge\tau}\to\Bun_{r,d}^{\ge\tau}]}
    {\bL^{-(r+l-2)d_1-(r+l-4)d_2-\ldots-(r-l+2)d_{l-1}}}=\frac{
    \bL^{(l-1)(d+(1-g)\frac{2r-l+2}2)}[\Jac]^{l-1}}
    {(\bL-1)^{l-1}\prod_{i=r-l+2}^r\!\zeta_X(\bL^{-i})}\;\mathbf1_{\Bun_{r,d}^{\ge\tau}}.
\]
Our theorem is equivalent to this statement with $l=r$. We use induction on $l$. For $l=2$ this is Proposition~\ref{pr:l=2} above. Assume that the formula is proved for $l-1$.

\begin{lemma}
We have
\begin{equation*}
    \Bun_{r,d_1,\ldots,d_l}^{\ge\tau}\simeq\Bun_{r,d_1,\ldots,d_{l-2},d_{l-1}+d_l}^{\ge\tau}\times_{\Bun_{r-l+2,d_{l-1}+d_l}^{\ge\tau}}
    \Bun_{r-l+2,d_{l-1},d_l}^{\ge\tau}.
\end{equation*}
\end{lemma}
\begin{proof}
The isomorphism sends $\cE_1\subset\ldots\subset\cE_l$ to the pair
\[
    (\cE_1\subset\ldots\subset\cE_{l-2}\subset\cE_l,\cE_{l-1}/\cE_{l-2}\subset\cE_l/\cE_{l-2}).
\]
\end{proof}

Let us return to the proof of the theorem. First, we fix $d,d_1,\ldots,d_{l-2}$. Set $d':=d-d_1-\ldots-d_{l-2}$.  Let
$f:\Bun_{r,d_1,\ldots,d_{l-2},d'}^{\ge\tau}\to\Bun_{r-l+2,d'}^{\ge\tau}$ and
$g:\Bun_{r,d_1,\ldots,d_{l-2},d'}^{\ge\tau}\to\Bun_{r,d}^{\ge\tau}$ be the projections. These projections together with the 1-morphisms of the previous lemma fit into the diagram:
\[
\begin{CD}
\Bun_{r,d_1,\ldots,d_{l-1},d'-d_{l-1}}^{\ge\tau} @>>>
\Bun_{r,d_1,\ldots,d_{l-2},d'}^{\ge\tau} @>g>>\Bun_{r,d}^{\ge\tau}\\
@VVV @VVfV \\
\Bun_{r-l+2,d_{l-1},d'-d_{l-1}}^{\ge\tau} @>>> \Bun_{r-l+2,d'}^{\ge\tau}.
\end{CD}
\]
Using the above lemma and Proposition~\ref{pr:l=2}, we calculate
\begin{multline}
    \lim_{d_{l-1}\to-\infty}
    \frac{[\Bun_{r,d_1,\ldots,d_{l-1},d-d_1-\ldots-d_{l-1}}^{\ge\tau}\to\Bun_{r,d}^{\ge\tau}]}
    {\bL^{-(r-l+2)d_{l-1}}}=\\
    \lim_{d_{l-1}\to-\infty}\frac{g_!f^*
    [\Bun_{r-l+2,d_{l-1},d'-d_{l-1}}^{\ge\tau}\to\Bun_{r-l+2,d'}^{\ge\tau}]}{\bL^{-(r-l+2)d_{l-1}}}=\\
    g_!f^*\left(
    \lim_{d_{l-1}\to-\infty}\frac{[\Bun_{r-l+2,d_{l-1},d'-d_{l-1}}^{\ge\tau}\to\Bun_{r-l+2,d'}^{\ge\tau}]}{\bL^{-(r-l+2)d_{l-1}}}
    \right)=\\
    g_!f^*\left(
    \frac{\bL^{d'+(r-l+2)(1-g)}[\Jac]}{(\bL-1)\zeta_X(\bL^{-(r-l+2)})}\mathbf1_{\Bun_{r-l+2,d'}^{\ge\tau}}
    \right)=
    \frac{\bL^{d'+(r-l+2)(1-g)}[\Jac]}{(\bL-1)\zeta_X(\bL^{-(r-l+2)})}[\Bun_{r,d_1,\ldots,d_{l-2},d'}^{\ge\tau}\to\Bun_{r,d}^{\ge\tau}].
\end{multline}

It remains to use the induction hypothesis:
\begin{multline}
\lim_{d_1\to-\infty}\ldots\lim_{d_{l-1}\to-\infty}
\frac{[\Bun_{r,d_1,\ldots,d_{l-1},d-d_1-\ldots-d_{l-1}}^{\ge\tau}\to\Bun_{r,d}^{\ge\tau}]}
{\bL^{-(r+l-2)d_1-(r+l-4)d_2-\ldots-(r-l+2)d_{l-1}}}=\\
\lim_{d_1\to-\infty}\ldots\lim_{d_{l-2}\to-\infty}
\frac{\bL^{d'+(r-l+2)(1-g)}[\Jac][\Bun_{r,d_1,\ldots,d_{l-2},d'}^{\ge\tau}\to\Bun_{r,d}^{\ge\tau}]}
{\bL^{-(r+l-2)d_1-(r+l-4)d_2-\ldots-(r-l+4)d_{l-2}}(\bL-1)\zeta_X(\bL^{-(r-l+2)})}=\\
\frac{\bL^{d+(r-l+2)(1-g)}[\Jac]}{(\bL-1)\zeta_X(\bL^{-(r-l+2)})}\lim_{d_1\to-\infty}\ldots\lim_{d_{l-2}\to-\infty}
\frac{[\Bun_{r,d_1,\ldots,d_{l-2},d'}^{\ge\tau}\to\Bun_{r,d}^{\ge\tau}]}
{\bL^{-(r+l-3)d_1-(r+l-5)d_2-\ldots-(r-l+3)d_{l-2}}}=\\
\frac{\bL^{d+(r-l+2)(1-g)}[\Jac]}{(\bL-1)\zeta_X(\bL^{-(r-l+2)})}\cdot
\frac{\bL^{(l-2)(d+(1-g)\frac{2r-l+3}2)}[\Jac]^{l-2}}
    {(\bL-1)^{l-2}\prod_{i=r-l+3}^r\!\zeta_X(\bL^{-i})}\;\mathbf1_{\Bun_{r,d}^{\ge\tau}}=
    \frac{
    \bL^{(l-1)(d+(1-g)\frac{2r-l+2}2)}[\Jac]^{l-1}}
    {(\bL-1)^{l-1}\prod_{i=r-l+2}^r\!\zeta_X(\bL^{-i})}\;\mathbf1_{\Bun_{r,d}^{\ge\tau}}.
\end{multline}
The theorem is proved.\qed

\section{Some identities in motivic Hall algebras }\label{sect:Hall}
For an ind-constructible abelian (or more generally triangulated $A_\infty$) category Kontsevich and Soibelman define its motivic Hall algebra in~\cite{KontsevichSoibelman08} (see also~\cite{JoyceConfigurationsII}). We will need this construction for the category of coherent sheaves on the curve $X$. In this case, the formulas of~\cite{KontsevichSoibelman08} simplify drastically, so we prefer to give a direct definition, referring the interested reader to~\cite{KontsevichSoibelman08} for the general case.\footnote{For a nice introduction to motivic Hall algebras of categories of coherent sheaves, see~\cite{BridgeleandHallIntro}.}

We also define a version of comultiplication.  Note that there is some peculiarity in the motivic case (in particular, coassociativity does not make literal sense). We notice that (in the particular case of the category of sheaves on curves) there is a compatibility between multiplication and comultiplication resembling Green's Theorem. Finally, we do some concrete calculations to be used in Section~\ref{sect:Proofs}.

In this section $\kk$ is a field of arbitrary characteristic except for Section~\ref{sect:torsheaves}, where we need the field to be of characteristic zero. We keep the assumptions from the previous section, in particular: $X$ is a smooth geometrically connected projective curve over the field $\kk$ and we assume that there is a divisor~$\divisor$ on $X$ defined over $\kk$ such that $\deg\divisor=1$. As before, $K$ denotes an extension of $\kk$.

\subsection{Motivic Hall algebra of the category of coherent sheaves}
For any  stack $\cX$ we consider the ring
\[
	\cMot(\cX)[\sqrt\bL]:=\cMot(\cX)[t]/(t^2-\bL).
\]	
We can easily extend pullbacks, pushforwards, and products to these rings. We note that $\cMot(\cX)\subset\cMot(\cX)[\sqrt\bL]$, since $t^2-\bL$ is a monic polynomial. Thus, when proving an identity in $\cMot(\cX)$, we may work in $\cMot(\cX)[\sqrt\bL]$.

Set
\[
    \Gamma:=\Z^2,\qquad\Gamma_+:=\{(r,d)\in\Z_{\ge0}\times\Z\,|\;d\ge0\text{ if }r=0\}
\]
so that $\Gamma_+$ is a subsemigroup of $\Gamma$. If $F$ is a coherent sheaf on $X_K$ of generic rank $r$ and degree $d$, we say that $F$ is of class $(r,d)\in\Gamma_+$, we also write $\cl(F)=(r,d)$. Let $\Coh_\gamma$ be the moduli stack of coherent sheaves on $X$ of class $\gamma\in\Gamma_+$. In particular, we have $\Coh_{(0,0)}=\Spec\kk$. We also consider $\Coh_r:=\sqcup_d\Coh_{(r,d)}$; this is the moduli stack of rank $r$ sheaves. Finally, set $\Coh:=\sqcup_{r\ge0}\Coh_r$.

For $(r_i,d_i)\in\Gamma$, $i=1,2$, we set
\[
    \langle(r_1,d_1),(r_2,d_2)\rangle=(1-g)r_1r_2+(r_1d_2-r_2d_1)
\]
and
\[
    ((r_1,d_1),(r_2,d_2))=\langle(r_1,d_1),(r_2,d_2)\rangle+\langle(r_2,d_2),(r_1,d_1)\rangle=(2-2g)r_1r_2.
\]
Note that the symmetrized form only involves $r_1$ and $r_2$.

Next, we note that for coherent sheaves $F_1$ and $F_2$ on $X$ we have
\[
    \dim\Hom(F_1,F_2)-\dim\Ext^1(F_1,F_2)=\langle\cl(F_1),\cl(F_2)\rangle.
\]

Set
\[
    \cH_\gamma:=\cMot(\Coh_\gamma)[\sqrt\bL]\text{ and }
    \cH_\gamma^{fin}:=\cMot^{fin}(\Coh_\gamma)[\sqrt\bL]
    \text{ for }\gamma\in\Gamma_+.
\]
Finally, set
\[
    \cH':=\bigoplus_{\gamma\in\Gamma_+}\cH_\gamma,\qquad
    \cH'_{fin}:=\bigoplus_{\gamma\in\Gamma_+}\cH_\gamma^{fin},\text{ and }
    \widehat\cH':=\prod_{\gamma\in\Gamma_+}\cH_\gamma.
\]
For $\gamma_1,\gamma_2\in\Gamma_+$ let $\Coh_{\gamma_2,\gamma_1}$ be the stack classifying pairs of sheaves $(F_1\subset F)$ such that $\cl(F_1)=\gamma_1$, $\cl(F/F_1)=\gamma_2$. We have a diagram
\[
    \Coh_{\gamma_2}\times\Coh_{\gamma_1}\xleftarrow{p}\Coh_{\gamma_2,\gamma_1}\xrightarrow{s}\Coh_{\gamma_1+\gamma_2}.
\]
Here $p(F_1\subset F)=(F/F_1,F_1)$, $s(F_1\subset F)=F$. Note that both $p$ and $s$ are 1-morphisms of finite type.

The multiplication on $\cH'$ is defined as follows: if $f_i\in\cH_{\gamma_i}$ ($i=1,2$), then
\[
    f_2f_1:=\bL^{\frac12\langle\gamma_2,\gamma_1\rangle}s_!p^*(f_2\boxtimes f_1).
\]
We extend this $\cH'$ by bilinearity. The above product makes $\cH'$ into a unital associative algebra over $\cMot(\kk)[\sqrt\bL]$.

Directly, one can define the $n$-fold multiplication on $\cH'$ as follows. Let $\Coh_{\gamma_n,\ldots,\gamma_1}$ denote the stack of filtrations of coherent sheaves $0=F_0\subset F_1\subset\ldots\subset F_n=F$ such that for all $i$ we have $\cl(F_i/F_{i-1})=\gamma_i$. We have a diagram
\[
    \Coh_{\gamma_n}\times\ldots\times\Coh_{\gamma_1}
    \xleftarrow{p_{(n)}}\Coh_{\gamma_n,\ldots,\gamma_1}\xrightarrow{s_{(n)}}\Coh_{\gamma_1+\ldots+\gamma_n}.
\]
The 1-morphisms $p_{(n)}$ and $s_{(n)}$ are defined similarly to $p$ and $s$; they are also of finite type. Now, for $f_i\in\cH_{\gamma_i}$ we have
\[
    f_n\ldots f_1=\bL^{\sum_{i>j}\frac12\langle\gamma_i,\gamma_j\rangle}s_{(n)!}p_{(n)}^*(f_n\boxtimes\ldots\boxtimes f_1).
\]

Note that $\cH'_{fin}\subset\cH'$ is a subalgebra. On the other hand, $\widehat\cH^{\prime}$ is not an algebra because multiplication would involve infinite summation. However, $\widehat\cH'_{tor}:=\prod_{d\ge0}\cH_{(0,d)}$ is. Moreover, the restriction of the multiplication on $\cH'$ to $\cH'\otimes\left(\bigoplus_d\cH_{0,d}\right)$ extends to the action $\widehat\cH'\otimes\widehat\cH'_{tor}\to\widehat\cH'$; this action preserves $\cH'$ and the rank gradation on $\cH'$.

Precisely, the action is defined as follows: Let $\Coh_{tor}:=\sqcup_{d\ge0}\Coh_{0,d}$ be the stack of torsion sheaves and let
$\Coh_{\bullet,tor}$ denote the stack classifying pairs $F_1\subset F$, where $F$ is a coherent sheaf on $X$, $F_1$ is subsheaf such that $F_1$ is torsion. We have projections
\[
    \Coh\times\Coh_{tor}\xleftarrow{p}\Coh_{\bullet,tor}\xrightarrow{s}\Coh
\]
defined by $p(F_1\subset F)=(F/F_1,F_1)$ and $s(F_1\subset F)=F$. Both projections are of finite type, so for $f_1\in\cH'_{tor}$ and $f_2\in\widehat\cH'$ we define
\[
    f_2f_1:=s_!p^*(T(f_2\boxtimes f_1)),
\]
where $T$ acts on $\Coh_{(0,d)}\times\Coh_{(r,e)}$ via the multiplication by $\bL^{-\frac12dr}$.

\begin{remark}
We can define Hall algebras using $\Mot(\Coh_\gamma)[\sqrt\bL]$ instead of $\cMot(\Coh_\gamma)[\sqrt\bL]$. Everything except for Proposition~\ref{pr:HallFormulas}\eqref{Halliv} would work. Thus Proposition~\ref{pr:MainCalc} is also true as a statement about series with coefficients in $\Mot(\kk)[\sqrt\bL]$.
\end{remark}

\subsection{Extended Hall algebras}
Set $\Gamma':=\Z$ and let $\Z[\Gamma']$ be the group algebra. We denote the element corresponding to $r\in\Gamma'$ by $k_r$. Thus $\Z[\Gamma']\approx\Z[k_1,k_1^{-1}]$ is the ring of Laurent polynomials. We let $\Gamma'$ act on $\cH'$ via
\[
    r\cdot f=\bL^{(1-g)rr'}f\text{ whenever }f\in\cH_{r',d}.
\]
This gives a semidirect product
\[
    \cH:=\cH'\otimes_\Z\Z[\Gamma'].
\]
Thus, $\cH$ is an associative algebra. Note that $\cH$ is graded by $\Gamma_+$. We view $\cH'$ and $\Z[\Gamma']$ as subalgebras of~$\cH$. We have in $\cH$: $k_r f=\bL^{(1-g)rr'}f k_r$ if $f\in\cH_{r',d}$. We define the subalgebra $\cH_{fin}:=\cH_{fin}'\otimes\Z[\Gamma']\subset\cH$, the $\cMot(\kk)[\sqrt\bL]$-module $\widehat\cH:=\widehat\cH'\otimes\Z[\Gamma']\supset\cH$ and the algebra $\widehat\cH_{tor}:=\widehat\cH'_{tor}\otimes\Z[\Gamma']$ acting on $\widehat\cH$ on the right.

\begin{remark}
One may define a larger algebra $\cH'\otimes\Z[\Gamma]$ as in~\cite[Sect.~4.1]{SchiffmannIndecomposable} by making $\gamma\in\Gamma$ act on $f\in\cH_{\gamma'}$ via $\gamma\cdot f=\bL^{\frac12(\gamma,\gamma')}f$. However, the symmetrized bilinear form depends on the ranks only, so the element of $\Z[\Gamma]$ corresponding to $(0,1)\in\Gamma$ is central. The quotient by the ideal generated by $(0,1)$ is isomorphic to $\cH$ (we identify $\Gamma'$ with $\Gamma/(0,1)\Z$).
\end{remark}

\subsection{``Comultiplication'' in the Hall algebra} One would like to define a comultiplication $\widehat\cH\to\widehat\cH\hat\otimes\widehat\cH$, where $\hat\otimes$ is the product completed with respect to the $\Gamma_+$-grading. However, in the motivic case this is not possible because $\cMot(\Coh_{\gamma_2}\times\Coh_{\gamma_1})\ne\cMot(\Coh_{\gamma_2})\hat\otimes\cMot(\Coh_{\gamma_1})$. We will circumvent this problem as follows. Set $\cH_{\gamma_2,\gamma_1}:=\cMot(\Coh_{\gamma_2}\times\Coh_{\gamma_1})[\sqrt\bL]$ and $\widehat\cH_{(2)}:=\prod_{\gamma_2,\gamma_1\in\Gamma_+}\cH_{\gamma_2,\gamma_1}\otimes\Z[(\Gamma')^2]$. Later, we will also need the space $\cH_{(2),fin}:=\oplus_{\gamma_2,\gamma_1\in\Gamma_+}\cH^{fin}_{\gamma_2,\gamma_1}\otimes\Z[(\Gamma')^2]$, where $\cH^{fin}_{\gamma_2,\gamma_1}:=\cMot(\Coh_{\gamma_2}\times\Coh_{\gamma_1})[\sqrt\bL]$.

We are going to construct a map
\begin{equation*}
    \Delta:\widehat\cH\to\widehat\cH_{(2)}.
\end{equation*}
To give such a map, one needs to give for each pair $(\gamma_2,\gamma_1)\in\Gamma_+^2$ a map $\Delta_{\gamma_2,\gamma_1}:\widehat\cH\to\cH_{\gamma_2,\gamma_1}\otimes\Z[(\Gamma')^2]$. This map is given by
\[
\Delta_{\gamma_2,\gamma_1}(f\otimes k_r):=\bL^{\frac12\langle\gamma_2,\gamma_1\rangle} p_!s^*f_{\gamma_1+\gamma_2}
\otimes k_{r_1+r}\otimes k_r,
\]
where $f\in\cH'$, $f_{\gamma_1+\gamma_2}$ is the projection of $f$ to $\cH_{\gamma_1+\gamma_2}$, and $\gamma_1=(r_1,d_1)$.

Note that we have a homomorphism of $\cMot(\kk)$-modules $\boxtimes:\widehat\cH\hat\otimes\widehat\cH\to\widehat\cH_{(2)}$ given by external product of motivic functions.

\begin{remark}
The coassociativity does not make sense for $\Delta$. However, one has the following replacement. First, one defines the $n$-point completed Hall algebra $\widehat\cH_{(n)}$ and the $n$-th comultiplication $\Delta_{(n)}:\widehat\cH\to\widehat\cH_{(n)}$. Assume that we have $\Delta(f)=\boxtimes(g)$, where $f\in\widehat\cH$, $g\in\widehat\cH\hat\otimes\widehat\cH$. Then for any $n$ and $m$ we have
\[
    \Delta_{(m+n)}(f)=(\Delta_{(m)}\boxtimes\Delta_{(n)})(g).
\]
We will not use this coassociativity.
\end{remark}

\begin{proposition}\label{pr:DeltaHom}
Assume that either $f_1,f_2\in\cH_{fin}$, or $f_1\in\widehat\cH$, $f_2\in\widehat\cH_{tor}$. Then
\[
    \Delta(f_1f_2)=\Delta(f_1)\Delta(f_2).
\]
In particular the product in the RHS converges.
\end{proposition}
Note that $\widehat\cH_{(2)}$ is not an algebra because the product involves infinite summation. The convergence part of the proposition means that, under assumptions of the proposition, for any degree $\delta\in(\Gamma_+)^2$ all but finitely many terms in the corresponding sum are zero. This is easy to check if $f_2\in\widehat\cH_{tor}$; in the case $f_1,f_2\in\cH_{fin}$ this follows from the fact that for any finite type substack $\cX\subset\Coh$ there is $d\in\Z$ such that for any $F\in\cX(K)$ and any quotient $F'$ of $F$ we have $\deg F'\ge d$.

We leave a lengthy proof of the equation to the reader but we observe that the argument of~\cite[Sect.~1.5]{SchiffmannLectures} is actually motivic.

\subsection{The bilinear form}
According to Section~\ref{sect:BilinearForm}, we have a bilinear form $\cH_\gamma^{fin}\otimes\cH_\gamma\to\cMot(\kk)[\sqrt\bL]$; we extend it to $\cH'_{fin}\otimes\widehat\cH'$ by letting $\cH_\gamma$ to be pairwise orthogonal. We extend it to $\cH_{fin}\otimes\widehat\cH$ by setting $(f\otimes k_r|g\otimes k_{r'})=\bL^{(1-g)rr'}(f|g)$. Similarly, we define a bilinear form $\cH_{(2),fin}\otimes\widehat\cH_{(2)}\to\cMot(\kk)[\sqrt\bL]$.

\begin{lemma}\label{lm:CoProdPairing}
Let $f\in\widehat\cH$, $g_1,g_2\in\cH_{fin}$. Then
\[
    (g_1g_2|f)=(g_1\boxtimes g_2|\Delta(f)).
\]
\end{lemma}
\begin{proof}
A simple calculation using Lemma~\ref{lm:MotBilProd}.
\end{proof}

\subsection{``Standard'' objects}
For $\gamma\in\Gamma_+$ set $\mathbf1_\gamma:=\mathbf1_{\Coh_\gamma}\in\cH_\gamma$, $\mathbf1_\gamma^{vec}:=\mathbf1_{\Bun_\gamma}\in\cH_\gamma$.
Define the generating series
\[
    E_r(z):=\sum_{d\in\Z}\mathbf1_{r,d}z^d\in\prod_{d\in\Z}\cH_{(r,d)}z^d\subset\cH[[z^{-1},z]].
\]
Define also $E_r^{vec}(z):=\sum_{d\in\Z}\mathbf1^{vec}_{r,d}z^d$. Note that $E_0(z)\in\cH_{tor}[[z]]$. Note also that $E_0^{vec}(z)=1$.

\begin{remark}\label{rm:Convergence}
The series $E_r(z)$ is homogeneous in the sense that the coefficient at $z^d$ belongs to $\cH_{r,d}$. Thus, for any $x\in\cMot(\kk)[\sqrt\bL]$ we can calculate $E_r(x)$ as an element of the completion $\widehat\cH$. Moreover, we can recover $E_r(z)$ from $E_r(1)$ as $E_r(z)=\sum_{d\in\Z}(E_r(1))_{(r,d)}z^d$, where the subscript $(r,d)$ stands for the $(r,d)$-component.

We can use this correspondence between homogeneous series and elements of $\widehat\cH$ to multiply any homogeneous series by a homogeneous series of rank 0 on the right because $\widehat\cH_{tor}$ acts on $\widehat\cH$. Proposition~\ref{pr:HallFormulas}(\ref{Hallv},\ref{Halliii2}) below should be understood in this sense.
\end{remark}

For $r>0$ set
\[
	vol_r:=\frac{\bL^{(g-1)(r^2-1)}}{\bL-1}[\Jac]
	\zeta_X(\bL^{-2})\ldots\zeta_X(\bL^{-r}) \in\cMot(\kk).
\]
\begin{remark}
The stack $\Bun_{r,d}$ is of infinite type for $r>1$. However, one can define its motivic class as
\[
	[\Bun_{r,d}]:=\lim_{\tau\to-\infty}[\Bun_{r,d}^{\ge\tau}]\in\cMot(\kk).
\]
(See~\cite[Lemma~3.1]{BehrendDhillon}.) It is an easy consequence of~\cite[Sect.~6]{BehrendDhillon} that $[\Bun_{r,d}]=vol_r$. We will never use this in the current paper but the notation $vol_r$ will be convenient when comparing our paper with~\cite{SchiffmannIndecomposable}.
\end{remark}

\begin{proposition}\label{pr:HallFormulas}
We have the following identities.\\ \stepzero
\noindstep\label{Hallv}
\[
    E_r(z)=E_r^{vec}(z)E_0(\bL^{-\frac12r}z).
\]
\noindstep\label{Halli}
\[
    E_0(z)E_0(w)=E_0(w)E_0(z).
\]
\noindstep\label{Hallii}
\[
    E_0(z)E_r^{vec}(w)=\left(\prod_{i=0}^{r-1}\zeta_X\left(\bL^{-\frac r2+i}\frac zw\right)\right)
    E_r^{vec}(w)E_0(z).
\]
\noindstep\label{Halliii}
\[
    \Delta(E_r(z))=
    \sum_{s+t=r}\bL^{\frac12st(g-1)}E_s(\bL^{\frac t2}z)k_t\boxtimes E_t(\bL^{-\frac s2}z).
\]
\noindstep\label{Halliii2}
\[
    \Delta(E_r^{vec}(z))=
    \sum_{s+t=r}\bL^{\frac12st(g-1)}
    E_s^{vec}(\bL^{\frac t2}z)E_0(\bL^{\frac{t-s}2}z)E_0^{-1}(\bL^{-\frac{t+s}2}z)k_t
    \boxtimes E_t^{vec}(\bL^{-\frac s2}z).
\]
\noindstep\label{Halliv}
\[
    E_r^{vec}(\bL^{\frac12(1-r)}z_1)=C\cdot\res_{\frac{z_2}{z_1}=\ldots=\frac{z_r}{z_{r-1}}=\bL^{-1}}
    (E_1^{vec}(z_r)\ldots E_1^{vec}(z_1))\prod_{i=2}^r\frac{dz_i}{z_i},
\]
where
\[
    C=\bL^{\frac14(1-g)r(r-1)}vol_1^{-r}vol_r.
\]
\end{proposition}
\begin{proof}
We start the proof with~\eqref{Hallv}. In view of Remark~\ref{rm:Convergence} and the definition of the action of $\cH_{tor}$ on~$\cH$, we only need to show that we have in $\cMot(\Coh)$:
\[
    \mathbf1_{\Coh}=s_!\mathbf1_\cX,
\]
where $\cX$ is the constructible subset of $\Coh_{\bullet,tor}$ corresponding to the pairs $(F_1\subset F)$ such that $F/F_1$ is a vector bundle. This follows from the uniqueness of the torsion subsheaf and Corollary~\ref{cor:pointwEqual}.

\eqref{Halli} is equivalent to the equation
\[
    [\Coh_{(0,l_1),(0,l_2)}\to\Coh_{(0,l_1+l_2)}]=
    [\Coh_{(0,l_2),(0,l_1)}\to\Coh_{(0,l_1+l_2)}].
\]
We will show a stronger equation in $\Mot(\Coh_{(0,l_1+l_2)}\times\Coh_{(0,l_1)})$:
\[
    [\Coh_{(0,l_1),(0,l_2)}\xrightarrow{\phi}\Coh_{(0,l_1+l_2)}\times\Coh_{(0,l_1)}]=
    [\Coh_{(0,l_2),(0,l_1)}\xrightarrow{\psi}\Coh_{(0,l_1+l_2)}\times\Coh_{(0,l_1)}].
\]
Here $\phi$ and $\psi$ are defined by $\phi(F'\subset F)=(F,F/F')$, $\psi(F'\subset F)=(F,F')$. Let $F$ and $F'$ be torsion sheaves on $X_K$ representing a point $\xi:\Spec K\to\Coh_{(0,l_1+l_2)}\times\Coh_{(0,l_1)}$. According to Proposition~\ref{pr:pointwise zero} we just need to check that the motivic classes of the $\xi$-fibers of $\phi$ and $\psi$ are equal. These fibers are equal to the space of surjective (resp.~injective) morphisms $\Hom_{sur}(F,F')$ (resp.~$\Hom_{inj}(F',F)$).

Let $Z\subset X$ be the union of scheme-theoretic supports of $F$ and $F'$. We may assume that $Z_{red}=z$ is a single point of $X_K$ because the space of injective (or surjective) morphisms decomposes into the product over the points of $Z_{red}$. Note that the restriction of $F$ to $z$ corresponds to a vector space over $\kk(z)$; the same is true for $F'$. Upon choosing bases in these vector spaces, we identify $\Hom_{sur}(F|_z,F'|_z)$ and~$\Hom_{inj}(F'|_z,F|_z)$ with spaces of matrices of maximum rank (of sizes $\dim F'|_z\times\dim F|_z$ and $\dim F|_z\times\dim F'|_z$ resp.); we see that the motivic classes of these spaces coincide.

Next, a morphism from $F\to F'$ is surjective if and only if its restriction to $z$ is surjective. Now it is easy to see that the fibers of the restriction morphism $\Hom_{sur}(F,F')\to\Hom_{sur}(F|_z,F'|_z)$ are vector spaces. Similarly, the fibers of the morphism
$\Hom_{inj}(F',F)\to\Hom_{inj}(F'|_z,F|_z)$ are vectors spaces easily seen to be of the same dimension. One more application of Proposition~\ref{pr:pointwise zero} completes the proof of \eqref{Halli}.

To prove~\eqref{Hallii}, note first that by Lemma~\ref{lm:zeta} we have
\[
    \prod_{i=0}^{r-1}\zeta_X\left(\bL^{-\frac r2+i}\frac zw\right)=
    \zeta_{X\times\P^{r-1}}\left(\bL^{-\frac r2}\frac zw\right).
\]
Thus \eqref{Hallii} is equivalent to the equation for all $d\ge0$ and $e\in\Z$:
\[
    \mathbf1_{(0,d)}\mathbf1_{(r,e)}^{vec}=
    \sum_{i=0}^d\bL^{-\frac{ir}2}[(X\times\P^{r-1})^{(i)}]\mathbf1^{vec}_{(r,e+i)}\mathbf1_{(0,d-i)}.
\]
Unwinding the definition of multiplication in the Hall algebra, we see that this is equivalent to the following equation. Let $\widehat{\Coh}_{(0,d),(r,e)}$ be the open substack of $\Coh_{(0,d),(r,e)}$ classifying pairs of sheaves $(F_1\subset F)$ such that $F_1$ is torsion free. Similarly, let $\widehat{\Coh}_{(r,e),(0,d)}$ be the open substack of $\Coh_{(r,e),(0,d)}$ classifying pairs of sheaves $(F_1\subset F)$ such that $F/F_1$ is torsion free. It is enough to show that in $\Mot(\Coh_{r,d+e})$ we have
\begin{equation}\label{eq:ii}
    [\widehat{\Coh}_{(0,d),(r,e)}\to\Coh_{r,d+e}]=
    \sum_{i=0}^d\bL^{r(d-i)}[(X\times\P^{r-1})^{(i)}][\widehat{\Coh}_{(r,e+i),(0,d-i)}\to\Coh_{r,d+e}].
\end{equation}
To this end, let $F$ be a coherent sheaf on $X_K$ of class $(r,d+e)$. Write $F=T\oplus E$, where $T$ is torsion and $E$ is torsion free. Set $i=\deg E-e$. The fiber $\cX_F$ of $\widehat{\Coh}_{(0,d),(r,e)}\to\Coh_{r,d+e}$ over $F$ is the scheme of subsheaves $F'\subset T\oplus E$ such that $F'$ is locally free of class $(r,e)$ (in particular, it is empty if $i<0$). Let $\pi:F\to E$ be the projection, the assignment $F'\mapsto\pi(F')$ is a morphism $\cX_F\to\Mod_i(E)$, where $\Mod_i(E)$ classifies degree $i$ modifications of $E$, that is, subsheaves $E'\subset E$ such that $E/E'$ is torsion of degree $i$. The fibers of this 1-morphism are isomorphic to vector spaces of dimension $\dim\Hom(F',T)=r\deg T=r(d-i)$. Thus
\[
    [\cX_F]=\bL^{r(d-i)}[\Mod_i(E)]=\bL^{r(d-i)}[(X\times\P^{r-1})^{(i)}],
\]
where the second equation follows from the proof of~\cite[Prop.~3.6]{GarciaPradaHeinlothSchmitt}.

Now we calculate the fiber of $\widehat{\Coh}_{(r,e+i),(0,d-i)}\to\Coh_{r,d+e}$ over $F$. This is the scheme of subsheaves $T'\subset F=T\oplus E$ such that $T'$ is torsion of degree $d-i$ and such that $F/T'$ is torsion free. But then we necessarily have $T=T'$. Thus, the fiber consists of a unique point if $d-i=d+e-\deg E$ and empty otherwise. Now we easily derive~\eqref{eq:ii} from Proposition~\ref{pr:pointwise zero}.

 Next, \eqref{Halliii} is equivalent to the following statement: for any $\gamma_1,\gamma_2\in\Gamma_+$ we have $\Delta_{\gamma_2,\gamma_1}(\mathbf1_{\gamma_1+\gamma_2})=
\bL^{-\frac12\langle\gamma_2,\gamma_1\rangle}\mathbf1_{\gamma_2}\boxtimes\mathbf1_{\gamma_1}$. Unwinding the definition of $\Delta_{\gamma_2,\gamma_1}$, we see that this is  equivalent to
\[
    [\Coh_{\gamma_2,\gamma_1}\to\Coh_{\gamma_2}\times\Coh_{\gamma_1}]=
    \bL^{-\langle\gamma_2,\gamma_1\rangle}\mathbf1_{\Coh_{\gamma_2}\times\Coh_{\gamma_1}}.
\]
Let $F_i$ be coherent sheaves on $X_K$ of class $\gamma_i$ ($i=1,2$). According to Proposition~\ref{pr:pointwise zero}, we just need to show that the motivic class of the moduli stack $\cX$ of exact sequences $0\to F_1\to F\to F_2\to0$ is equal to $\bL^{-\langle\gamma_2,\gamma_1\rangle}$ in $\Mot(K)$. This follows easily from the fact that we have an affine bundle $\Ext^1(F_2,F_1)\to\cX$ modeled over the additive group $\Hom(F_2,F_1)$. (Recall that $\dim\Ext^1(F_2,F_1)-\dim\Hom(F_2,F_1)=-\langle\gamma_2,\gamma_1\rangle$.)

For~\eqref{Halliii2}, note first that $E_0(z)$ is invertible in $\cH[[z]]$. By part~\eqref{Hallv} we have
\[
    E_r^{vec}(z)=E_r(z)E_0^{-1}(\bL^{-\frac12r}z),
\]
where this equation should be understood as explained in Remark~\ref{rm:Convergence}. It remains to apply the comultiplication $\Delta$, and use Proposition~\ref{pr:DeltaHom} and part~\eqref{Halliii} of the current proposition.

For part~\eqref{Halliv}, we have
\begin{equation}\label{eq:E1prod}
E_1^{vec}(z_r)\ldots E_1^{vec}(z_1)=\sum_{d_1,\ldots,d_r}
\bL^{\frac{r(r-1)}4(1-g)+\frac{(r-1)d_1+(r-3)d_2+\ldots+(1-r)d_r}2}
[\Bun_{r,d_1,\ldots,d_r}\to\Coh_r]z_1^{d_1}\ldots z_r^{d_r},
\end{equation}
where $\Bun_{r,d_1,\ldots,d_r}$ is defined in Section~\ref{sect:BorelRed}. Since both sides of our equation are supported on $\Bun_r$, it is enough to prove the statement upon restricting to $\Bun_r$. Note that convergence on $\cMot(\Bun_r)$ is convergence on open substacks of finite type (see Section~\ref{sect:MotFun1}), so it is enough to show \eqref{Halliv} upon restricting to $\Bun_{r,d}^{\ge\tau}$ (see Lemma~\ref{lm:Bun+}(iv) and the definition of the residue in Section~\ref{sect:res}). We get
\[
E_1^{vec}(z_r)\ldots E_1^{vec}(z_1)|_{\Bun_{r,d}^{\ge\tau}}=\bL^{\frac{r(r-1)}4(1-g)+\frac{(1-r)d}2}E_{r,d}^{\ge\tau}(z_1,\ldots,z_r),
\]
where $E_{r,d}^{\ge\tau}(z_1,\ldots,z_r)$ is defined in Section~\ref{sect:Eisenstein}. It remains to use Theorem~\ref{th:ResHarder2}.
\end{proof}

\subsection{Truncated generating series}
Note that the slope of a non-zero torsion sheaf is equal to $+\infty$. Thus, if $0=E_0\subset E_1\subset\ldots\subset E_t=E$ is the HN-filtration on a vector bundle $E$ and $T$ is a torsion sheaf, then the HN-filtration on $T\oplus E$ is given by
\[
    0\subset T\subset T\oplus E_1\subset\ldots\subset T\oplus E_t=T\oplus E.
\]
We define $\Coh_{r,d}^{\ge0}$ as the constructible (in fact, open) subset of $\Coh_{r,d}$ classifying HN-nonnegative sheaves, that is, sheaves $T\oplus E$ as above such that $E$ is HN-nonnegative. We define $\Bun_{r,d}^{<0}$ to be the constructible subset of $\Coh_{r,d}$ classifying sheaves with strictly negative HN-type. The reason for the notation is that every such sheaf is a vector bundle. It follows easily from Lemma~\ref{lm:Bun+}(iii) that these subsets are of finite type. Set
\[
\mathbf1_{r,d}^{\ge0}=\mathbf1_{\Coh_{r,d}^{\ge0}},\qquad
\mathbf1_{r,d}^{vec,\ge0}=\mathbf1_{\Bun_{r,d}^{\ge0}},\qquad
\mathbf1_{r,d}^{<0}=\mathbf1_{\Bun_{r,d}^{<0}}.
\]
We also define the generating series
\[
E_r^{\ge0}(z):=\sum_{d\in\Z}\mathbf1_{r,d}^{\ge0}z^d\in\cH_{fin}[[z]].
\]
Define similarly $E_r^{vec,\ge0}(z)\in\cH_{fin}[[z]]$ and $E_r^{<0}(z)\in z^{-1}\cH_{fin}[[z^{-1}]]$.

\begin{lemma}\label{lm:HallProduct}We have
\[
\begin{split}
(i)\qquad & E_r(z)=\sum_{\substack{s+t=r \\ s,t\ge0}}\bL^{\frac12(g-1)st}E_s^{<0}(\bL^{\frac t2}z)E_t^{\ge0}(\bL^{-\frac s2}z),\\
(ii)\qquad & E_r^{vec}(z)=\sum_{\substack{s+t=r \\ s,t\ge0}}\bL^{\frac12(g-1)st}E_s^{<0}(\bL^{\frac t2}z)E_t^{vec,\ge0}(\bL^{-\frac s2}z),\\
(iii)\qquad & E_r^{\ge0}(z)=E_r^{vec,\ge0}(z)E_0(\bL^{-\frac12r}z).
\end{split}
\]
\end{lemma}
\begin{remark}\label{rm:HallInfinite}
Note that the RHS of (i) and (ii) involve infinite summation. As we will see from the proof, the restrictions of the series to every finite type substack of each $\Coh_\gamma$ have only finitely many non-zero terms. (cf.~the discussion of the topology on $\cMot$ in Section~\ref{sect:MotFun1}).
\end{remark}
\begin{proof}
Let $\Coh^{\pm}_{\gamma_2,\gamma_1}$ be the constructible subset of $\Coh_{\gamma_2,\gamma_1}$ classifying pairs $F_1\subset F$ such that $F_1$ is HN-nonnegative, $F/F_1$ has strictly negative HN-type.

Let $s_{\gamma_2,\gamma_1}:\Coh_{\gamma_2,\gamma_1}\to\Coh_{\gamma_1+\gamma_2}$ be the forgetful 1-morphism (denoted simply by $s$ above), let $\Coh^{\pm,\prime}_{\gamma_2,\gamma_1}$ be the constructible image of $\Coh^{\pm}_{\gamma_2,\gamma_1}$ under this 1-morphism. Since for every sheaf $F$ there is a unique exact sequence $0\to F^{\ge0}\to F\to F^{<0}\to0$ with HN-nonnegative $F^{\ge0}$, $F^{<0}$ having strictly negative HN-type, we get
\[
    \sum_{\gamma_1+\gamma_2=\gamma}(s_{\gamma_2,\gamma_1})_!\mathbf1_{\Coh^{\pm}_{\gamma_2,\gamma_1}}=
    \sum_{\gamma_1+\gamma_2=\gamma}\mathbf1_{\Coh^{\pm,\prime}_{\gamma_2,\gamma_1}}=\Coh_{\gamma_1+\gamma_2}.
\]
We note that the sums are finite on each substack of finite type according to Lemma~\ref{lm:Bun+}(iv). Writing $\gamma_1+\gamma_2=(r,d)$, we get the following Hall algebra identity
\[
    \mathbf1_{(r,d)}=\sum_{\substack{s+t=r \\ s,t\ge0}}\sum_{e+f=d}
    \bL^{\frac12((g-1)st+te-sf)}\mathbf1_{s,e}^{<0}\mathbf1_{t,f}^{\ge0}.
\]
This is equivalent to the first formula of the lemma. The second formula is proved similarly. The proof of the third formula is completely similar to the proof of Proposition~\ref{pr:HallFormulas}\eqref{Hallv}.
\end{proof}

\subsection{Torsion sheaves}\label{sect:torsheaves}
Note that $\mathbf1_{0,l}\in\cH_{(0,l)}^{fin}$.
\begin{proposition}\label{pr:torsionsheaves}
We have
\[
    \sum_{l\ge0}(\mathbf1_{0,l}|\mathbf1_{0,l})z^l=\Exp\left(\frac{[X]}{\bL-1}z\right)=
    \prod_{i\ge1}\zeta_X(\bL^{-i}z).
\]
\end{proposition}
\begin{proof}
We need some preliminaries. Let $\cN_d$ be the stack of dimension $d$ vector spaces with nilpotent endomorphisms (later, we will identify $\cN_d$ with the stack of coherent sheaves supported at a point on a curve set-theoretically).
\begin{lemma}\label{cltorspt}
We have
\[
    [\cN_d]=\frac{\bL^{d(d-1)}}{(\bL^d - 1)\cdots(\bL^d - \bL^{d-1})}.
\]
\end{lemma}
\begin{proof}
Clearly, $[\cN_d] = [Nil_d]/[\GL_d]$, where $Nil_d$ is the nilpotent cone for $\gl_d$. Thus we only need to show that $[Nil_d]=\bL^{d(d-1)}$.

To compute $[Nil_d]$, note that for every $f\in\gl_d$ the Fitting Decomposition Theorem lets us write $\kk^d=\Ker(f^d)\oplus\Im(f^d)$. We can write $\gl_d=\bigsqcup_{m=0}^d E_m$ as a disjoint union of subvarieties, where $E_m$ consists of $f\in\gl_d$ such that $\Ker(f^d)$ in the Fitting decomposition has dimension equal to $m$.

For each $m$, denote by $V_m$ the scheme parameterizing decompositions $\kk^d=L_1\oplus L_2$, where $L_1$ is of dimension $m$. Let $\tilde V_m\subset E_m\times V_m$ be the incidence variety consisting of triples $(f,L_1,L_2)$ such that $\Ker f^d=L_1$, $\Im f^d=L_2$. For every extension $K\supset\kk$ the $K$-fibers of the projection $\tilde V_m\to E_m$ are points, while the $K$-fibers of the projections $\tilde V_m\to V_m$ are easily seen to be isomorphic to $(Nil_m)_K\times(\GL_{d-m})_K$. Now, using Proposition~\ref{pr:pointwise zero}, we get:
\begin{multline*}
    [E_m]=[\tilde V_m]=[Nil_m][\GL_{d-m}][V_m]=\\ [Nil_m][\GL_{d-m}]\left[\GL_d/(\GL_m\times\GL_{d-m})\right]
    =\frac{[\GL_d][Nil_m]}{[\GL_m]}.
\end{multline*}
Thus
\[
    [\gl_d]=\bL^{d^2}=\sum_{m=0}^d \frac{[\GL_d][Nil_m]}{[\GL_m]}.
\]
Now, it is easy to see by induction that $[Nil_d] = \bL^{d(d-1)}$.
\end{proof}

\begin{lemma}\label{lm:cNl}
\[
    \sum_{l\ge0}[\cN_l]z^l=\Exp\left(\frac z{\bL-1}\right).
\]
\end{lemma}
\begin{proof}
\begin{multline*}
\sum_{l\ge0}[\cN_l]z^l=1+\sum_{i\ge1}\frac{\bL^{i(i-1)}}{(\bL^i-1)\cdots(\bL^i-\bL^{i-1})}z^i =
    1+\sum_{i\ge1}\frac{(\bL^{-1}z)^i}{(1 - \bL^{-i})\cdots(1-\bL^{-1})}\\
    =\prod_{k>0}\frac1{1-\bL^{-k}z} =\Exp\left(\frac{\bL^{-1}z}{1-\bL^{-1}}\right)=\Exp\left(\frac z{\bL-1}\right).
\end{multline*}
\end{proof}

Let us view $\sqcup_{l\ge0}\cN_l$ as a $\Z$-graded stack. Similarly to Lemma~\ref{lm:Pow} consider pairs $(T,\phi)$, where $T\subset X$ is a finite subset of closed points, $\phi:T\to\sqcup_{l\ge0}\cN_l$ is a 1-morphism of degree $d$.  We define $\deg(\phi):=\sum_{x\in T}[\kk(x):\kk]\deg\phi(x)$, where $\deg\phi(x)=l$ if $\phi(x)\in\cN_l(\kk(x))$. We let $\cY_l$ be the stack classifying such pairs $(T,\phi)$ with $\deg\phi=l$.

\begin{lemma}
\[
    [\cY_d]=[\Coh_{0,d}].
\]
\end{lemma}
\begin{proof}
Let $\cZ_d$ be the stack classifying pairs $(T,\cE)$, where $T\subset X$ is as above, $\cE$ is a torsion sheaf on $X$ of degree $d$ set theoretically supported on $T$. We have a forgetful map $\cZ_d\to\Coh_{0,d}$ and an application of Corollary~\ref{cor:pointwEqual} gives $[\cZ_d\to\Coh_{0,d}]=\mathbf1_{\Coh_{0,d}}$ (indeed, the set-theoretical reduced support is uniquely defined) so that $[\cZ_d]=[\Coh_{0,d}]$.

On the other hand, denote by $\cZ_d^{(i)}$ the locally closed substack of $\cZ_d$ corresponding to $T$ such that $\deg T=\sum_{x\in T}[\kk(x):\kk]=i$. Define $\cY_d^{(i)}$ similarly.  We claim that
\begin{equation}\label{eq:Ydi}
    [\cY_d^{(i)}\to X^{(i)}]=[\cZ_d^{(i)}\to X^{(i)}].
\end{equation}
Indeed, if $K\supset\kk$ is a field extension, then a $K$-point of $X^{(i)}$ is given by a finite subset $T\subset X_K$. Choose local coordinates at the points of $T$. The fiber of $\cZ_d^{(i)}\to X^{(i)}$ over $T\subset X_K$, parameterizes all degree $d$ torsion sheaves supported on $T$.  Every such sheaf $E$ can be written uniquely as $\bigoplus_{x\in T}E_x$, and each $E_x$ can be identified with a pair consisting of a vector space over $\kk(x)$ and a nilpotent endomorphism. This gives an isomorphism between this fiber and the corresponding fiber of $\cY_d^{(i)}\to X^{(i)}$. It remains to use Proposition~\ref{pr:pointwise zero}.

Now we derive from~\eqref{eq:Ydi} that
\[
    [\cY_d]=\sum_i[\cY_d^{(i)}]=\sum_i[\cZ_d^{(i)}]=[\cZ_d].
\]
\end{proof}

Now we prove the proposition using Lemma~\ref{lm:Pow} and Lemma~\ref{lm:cNl}:
\[
    \sum_{d\ge0}[\Coh_{0,d}]z^d=\sum_{d\ge0}[\cY_d]z^d=\Pow\left(\sum_{d\ge0}[\cN_d]z^d,[X]\right)=
    \Exp\left(\frac{[X]}{\bL-1}z\right).
\]

\end{proof}

\section{Motivic classes of the stacks of vector bundles with filtrations and proofs of Theorems~\ref{th:NilpEnd} and~\ref{th:ExplAnsw}}\label{sect:Proofs}
In this section $\kk$ is a field of characteristic zero. We keep assumptions from the previous sections: $X$ is a smooth geometrically connected projective curve over the field $\kk$ and there is a divisor $\divisor$ on $X$ defined over $\kk$ such that $\deg\divisor=1$.

Fix $s\in\Z_{>0}$ and put $\underline z=(z_s,\ldots,z_1)$. Let $\underline r=(r_s,\ldots,r_1)$ be an $s$-tuple of positive integers; set $n=\sum_ir_i$. Set
\[
    G^{\ge0}_{\underline r}(\underline z,w):=
    \left(E_{r_s}(z_s)\ldots E_{r_1}(z_1)\left|E_n^{\ge0}(w)\right.\right)
\]
and
\[
    Y^{\ge0}_{\underline r}(\underline z,w):=
    \left(E_{r_s}^{vec}(z_s)\ldots E_{r_1}^{vec}(z_1)\left|E_n^{\ge0}(w)\right.\right).
\]
The product is taken in $\cH$. Note that, up to some powers of $\bL$, $G^{\ge0}_{\underline r}(\underline z,w)$ (resp.~$Y^{\ge0}_{\underline r}(\underline z,w)$) is the generating series for the motivic classes of the moduli stacks of rank $r$ HN-nonnegative coherent sheaves (resp.~vector bundles) with partial flags of type $(r_1,\ldots,r_s)$.

\subsection{Relating $G^{\ge0}$ with $Y^{\ge0}$}
\begin{proposition}\label{pr:GtoY}
We have an equation for series with coefficients in $\cMot(\kk)[\sqrt\bL]$:
\[
    G^{\ge0}_{\underline r}(\underline z,w)=X^{\ge0}_{\underline r}(\underline z,w)Y^{\ge0}_{\underline r}(\underline z,w),
\]
where
\[
    X^{\ge0}_{\underline r}(\underline z,w)=
    \Exp\left(\frac{[X]}{\bL-1}
    \left[
    \sum_i\bL^{-\frac12(n+r_i)}z_iw+\sum_{i>j}\frac{z_i}{z_j}
    \left(\bL^{\frac{r_j}2}-\bL^{-\frac{r_j}2}\right)\bL^{-\frac{r_i}2}
    \right]\right).
\]

\end{proposition}
Note that, for this to make sense, we need to extend $\Exp$ to the ideal of the ring
\[
    \cMot(\kk)[\sqrt\bL]\left[\!\left[z_1,\frac{z_2}{z_1},\ldots,\frac{z_n}{z_{n-1}},w\right]\!\right]
\]
consisting of series without constant term; but this is straightforward.

\begin{proof}
The proof repeats that of~\cite[Prop.~5.1]{SchiffmannIndecomposable}. It uses Proposition~\ref{pr:DeltaHom}, Lemma~\ref{lm:CoProdPairing}, Proposition~\ref{pr:HallFormulas}, Lemma~\ref{lm:HallProduct}(iii), and Proposition~\ref{pr:torsionsheaves}.
The only slight difference is that we do not use coassociativity to prove the equation
\[
\left(
\prod_{i=1}^s E_0(\bL^{-\frac{r_i}2}z_i)\left|E_0(\bL^{-\frac n2}w)
\right.\right)=
\prod_{i=1}^s \left(E_0(\bL^{-\frac{r_i}2}z_i)\left|E_0(\bL^{-\frac n2}w)\right.\right)
\]
but we apply Proposition~\ref{pr:HallFormulas}\eqref{Halli} and Lemma~\ref{lm:CoProdPairing} $s-1$ times instead.
\end{proof}

\subsection{Calculation of $Y^{\ge0}$}
Our first goal is to calculate $Y^{\ge0}_{\underline1}(\underline z,w)$, where $\underline1=(1,\ldots,1)=1^s$. Set also
\[
    Y^{<0}_{\underline1}(\underline z,w)=\left(
    E_1^{vec}(z_s)\ldots E_1^{vec}(z_1)\left|E_s^{<0}(w)
    \right.\right),\qquad
        Y_{\underline1}(\underline z,w)=\left(
    E_1^{vec}(z_s)\ldots E_1^{vec}(z_1)\left|E_s(w)
    \right.\right).
\]
We note that $E_1^{vec}(z_i)\in\cH_{fin}[[z_i^{-1},z_i]]$, so $Y_{\underline1}(\underline z,w)$ makes sense. We need a simple lemma.
\begin{lemma}\label{lm:Y1}
\[
    Y_{\underline1}(\underline z,w)=
        \bL^{(g-1)\frac{s(s-1)}4}[\Jac]^s
    \sum_{d_1,\ldots,d_s\in\Z}z_1^{d_1}\ldots z_s^{d_s}
    \bL^{\frac12\sum_i d_i(2i-s-1)}w^{\sum_i d_i}.
\]
\end{lemma}
\begin{proof}
According to~\eqref{eq:E1prod}, we have
\[
    Y_{\underline1}(\underline z,w)=\bL^{(1-g)\frac{s(s-1)}4}
    \sum_{d_1,\ldots,d_s\in\Z}
\bL^{\frac{(s-1)d_1+(s-3)d_2+\ldots+(1-s)d_s}2}
[\Bun_{s,d_1,\ldots,d_s}]z_1^{d_1}\ldots z_s^{d_s}w^{\sum_i d_i}.
\]
Thus it is enough to show that
\[
    [\Bun_{s,d_1,\ldots,d_s}]=\bL^{(g-1)\frac{s(s-1)}2+(s-1)d_s+\ldots+(1-s)d_1}[\Jac]^s.
\]
This is proved by induction on $s$. Consider the morphism
\[
    \Bun_{s,d_1,\ldots,d_s}\to
    \Bun_{s-1,d_1,\ldots,d_{s-1}}\times\Pic^{d_s}
\]
sending $(E_1\subset\ldots\subset E_{s-1}\subset E_s)$ to $((E_1\subset\ldots\subset E_{s-1}),E_s/E_{s-1})$. It is enough to show that the motivic classes of its fibers are equal to
\[
    \bL^{(g-1)(s-1)+(s-1)d_s-d_1-\ldots-d_{s-1}}.
\]
The fibers are the stacks $\Ext^1(E_s/E_{s-1},E_{s-1})/\Hom(E_s/E_{s-1},E_{s-1})$ of dimension
\begin{multline*}
    \dim\Ext^1(E_s/E_{s-1},E_{s-1})-\dim\Hom(E_s/E_{s-1},E_{s-1})=
    -\langle(1,d_s),(s-1,d_1+\ldots+d_{s-1})\rangle=\\
    (g-1)(s-1)+(s-1)d_s-d_1-\ldots-d_{s-1}.
\end{multline*}
This completes the proof. (Cf.~the proof of Lemma~\ref{lm:VectSpStack}.)
\end{proof}

Our nearest goal is to prove the motivic analogue of~\cite[Prop.~5.3]{SchiffmannIndecomposable}. The proof is very similar to the one given in~\cite{SchiffmannIndecomposable} except for two points. The first is that we do not have an honest comultiplication on $\cH$ but this problem is minor. The more important thing is that we do not know a priori that our series are expansions of rational functions. We will see, however, that this follows from the proof.

Recall that we defined normalized motivic zeta-function $\tilde\zeta_X$ and regularized motivic zeta-function $\zeta_X^*$ in Section~\ref{sect:zeta}. We will drop the index $X$ from now on.
\begin{proposition}\label{pr:MainCalc}
For any $s\ge1$ we have
\begin{equation}\label{eq:Y1>Formula}
    Y^{\ge0}_{\underline1}(\underline z,w)=
    \frac{\bL^{\frac14(g-1)s(s-1)}[\Jac]^s}{\prod_{i<j}\tilde\zeta\left(\frac{z_i}{z_j}\right)}
    \sum_{\sigma\in\Symm_s}\sigma
    \left[
    \prod_{i<j}\tilde\zeta\left(\frac{z_i}{z_j}\right)\cdot
    \frac1{\prod_{i<s}\left(1-\bL\frac{z_{i+1}}{z_i}\right)}\cdot
    \frac1{1-\bL^{\frac{1-s}2}z_1w}
    \right]
\end{equation}
and
\begin{equation}\label{eq:Y1<Formula}
    Y^{<0}_{\underline1}(\underline z,w)=(-1)^s
    \frac{\bL^{\frac14(g-1)s(s-1)}[\Jac]^s}{\prod_{i<j}\tilde\zeta\left(\frac{z_i}{z_j}\right)}
    \sum_{\sigma\in\Symm_s}\sigma
    \left[
    \prod_{i<j}\tilde\zeta\left(\frac{z_i}{z_j}\right)\cdot
    \frac1{\prod_{i<s}\left(1-\bL^{-1}\frac{z_i}{z_{i+1}}\right)}\cdot
    \frac1{1-\bL^{\frac{s-1}2}z_sw}
    \right],
\end{equation}
where the rational functions are expanded in the regions $z_1\gg\ldots\gg z_s$, $w\ll1$ and $z_1\gg\ldots\gg z_s$, $w\gg1$ respectively.
\end{proposition}
\begin{remark}\label{rm:RatExpand}
The coefficients of the rational functions in the RHS of~\eqref{eq:Y1>Formula} and~\eqref{eq:Y1<Formula} belong to the ring $\cMot(\kk)[\sqrt\bL]$, and it is not known whether this ring is integral. Thus some care should be taken. Note, however, that the RHS of~\eqref{eq:Y1>Formula} can be written in the form
\[
    \frac{P(\underline z,w)}{M(\underline z,w)(1+Q(\underline z,w))},
\]
where $P(\underline z,w)$ is a polynomial, $M(\underline z,w)$ is a monomial in $\underline z$ and $w$, $Q(\underline z,w)$ is a polynomial in $z_{i+1}/z_i$ and $z_iw$ \emph{without constant term} (see Lemma~\ref{lm:zetaX}). Thus we define the expansion as
\[
      \frac{P(\underline z,w)}{M(\underline z,w)(1+Q(\underline z,w))}=
    \frac{P(\underline z,w)}{M(\underline z,w)}\left(
    \sum_{i\ge0}(-Q(\underline z,w))^i
    \right).
\]
Similar considerations apply to~\eqref{eq:Y1<Formula} and to the coefficients of $w$-expansions of the RHS of~\eqref{eq:Y1>Formula} and~\eqref{eq:Y1<Formula}.
\end{remark}

\begin{proof}[Proof of Proposition~\ref{pr:MainCalc}]
The proof is analogous to that of~\cite{SchiffmannIndecomposable}, we indicate the places, where some changes are needed. We use induction on $s$. The case $s=1$ is easy (the proof repeats that of~\cite{SchiffmannIndecomposable}). Next, set $E^{vec}_{\underline1}(\underline z)=E_1^{vec}(z_s)\ldots E_1^{vec}(z_1)$. Using Lemma~\ref{lm:HallProduct}(ii) and the fact that $\cH_{fin}$ is a subalgebra of $\cH$, we get (cf.~also Remark~\ref{rm:HallInfinite})

\begin{equation}\label{eq:Y1}
Y_{\underline1}(\underline z,w)=Y^{\ge0}_{\underline1}(\underline z,w)+Y^{<0}_{\underline1}(\underline z,w)+
\sum_{\substack{u+t=s\\u,t>0}}\bL^{\frac12(g-1)ut}
\left(
E_{\underline1}^{vec}(\underline z)|
E_u^{<0}(\bL^{\frac t2}w) E_t^{\ge0}(\bL^{-\frac u2}w)
\right).
\end{equation}

By Proposition~\ref{pr:HallFormulas}\eqref{Halliii2} we get
\[
    \Delta(E_1^{vec}(z))=
    E_1^{vec}(z)\boxtimes1+
    E_0(\bL^{\frac12}z)E_0(\bL^{-\frac12}z)^{-1}k_1\boxtimes E_1^{vec}(z).
\]

Let $\sigma:\{1,\ldots,s\}\to\{1,2\}$ be a map, set
\[
    X_\sigma=\prod_i^\to C_{\sigma(i)}(z_i),
\]
where $C_1(z)=E_1^{vec}(z)\boxtimes1$, $C_2(z)=E_0(\bL^{\frac12}z)E_0(\bL^{-\frac12}z)^{-1}k_1\boxtimes E_1^{vec}(z)$. We get by Proposition~\ref{pr:DeltaHom}
\[
    \Delta(E_{\underline1}^{vec}(\underline z))=\sum_\sigma X_\sigma.
\]
Thus, denoting by $\Delta_{u,t}$ the component of $\Delta$ of rank $(u,t)$, we get
\[
    \Delta_{u,t}(E_{\underline1}^{vec}(\underline z))=\sum_{\sigma\in\Sh_{u,t}} X_\sigma.
\]
where $\Sh_{u,t}$ denotes the set of $(u,t)$-shuffles, that is, maps $\sigma:\{1,\ldots,u+t\}\to\{1,2\}$ such that 1 has exactly $u$ preimages. Now combining~\eqref{eq:Y1} and Lemma~\ref{lm:CoProdPairing} we get
\begin{equation}\label{eq:Y1prelim}
Y_{\underline1}(\underline z,w)=Y^{\ge0}_{\underline1}(\underline z,w)+Y^{<0}_{\underline1}(\underline z,w)+
\sum_{\substack{u+t=s\\u,t>0}}\sum_{\sigma\in\Sh_{u,t}}\bL^{\frac12(g-1)ut}
\left(
X_\sigma|
E_u^{<0}(\bL^{\frac t2}w)\boxtimes E_t^{\ge0}(\bL^{-\frac u2}w)
\right).
\end{equation}

Fix $\sigma\in\Sh_{u,t}$ and set
\[
    H_\sigma(\underline z)=\prod_{\substack{(i,j),j>i\\ \sigma(i)=1,\sigma(j)=2}}
    \frac{\tilde\zeta\left(\frac{z_j}{z_i}\right)}{\tilde\zeta\left(\frac{z_i}{z_j}\right)}
\]
(We expand $H_\sigma$ in $z_1\gg\ldots\gg z_s$.)
Now, repeating literally the argument from~\cite{SchiffmannIndecomposable} and using Proposition~\ref{pr:HallFormulas}(\ref{Halli}, \ref{Hallii}) and Lemma~\ref{lm:zetaX}(iii), we get
\[
    X_\sigma=H_\sigma(\underline z)\cdot\left(
    \prod_{i,\sigma(i)=1}^\to E_1^{vec}(z_i)
    \prod_{j,\sigma(j)=2}^\to E_0(\bL^{\frac12}z_j)E_0(\bL^{-\frac12}z_j)^{-1}k_1^t
    \right)\boxtimes
    \prod_{j:\sigma(j)=2}^\to E_1^{vec}(z_j).
\]
Plugging this into~\eqref{eq:Y1prelim}, we get as in~\cite{SchiffmannIndecomposable}
\[
Y_{\underline1}(\underline z,w)=Y^{\ge0}_{\underline1}(\underline z,w)+Y^{<0}_{\underline1}(\underline z,w)+
\sum_{\substack{u+t=s\\u,t>0}}\sum_{\sigma\in\Sh_{u,t}}Y_\sigma(\underline z,w),
\]
with
\begin{equation}\label{eq:Ysigma}
    Y_\sigma(\underline z,w)=\bL^{\frac12(g-1)ut}H_\sigma(\underline z)
    Y_{\underline1}^{<0}(z_{i_u},\ldots,z_{i_1},\bL^{\frac t2}w)
    Y_{\underline1}^{\ge0}(z_{j_t},\ldots,z_{j_1},\bL^{-\frac u2}w),
\end{equation}
where $(i_u,\ldots,i_1)$ (resp.~ $(j_t,\ldots,j_1)$) are the reordering in the decreasing order of the set $\sigma^{-1}(1)$ (resp.\ $\sigma^{-1}(2)$).

Now let us write
\[
    Y^{\ge0}_{\underline 1}(\underline z,w)=\sum_{n\ge0}y_n^{\ge0}(\underline z)w^n,\qquad
    Y^{<0}_{\underline 1}(\underline z,w)=\sum_{n<0}y_n^{<0}(\underline z)w^n,\qquad
    Y_\sigma(\underline z,w)=\sum_n y_{\sigma,n}(\underline z)w^n.
\]

As in~\cite[(5.11)]{SchiffmannIndecomposable}, using Lemma~\ref{lm:Y1}, we can re-write~\eqref{eq:Ysigma} as
\begin{equation}\label{eq:yn}
    \bL^{(g-1)\frac{s(s-1)}4}[\Jac]^s
    \sum_{\substack{l_1,\ldots,l_s\in\Z\\ \sum_i l_i=n}}z_1^{l_1}\ldots z_s^{l_s}
    \bL^{\frac12\sum_i l_i(2i-s-1)}=
    y_n^{\ge0}(\underline z)+\sum_{u,\sigma}y_{\sigma,n}(\underline z)
\end{equation}
for $n\ge0$. (And the similar statement is true for $n<0$ if we replace $y_n^{\ge0}$ by $y_n^{<0}$.)

We know from the induction hypothesis that each $y_{\sigma,n}$ is an expansion of a rational function in a certain asymptotic region in the sense of Remark~\ref{rm:RatExpand}. However, we do not know a priori that $y_n^{\ge0}$ and $y_n^{<0}$ are expansions of rational functions. Let us prove this. Note two things.

(*) There is a polynomial $R(\underline z)$ in $z_{i+1}/z_i$ with constant term one such that $R(\underline z)y_n^{\ge0}(\underline z)$ is a Laurent polynomial. This follows from Remark~\ref{rm:RatExpand} and the fact that the LHS of~\eqref{eq:yn} is annihilated by $\prod_i(1-\bL z_{i+1}/z_i)$.

(**) $y_n^{\ge0}(\underline z)$ belongs to
\begin{equation}\label{eq:NonZeroDiv}
\cMot(\kk)[\sqrt\bL]
\left(\!\left(z_1,\frac{z_2}{z_1},\ldots,\frac{z_s}{z_{s-1}}\right)\!\right).
\end{equation}
This is proved similarly to Lemma~\ref{lm:LaurSer}.

\begin{lemma}
Any formal power series satisfying (*) and (**) is an expansion of a rational function in the region $z_1\gg\ldots\gg z_s$.
\end{lemma}
\begin{proof}
    Let $Q(\underline z)$ be a formal power series satisfying (*) and (**). let $R(\underline z)$ be a polynomial in $z_{i+1}/z_i$ with constant term one such that $P(\underline z)=R(\underline z)Q(\underline z)$ is a Laurent polynomial. Subtracting from $Q(\underline z)$ the expansion of $P(\underline z)/R(\underline z)$ in the region $z_1\gg\ldots\gg z_s$, we may assume that $P=0$. However, $R$ is a non-zero divisor in~\eqref{eq:NonZeroDiv}.
\end{proof}

We see that $y_n^{\ge0}(\underline z)$ is an expansion of a rational function in the asymptotic region $z_1\gg\ldots\gg z_s$. Similar considerations show that $y_n^{<0}(\underline z)$ is an expansion of a rational function in the same region. The rest of the proof is completely similar to that of~\cite[Prop.~5.3]{SchiffmannIndecomposable}.
\end{proof}

Recall that in Section~\ref{sect:explicit} for a partition $\lambda=1^{r_1}2^{r_2}\ldots t^{r_t}$ such that $\sum_i r_i=n$, we defined the iterated residue $\res_\lambda$.

\begin{corollary}\label{cor:Y}
We have the following equations of series with coefficients in $\cMot(\kk)[\sqrt\bL]$.
\begin{multline*}
Y^{\ge0}_{\underline r}(\bL^{-\frac12r_t}z_{1+r_{<t}},\ldots,\bL^{-\frac12r_i}z_{1+r_{<i}},\ldots,
\bL^{-\frac12r_1}z_1,w)=\bL^{b(\underline r)}\prod_i vol_{r_i}\\
\cdot\res_\lambda\left[
    \frac1{\prod_{i<j}\tilde\zeta\left(\frac{z_i}{z_j}\right)}
    \sum_{\sigma\in\Symm_n}
    \left\{
    \prod_{i<j}\tilde\zeta\left(\frac{z_i}{z_j}\right)\cdot
    \frac1{\prod_{i<n}\left(1-\bL\frac{z_{i+1}}{z_i}\right)}\cdot
    \frac1{1-\bL^{\frac{-n}2}z_1w}
\right\}
\right]\prod_{\substack{j=1\\j\notin\{r_{<i}\}}}^n\frac{dz_j}{z_j},
\end{multline*}
where
\[
	b(\lambda)=\frac12(g-1)\sum_{i<j}r_ir_j.
\]
\end{corollary}
\begin{proof}
Combine the previous proposition with $s=n$ and Proposition~\ref{pr:HallFormulas}\eqref{Halliv} (cf.~\cite[(5.15)]{SchiffmannIndecomposable}).
\end{proof}

\subsection{Proof of Theorem~\ref{th:NilpEnd}}\label{sect:NilpEnd}
Combining Proposition~\ref{pr:GtoY} with Proposition~\ref{cor:Y}, we get a formula for $G^{\ge0}_{\underline r}$. Let $\cE_{r,d}^{coh,nilp}$ be the moduli stack of coherent sheaves on $X$ of class $(r,d)$ with nilpotent endomorphisms. Let $\cE_{r,d}^{\ge0,coh,nilp}$ denote the constructible subset corresponding to HN-nonnegative sheaves.
Repeating almost literally the arguments from~\cite[Sect.~3, Sect.~5.1, Sect.~5.6--5.7]{SchiffmannIndecomposable}, we get
\begin{equation}\label{eq:CohNilp}
    \sum_{r,d\ge0}[\cE_{r,d}^{\ge0,coh,nilp}]w^rz^d=\sum_{\lambda}
    \bL^{(g-1)\langle\lambda,\lambda\rangle}J_\lambda^{mot}(z) H_\lambda^{mot}(z)w^{|\lambda|}
    \cdot\Exp\left(\frac{[X]}{\bL-1}\cdot\frac{z}{1-z}\right),
\end{equation}
where $\displaystyle\frac z{1-z}$ is expanded in powers of $z$.

\begin{lemma} We have in $\Mot(\kk)[[z,w]]$
\[
    \sum_{r,d\ge0}[\cE_{r,d}^{\ge0,coh,nilp}]w^rz^d=
    \sum_{r,d\ge0}[\cE_{r,d}^{\ge0,nilp}]w^rz^d\cdot
    \sum_{d\ge0}[\cE_{0,d}^{coh,nilp}]z^d.
\]
\end{lemma}
\begin{proof}
Fix $r$ and $d$. Consider the stratification $\cE_{r,d}^{\ge0,coh,nilp}=\sqcup_i\cE_i$, where $\cE_i$ consists of pairs $(F,\Phi)$ such that the torsion part of $F$ has degree $i$. It is enough to show that
\[
    [(\cE_{r,d-i}^{\ge0,nilp}\times\Coh_{0,i}^{nilp})\xrightarrow{\phi}\Coh_{r,d}]=
    [\cE_i\xrightarrow{\psi}\Coh_{r,d}],
\]
where the 1-morphism $\phi$ is defined as $((E,\Psi),(T,\Phi))\mapsto E\oplus T$. Let $\xi:\Spec K\to\Coh_{r,d}$ be a point represented by a coherent sheaf $F$ over $X_K$. Write $F=T\oplus E$, where $T$ is torsion, $E$ is torsion free. According to Proposition~\ref{pr:pointwise zero}, we need to check that the $\xi$-fibers of $\phi$ and $\psi$ have the same motivic classes. We may assume that $T$ has degree $i$ and that $E$ is HN-nonnegative (otherwise both fibers are empty).

The $\xi$-fiber of $\psi$ is the motivic class of $Nil(T\oplus E)$, where the notation stands for the nilpotent cone of the algebra $\End(T\oplus E)$. The $\xi$-fiber of the direct sum morphism $\Coh_{r,d-i}\times\Coh_{0,i}\to\Coh_{r,d}$ is the additive group $\Hom(T,E)$. Thus the $\xi$-fiber of $\phi$ is equal to
\[
    Nil(T)\times Nil(E)\times\Hom(T,E).
\]
Using the fact that every endomorphism of $T\oplus E$ preserves $T$, it is easy to see that the above scheme is isomorphic to $Nil(T\oplus E)$. Proposition~\ref{pr:pointwise zero} completes the proof, (cf.~the proof of~\cite[Thm.~1.4]{MozgovoySchiffmanOnHiggsBundles}).
\end{proof}

\begin{lemma}
\[
        \sum_{d\ge0}[\cE_{0,d}^{coh,nilp}]z^d=
        \Exp\left(\frac{[X]}{\bL-1}\cdot\frac{z}{1-z}\right).
\]
\end{lemma}
\begin{proof}
Put $w=0$ in~\eqref{eq:CohNilp}.
\end{proof}

Combining~\eqref{eq:CohNilp} and the last two lemmas, we get Theorem~\ref{th:NilpEnd}.

\subsection{Proof of Theorem~\ref{th:ExplAnsw}}\label{sect:ExplAnsw}
We claim that
\begin{equation}\label{eq:Hrd=Mge0}
[\cM_{r,d}^{\ge0,ss}]=H_{r,d}.
\end{equation}
(We use notation from Sections~\ref{sect:explicit} and~\ref{Sect:Higgs}.) According to Lemma~\ref{lm:EqSlp}, we just need to show that
\begin{equation*}
    \prod_{\tau\ge0}\left(1+\sum_{d/r=\tau}{\bL^{(1-g)r^2}[\cM_{r,d}^{\ge0,ss}]}w^rz^d\right)=
    \prod_{\tau\ge0}\left(\Exp\left(
        \sum_{d/r=\tau}B_{r,d}w^rz^d.
    \right)\right).
\end{equation*}
By Proposition~\ref{pr:KS} the LHS is equal to
\[
    1+\sum_{r>0,d\ge 0}\bL^{(1-g)r^2}[\cM^{\ge0}_{r,d}]w^rz^d.
\]
The RHS is equal to
\begin{multline*}
    \Exp\left(\sum_{d,r\ge0}B_{r,d}w^rz^d\right)=
    \Exp\left(\bL\Log\left(\sum_\lambda\bL^{(g-1)\langle\lambda,\lambda\rangle}J_\lambda^{mot}(z) H_\lambda^{mot}(z)w^{|\lambda|}\right)\right)=\\
    \Pow\left(\sum_\lambda\bL^{(g-1)\langle\lambda,\lambda\rangle}J_\lambda^{mot}(z) H_\lambda^{mot}(z)w^{|\lambda|},\bL\right).
\end{multline*}
(We used a property of $\Exp$ and the definition of $\Pow$.) According to Theorem~\ref{th:NilpEnd} and Proposition~\ref{pr:NilpEndPow}, the last expression is equal to $\sum_{r,d}[\cE^{\ge0}_{r,d}]w^rz^d$. Now Lemma~\ref{lm:HiggsEnd} completes the proof of~\eqref{eq:Hrd=Mge0}.

Finally, Lemma~\ref{lm:+ss}(i) completes the proof of Theorem~\ref{th:ExplAnsw}.


\begin{thebibliography}{GZLMH}

\bibitem[Ati]{Atiyah-KrullSchmidt}
M.~Atiyah.
\newblock On the {K}rull-{S}chmidt theorem with application to sheaves.
\newblock {\em Bull. Soc. Math. France}, 84:307--317, 1956.

\bibitem[BD]{BehrendDhillon}
K.~Behrend and A.~Dhillon.
\newblock On the motivic class of the stack of bundles.
\newblock {\em Adv. Math.}, 212(2):617--644, 2007.

\bibitem[BM]{BryanMorrison}
J.~Bryan and A.~Morrison.
\newblock Motivic classes of commuting varieties via power structures.
\newblock {\em J. Algebraic Geom.}, 24(1):183--199, 2015.

\bibitem[Bri]{BridgeleandHallIntro}
T.~Bridgeland.
\newblock An introduction to motivic {H}all algebras.
\newblock {\em Adv. Math.}, 229(1):102--138, 2012.

\bibitem[CDDP]{ChuangDiaconescuDonagiPantev}
W.~Chuang, D.-E.~Diaconescu, R.~Donagi, and T.~Pantev.
\newblock Parabolic refined invariants and {M}acdonald polynomials.
\newblock {\em Comm. Math. Phys.}, 335(3):1323--1379, 2015.

\bibitem[DDP]{DiaconescuDonagiPantev}
D.-E. {Diaconescu}, R.~{Donagi}, and T.~{Pantev}.
\newblock {BPS states, torus links and wild character varieties}.
\newblock {\em ArXiv e-prints}, April 2017.


\bibitem[D]{Diaconescu}
D.-E.{Diaconescu}. Local curves, wild character varieties, and degenerations, {\em ArXiv e-prints}, May 2017.

\bibitem[{Eke}]{Ekedahl09}
T.~{Ekedahl}.
\newblock {The Grothendieck group of algebraic stacks}.
\newblock {\em ArXiv e-prints}, March 2009.

\bibitem[Fed]{FedorovIsoStokes}
R.~Fedorov.
\newblock Algebraic and hamiltonian approaches to iso{S}tokes deformations.
\newblock {\em Transform. Groups}, 11(2):137--160, 2006.

\bibitem[GPHS]{GarciaPradaHeinlothSchmitt}
O.~Garc{\'{\i}}a-Prada, J.~Heinloth, and A.~Schmitt.
\newblock On the motives of moduli of chains and {H}iggs bundles.
\newblock {\em J. Eur. Math. Soc. (JEMS)}, 16(12):2617--2668, 2014.

\bibitem[GZLMH]{GuzeinZadeEtAlOnLambdaRingStacks}
S.~Gusein-Zade, I.~Luengo, and A.~Melle-Hern{\'a}ndez.
\newblock On the pre-{$\lambda$}-ring structure on the {G}rothendieck ring of
  stacks and the power structures over it.
\newblock {\em Bull. Lond. Math. Soc.}, 45(3):520--528, 2013.

\bibitem[Har]{HarderAnnals}
G.~Harder.
\newblock Chevalley groups over function fields and automorphic forms.
\newblock {\em Ann. of Math. (2)}, 100:249--306, 1974.

\bibitem[HL]{HuybrechtsLehnModuli}
D.~Huybrechts and M.~Lehn.
\newblock {\em The geometry of moduli spaces of sheaves}.
\newblock Cambridge Mathematical Library. Cambridge University Press,
  Cambridge, second edition, 2010.

\bibitem[HLetRV]{HLRV}
T.~Hausel, E.~Letellier, and F.~Rodriguez-Villegas.
\newblock Arithmetic harmonic analysis on character and quiver varieties.
\newblock {\em Duke Math. J.}, 160(2):323--400, 2011.

\bibitem[HML]{HauselMerebWong}
T.~{Hausel}, M.~{Mereb}, and M.~{Lennox Wong}.
\newblock {Arithmetic and representation theory of wild character varieties}.
\newblock {\em ArXiv e-prints}, April 2016.

\bibitem[HN]{HarderNarasimhan}
G.~Harder and M.~Narasimhan.
\newblock On the cohomology groups of moduli spaces of vector bundles on
  curves.
\newblock {\em Math. Ann.}, 212:215--248, 1974/75.

\bibitem[Joy1]{JoyceConfigurationsII}
D.~Joyce.
\newblock Configurations in abelian categories. {II}. {R}ingel-{H}all algebras.
\newblock {\em Adv. Math.}, 210(2):635--706, 2007.

\bibitem[Joy2]{Joyce07}
D.~Joyce.
\newblock Motivic invariants of {A}rtin stacks and `stack functions'.
\newblock {\em Q. J. Math.}, 58(3):345--392, 2007.

\bibitem[{Kap}]{KapranovMotivic}
M.~{Kapranov}.
\newblock {The elliptic curve in the S-duality theory and Eisenstein series for
  Kac-Moody groups}.
\newblock {\em ArXiv e-prints}, January 2000.

\bibitem[Kle]{KleimanFGA}
S.~Kleiman.
\newblock The {P}icard scheme.
\newblock In {\em Fundamental algebraic geometry}, volume 123 of {\em Math.
  Surveys Monogr., Chapter~9}, pages 235--321. Amer. Math. Soc., Providence,
  RI, 2005.

\bibitem[Kre]{KreschStacks}
A.~Kresch.
\newblock Cycle groups for {A}rtin stacks.
\newblock {\em Invent. Math.}, 138(3):495--536, 1999.

\bibitem[KS1]{KontsevichSoibelman08}
M.~{Kontsevich} and Y.~{Soibelman}.
\newblock {Stability structures, motivic Donaldson-Thomas invariants and
  cluster transformations}.
\newblock {\em ArXiv e-prints}, November 2008.

\bibitem[KS2]{KontsevichSoibelman10}
M.~{Kontsevich} and Y.~{Soibelman}.
\newblock {Cohomological Hall algebra, exponential Hodge structures and motivic
  Donaldson-Thomas invariants}.
\newblock {\em ArXiv e-prints}, June 2010.

\bibitem[KS3]{KontsevichSoibelman13}
M.~{Kontsevich} and Y.~{Soibelman}.
\newblock {Wall-crossing structures in Donaldson-Thomas invariants, integrable
  systems and Mirror Symmetry}.
\newblock {\em ArXiv e-prints}, March 2013.

\bibitem[Lan]{Langton75}
S.~Langton.
\newblock Valuative criteria for families of vector bundles on algebraic
  varieties.
\newblock {\em Ann. of Math. (2)}, 101:88--110, 1975.

\bibitem[{Let}]{LetellierParabolicHiggs}
E.~{Letellier}.
\newblock {Higgs bundles and indecomposable parabolic bundles over the
  projective line}.
\newblock {\em ArXiv e-prints}, September 2016.

\bibitem[{Lin}]{LinSpherical}
J.-A. {Lin}.
\newblock {Spherical Hall algebras of a weighted projective curve}.
\newblock {\em ArXiv e-prints}, October 2014.

\bibitem[LMB]{LaumonMoretBailly}
G.~Laumon and L.~Moret-Bailly.
\newblock {\em Champs alg\'ebriques}, volume~39 of {\em Ergebnisse der
  Mathematik und ihrer Grenzgebiete. 3. Folge. A Series of Modern Surveys in
  Mathematics [Results in Mathematics and Related Areas. 3rd Series. A Series
  of Modern Surveys in Mathematics]}.
\newblock Springer-Verlag, Berlin, 2000.

\bibitem[Mar]{MaruyamaBoundedness}
M.~Maruyama.
\newblock On boundedness of families of torsion free sheaves.
\newblock {\em J. Math. Kyoto Univ.}, 21(4):673--701, 1981.

\bibitem[Moz]{MozgovoyADHM}
S.~Mozgovoy.
\newblock Solutions of the motivic {ADHM} recursion formula.
\newblock {\em Int. Math. Res. Not. IMRN}, (18):4218--4244, 2012.

\bibitem[MS]{MozgovoySchiffmanOnHiggsBundles}
S.~Mozgovoy and O.~Schiffmann.
\newblock {Counting Higgs bundles}.
\newblock {\em ArXiv e-prints, November 2014}.

\bibitem[Mum]{MumfordAbelian}
D.~Mumford.
\newblock {\em Abelian varieties}.
\newblock Tata Institute of Fundamental Research Studies in Mathematics, No. 5.
  Published for the Tata Institute of Fundamental Research, Bombay, 1970.

\bibitem[RS]{RenSoibelman}
J.~{Ren} and Y.~{Soibelman}.
\newblock {Cohomological Hall algebras, semicanonical bases and
  Donaldson-Thomas invariants for $2$-dimensional Calabi-Yau categories (with
  an appendix by Ben Davison)}.
\newblock {\em ArXiv e-prints}, August 2015.

\bibitem[{Sch}1]{SchiffmannLectures}
O.~{Schiffmann}.
\newblock {Lectures on Hall algebras}.
\newblock {\em ArXiv e-prints}, November 2006.

\bibitem[Sch2]{SchiffmannIndecomposable}
O.~Schiffmann.
\newblock {Indecomposable vector bundles and stable Higgs bundles over smooth
  projective curves}.
\newblock {\em Ann. of Math. (2)}, (183):297--362, 2014.

\bibitem[Sha]{ShatzHN}
S.~Shatz.
\newblock The decomposition and specialization of algebraic families of vector
  bundles.
\newblock {\em Compositio Math.}, 35(2):163--187, 1977.

\bibitem[Sim]{Simpson1}
C.~Simpson.
\newblock Moduli of representations of the fundamental group of a smooth
  projective variety. {I}.
\newblock {\em Inst. Hautes \'Etudes Sci. Publ. Math.}, (79):47--129, 1994.

\bibitem[{Sza}]{SzaboIrregularHiggs}
S.~{Szab{\'o}}.
\newblock {The birational geometry of irregular Higgs bundles}.
\newblock {\em ArXiv e-prints}, February 2015.
\end{thebibliography}
\end{document}